%% file: main.tex
\documentclass{article}
\usepackage{amsfonts, amsmath, amssymb, amscd, amsthm, graphicx,enumerate,comment}
\usepackage{hyperref,xypic}
\usepackage[usenames]{color}
\usepackage{tikz}
\usepackage{calc}
\usepackage{braket}
\usepackage{mathrsfs}
\usepackage{geometry}
\usepackage{overpic}
\usepackage{cleveref}
\crefname{enumi}{PV}{parts}
\usepackage{mathtools}
\usepackage{caption}
\usepackage{subcaption}
\usepackage{titlesec}
\usepackage{mathabx}

\usepackage{thmtools}
\usepackage{thm-restate}

\makeatletter
\def\blfootnote{\xdef\@thefnmark{}\@footnotetext}
\makeatother

\newtheorem{thm}{Theorem}[section]
\newtheorem{cor}[thm]{Corollary}
\newtheorem{lem}[thm]{Lemma}
\newtheorem{prop}[thm]{Proposition}
\newtheorem{prob}[thm]{Problem}

\theoremstyle{definition}
\newtheorem{defn}[thm]{Definition}
\theoremstyle{remark}
\newtheorem{rem}[thm]{Remark}
\newtheorem{ex}[thm]{Example}

\newtheorem{claim}[thm]{Claim}

\newtheorem{conv}[thm]{Convention}

\newfont{\eufm}{eufm10}

\renewcommand{\phi}{\varphi}

\newcommand{\R}{\mathbb R}
\newcommand{\N}{\mathbb N}
\newcommand{\Z}{\mathbb Z}
\newcommand{\Q}{\mathbb Q}
\newcommand{\C}{\mathbb C}
\renewcommand{\H}{\mathbb{H}}
\newcommand{\Hl}{\mathcal{H}}

\newcommand{\pval}{\widetilde{v}}
\newcommand{\Div}{\operatorname{Div}}
\newcommand{\GL}{\operatorname{GL}}
\renewcommand{\Re}{\operatorname{Re}}
\renewcommand{\Im}{\operatorname{Im}}

\newcommand{\semi}{\rtimes}

\newcommand{\Ga}{\Gamma}
\newcommand{\acts}{\curvearrowright}
\newcommand{\op}{\operatorname}
\newcommand{\Aut}{\op{Aut}}

\newcommand{\HG}{\mathcal{H}(G)}
\newcommand{\GG}{\mathcal{G}(G)}

\newcommand{\mc }{\mathcal}
\newcommand{\mf}{\mathfrak}

\newcommand{\ol}{\overline}
\newcommand{\conf}{\op{Conf}}
\newcommand{\Pmap}{\op{Pval}}

\begin{document}
\title{Valuations, completions, and hyperbolic actions of metabelian groups}
\author{Carolyn R. Abbott \and Sahana Balasubramanya \and Alexander J. Rasmussen \and Appendix \ref{sec:appendix} co-authored with Sam Payne}
\date{}

\maketitle

\begin{abstract}
  Actions on hyperbolic metric spaces are an important tool for studying groups, and so it is natural, but difficult,  to attempt to classify all such actions of a fixed group.  In this paper, we build strong connections between hyperbolic geometry and commutative algebra  in order  to classify the cobounded hyperbolic actions of numerous metabelian groups up to a coarse equivalence. In particular, we turn this classification problem into the problems of classifying ideals in the completions of certain rings and calculating invariant subspaces of matrices. We use this framework to classify the cobounded hyperbolic actions of many abelian-by-cyclic groups associated to expanding integer matrices.
Each such action  is equivalent to an action on a tree or on a Heintze group (a classically studied class of negatively curved Lie groups). Our investigations incorporate number systems, factorization in formal power series rings, completions, and valuations.
\end{abstract}

\tableofcontents

\section{Introduction}

Actions on hyperbolic metric spaces (which we call hyperbolic actions) are among the most important tools in geometric group theory. Hyperbolic actions may be used to study the coarse geometry of a wide variety of groups and are  crucial tools for studying bounded cohomology (\cite{bf}), quotients of groups (\cite{dgo}), and decision problems for groups (\cite{geometry}). It is thus a natural, but typically very difficult, problem  to attempt to classify all of the hyperbolic actions of a group.

To solve this problem, one is forced to study \emph{cobounded} hyperbolic actions up to coarse (quasi-isometric) equivalence. These are minor and natural restrictions, especially as hyperbolicity is preserved under quasi-isometries. However, even after these restrictions, many groups of geometric interest have uncountably many distinct equivalence classes of hyperbolic actions (see, e.g., \cite[Theorem 2.6]{ABO}), making the classification problem seem intractable. It is thus natural to attempt to solve this classification problem for groups with fewer hyperbolic actions. The authors made the first steps towards this goal in \cite{Qp, AR, ABR} by completely classifying the hyperbolic actions of certain classically studied solvable groups. See also \cite{bcfs} for a recent classification of hyperbolic actions of certain lattices.
The present paper was born from a realization that all of the classification results of \cite{Qp, AR, ABR} fit within a common framework. By identifying this common framework and developing new tools for working with groups that fit within it, we are able to classify the hyperbolic actions of numerous metabelian, and specifically abelian-by-cyclic, groups. The key idea that we introduce in this paper is a tool for converting this classification problem into a commutative algebra problem for certain groups. Our analysis unites radix representations, completions, and valuations on commutative rings and relates them to properties of hyperbolic actions.

The set $\mathcal H(G)$ of equivalence classes of cobounded hyperbolic actions  of a group $G$ is naturally a poset, as shown in \cite{ABO}. The partial order is roughly given by considering one action to be smaller than another when the smaller action may be obtained by an operation of collapsing an equivariant family of subspaces in the larger action. The equivalence classes in $\mathcal H(G)$ are called \emph{hyperbolic structures on $G$}. The results of \cite{Qp, AR, ABR} all give a complete description of the poset $\mathcal H(G)$ for the group $G$ in question. In this paper, for any group $G$ fitting into our common framework, we give a partial description of the poset $\mathcal H(G)$, while for groups $G$ satisfying certain additional properties we  give a complete description of $\mathcal H(G)$. Before stating our main theorems, we describe the common algebraic framework in which the results apply.

Consider a ring $R$ that is generated as a $\Z$--algebra by an element $\gamma$. Assuming that $\gamma$ is neither a unit nor a zero divisor, we may define an ascending HNN extension $G\vcentcolon= G(R,\gamma)$ of $R$ by the endomorphism of (the abelian group) $R$ defined by multiplication by $\gamma$.  Further assuming some natural restrictions on the algebra of $R$ (see axioms (A1)--(A4) in Section \ref{sec:axioms}), one may write elements of $R$ in ``base-$\gamma$'' as infinite sums $a_0+a_1\gamma+a_2\gamma^2+\cdots$ where the ``digits'' $a_i$ can be chosen to lie in a finite set. We introduce one more axiom (A5), which  ensures that these \emph{radix representations} are somewhat well-behaved. Whenever these axioms hold, we derive a description of the poset of hyperbolic actions $\mathcal H(G)$. This description is in terms of the \emph{$(\gamma)$-adic completion $\widehat R$ of $R$}, a standard tool from commutative algebra.  To the best of our knowledge, this introduces completions of rings for the first time  as a tool in geometric group theory. 

\begin{thm}\label{thm:main}
 Under the assumptions on $G$ above, the poset $\mathcal H(G)$ consists of the following structures (see Figure \ref{fig:metabelian}): a single elliptic structure, which is dominated by a single lineal structure, and two subposets $\mc P_-(G)$ and $\mathcal P_+(G)$ that intersect in the single lineal structure.  The subposet $\mc P_+(G)$ is isomorphic to the opposite of the poset of ideals of the $(\gamma)$-adic completion $\widehat R$ considered up to multiplication by $\gamma$. The poset $\mc P_-(G)$ is a lattice. Moreover, each element of $\mathcal P_+(G)$ contains an action on a simplicial tree.
\end{thm}

 Theorem \ref{thm:main} broadly expands our understanding of $\mc H(G)$ for a large class of groups $G$.  In particular, roughly half the work of describing $\mathcal H(G)$, a geometric problem, is reduced to the algebraic problem of classifying the ideals of the completion $\widehat R$ up to equivalence of ideals under multiplication by $\gamma$. (See Section \ref{sec:P_+isotoideals} for the precise definition of the equivalence relation on ideals of $\widehat R$.) 
However, the lattice $\mathcal P_-(G)$ described in Theorem \ref{thm:main} is mysterious in general. We dispel some of the mystery around the lattice $\mathcal P_-(G)$ by describing it completely for a wide class of metabelian groups $G$.

\begin{figure}[h]
\centering

\begin{tabular}{c c c}

\begin{subfigure}[b]{0.32\textwidth}
\centering

\begin{tikzpicture}[scale=0.35]

\node[circle, draw, minimum size=0.8cm] (triv) at (0,-3) {$*$};
\node[circle, draw, minimum size=0.8cm] (lin) at (0,0) {$\R$};
\node[circle, draw, minimum size=0.8cm] (bs) at (6,3) {$T$};

\node (low1) at (2,1) {};
\node (low2) at (3,1) {};
\node (low3) at (4,1) {};

\node (up1) at (2,1.1) {};
\node (up2) at (3,1.1) {};
\node (up3) at (4,1.1) {};

\node (Low1) at (-2,1) {};
\node (Low2) at (-3,1) {};
\node (Low3) at (-4,1) {};

\draw[thick] (triv) -- (lin);
\draw[thick] (lin) -- (low1);
\draw[thick] (lin) -- (low2);
\draw[thick] (lin) -- (low3);

\draw[thick] (lin) -- (Low1);
\draw[thick] (lin) -- (Low2);
\draw[thick] (lin) -- (Low3);

\draw[thick] (up1) -- (bs);
\draw[thick] (up2) -- (bs);
\draw[thick] (up3) -- (bs);

\draw[thick, dotted, red, rotate=296.5] (-0.05, 3.4) ellipse (60pt and 140pt);

\draw[thick, dotted, blue, rotate=63.5] (-0.05, 3.4) ellipse (60pt and 140pt);

\node[red] (ideals) at (6,-1) {$\mathcal{P}_+(G)$};

\node[blue] (others) at (-6,-1) {$\mathcal{P}_-(G)$};

\end{tikzpicture}
\caption{$\mathcal{H}(G)$}
\label{fig:metabelian}
\end{subfigure}

&

\begin{subfigure}[b]{0.32\textwidth}
\centering

\begin{tikzpicture}[scale=0.35]

\node[circle, draw, minimum size=0.8cm] (triv) at (0,-3) {$*$};
\node[circle, draw, minimum size=0.8cm] (lin) at (0,0) {$\R$};
\node[circle, draw, minimum size=0.8cm] (bs) at (6,3) {$T_+$};
\node[circle, draw, minimum size=0.8cm] (Bs) at (-6,3) {$T_-$};

\node (low1) at (2,1) {};
\node (low2) at (3,1) {};
\node (low3) at (4,1) {};

\node (up1) at (2,1.1) {};
\node (up2) at (3,1.1) {};
\node (up3) at (4,1.1) {};

\node (Low1) at (-2,1) {};
\node (Low2) at (-3,1) {};
\node (Low3) at (-4,1) {};

\node (Up1) at (-2,1.1) {};
\node (Up2) at (-3,1.1) {};
\node (Up3) at (-4,1.1) {};

\draw[thick] (triv) -- (lin);
\draw[thick] (lin) -- (low1);
\draw[thick] (lin) -- (low2);
\draw[thick] (lin) -- (low3);

\draw[thick] (lin) -- (Low1);
\draw[thick] (lin) -- (Low2);
\draw[thick] (lin) -- (Low3);

\draw[thick] (up1) -- (bs);
\draw[thick] (up2) -- (bs);
\draw[thick] (up3) -- (bs);

\draw[thick] (Up1) -- (Bs);
\draw[thick] (Up2) -- (Bs);
\draw[thick] (Up3) -- (Bs);

\draw[thick, dotted, red, rotate=296.5] (-0.05, 3.4) ellipse (60pt and 140pt);

\draw[thick, dotted, blue, rotate=63.5] (-0.05, 3.4) ellipse (60pt and 140pt);

\node[red] (ideals) at (6,-1) {$\operatorname{Sub}(\Z/n\Z)$};

\node[blue] (others) at (-6,-1) {$\operatorname{Sub}(\Z/n\Z)$};

\end{tikzpicture}
\caption{$\mathcal{H}((\Z/n\Z)\wr \Z)$}
\label{fig:lamplighter}
\end{subfigure}

&

\begin{subfigure}[b]{0.32\textwidth}
\centering

\begin{tikzpicture}[scale=0.35]

\node[circle, draw, minimum size=0.8cm] (triv) at (0,-3) {$*$};
\node[circle, draw, minimum size=0.8cm] (lin) at (0,0) {$\R$};
\node[circle, draw, minimum size=0.8cm] (hyp) at (-3,3) {$\H^2$};
\node[circle, draw, minimum size=0.8cm] (bs) at (6,3) {$T$};

\node (low1) at (2,1) {};
\node (low2) at (3,1) {};
\node (low3) at (4,1) {};

\node (up1) at (2,1.1) {};
\node (up2) at (3,1.1) {};
\node (up3) at (4,1.1) {};

\draw[thick] (triv) -- (lin) -- (hyp);
\draw[thick] (lin) -- (low1);
\draw[thick] (lin) -- (low2);
\draw[thick] (lin) -- (low3);

\draw[thick] (up1) -- (bs);
\draw[thick] (up2) -- (bs);
\draw[thick] (up3) -- (bs);

\draw[thick, dotted, red, rotate=296.5] (-0.05, 3.4) ellipse (60pt and 140pt);

\node[red] (poset) at (6,-1) {$2^{\{1,\ldots,r\}}$};

\end{tikzpicture}
\caption{$\mathcal{H}(BS(1,n))$}
\label{fig:bs1n}

\end{subfigure}

\end{tabular}

\caption{Posets of hyperbolic actions. Here $2^{\{1,\ldots,k\}}$ denotes the power set with inclusion and $\operatorname{Sub}(\Z/n\Z)$ denotes the poset of subgroups with inclusion.}
\label{fig:posets}
\end{figure}
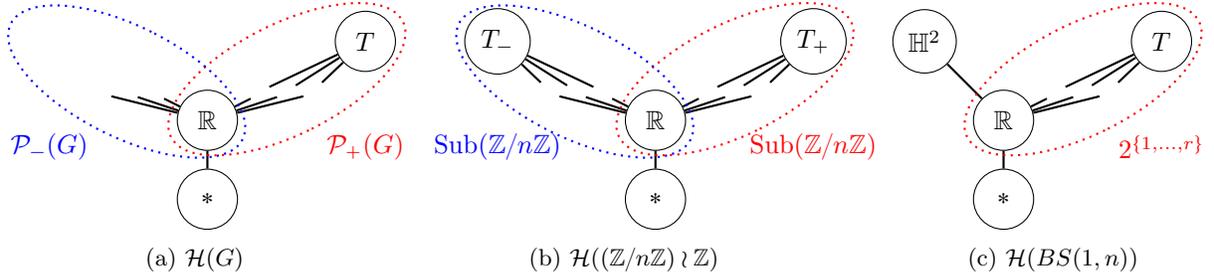

Recall that an \emph{abelian-by-cyclic} group is one of the form $G=A\rtimes \Z$ where $A$ is abelian. We focus on torsion-free, finitely presented abelian-by-cyclic groups. These groups are of interest in dynamics and coarse geometry (see, e.g., \cite{farb_mosher}). Crucially for applying our methods,
such an abelian-by-cyclic group $G$ can be described as an ascending HNN extension of a free abelian group $\Z^n$ \cite{finpres}. That is, there is an endomorphism of $\Z^n$, represented by a matrix $\gamma\in M_n(\Z)$ with $\det \gamma\neq 0$, such that:
\[
G=\left\langle \Z^n,t: tzt^{-1}=\gamma z,\, \forall z\in \Z^n\right\rangle.
\]

In this case, we  completely describe $\mathcal H(G)$ under a few hypotheses on the matrix $\gamma$. We require that all eigenvalues of $\gamma$ lie outside the unit disk in $\C$ (i.e., $\gamma$ is \emph{expanding}) and $\Z^n$ is a cyclic $\Z[x]$-module, where the action of $x$ is induced by that of $\gamma$. We call a matrix $\gamma$ with these  properties \emph{admissible}. The expanding hypothesis on $\gamma$ is the most necessary.  That $\Z^n$ is a cyclic $\Z[x]$-module is always satisfied up to finite index, and  we expect this hypothesis can be removed in future work.

To classify the cobounded hyperbolic actions of $G$, we consider the prime factorization of the characteristic polynomial (which, in this case, is also the minimal polynomial \cite[Section 7.1]{hoffman_kunze}) $p=up_1^{n_1}\cdots p_r^{n_r}$ in the formal power series ring $\Z[[x]]$. Here  $u\in \Z[[x]]$ is a unit and each $p_i\in\Z[[x]]$ is an irreducible power series. Let $[n]$ be the set $\{0,\ldots,n-1\}$ endowed with the standard partial order $\leq$ on $\Z$, and let $\Div(n_1,\ldots,n_r)$ be the poset $[n_1+1]\times \cdots \times [n_r+1]$, which is isomorphic to the poset of divisors of $p=up_1^{n_1}\cdots p_r^{n_r}$ considered up to multiplication by a unit.

\begin{restatable}{thm}{charmin}
\label{thm:char=min}
Let $G$ be an ascending HNN extension of $\Z^n$ by an admissible matrix $\gamma$. Then $\mathcal H(G)$ is as described in Theorem \ref{thm:main} and: 
\begin{enumerate}
\item $\mc P_+(G)$ is isomorphic to $\Div(n_1,\ldots,n_r)$;
\item $\mc P_-(G)$ is isomorphic to the poset of  subspaces  of $\R^n$ that are invariant under $\gamma$; and
\item each element of $\mc P_-(G)$ contains an action on a quasi-convex subspace of a Heintze group.
\end{enumerate}
\end{restatable}

\emph{Heintze groups} are negatively-curved Lie groups of the form $N\rtimes \R$, where $N$ is a nilpotent Lie group. Under the assumptions of Theorem \ref{thm:char=min},  the group  $G$ admits actions on Heintze groups of the form $\C^k\rtimes \R$, where $\R$ acts on $\C^k$ via a 1-parameter subgroup of $\operatorname{GL}_k(\C)$; see Section~\ref{sec:heintze}. However, these actions are never cobounded, and so do not represent elements of $\mc H(G)$. We show that the subspace $\R^k\times \R$ (which is not generally a subgroup) is quasi-isometrically embedded in this group $\C^k\rtimes \R$. Hence the subspace $\R^k\times \R$ is itself Gromov hyperbolic, and admits a cobounded action of $G$. The elements of $\mc P_-(G)$ are represented by such cobounded hyperbolic actions.

Theorem \ref{thm:char=min} reduces the classification of hyperbolic structures of abelian-by-cyclic groups based on admissible matrices $\gamma$ to two algebraic computations: computing the invariant subspaces of $\gamma$ and computing the prime factorization of $p$ in the formal power series ring $\Z[[x]]$. Computing the invariant subspaces is straightforward using Jordan normal forms (see, e.g., \cite{invariant}), while computing factorizations in formal power series rings can be solved algorithmically (see, e.g., \cite{elliott}).

Theorems \ref{thm:main} and  \ref{thm:char=min} completely recover the main theorems of \cite{AR} and \cite{Qp}. Moreover, they can be applied to describe the poset $\mc H(G)$ for numerous new abelian-by-cyclic groups. We give some examples below; see Section \ref{sec:examples} for details. The methods of Section \ref{sec:examples} can be extended to many other groups.

\begin{ex}
Consider a lamplighter group $(\Z/n\Z)\wr \Z$, where $n\geq 2$. Then $\mathcal H((\Z/n\Z)\wr \Z)$ is as pictured in Figure \ref{fig:lamplighter}: there is a unique lineal structure and two copies of the poset $\operatorname{Sub}(\Z/n\Z)$ of subgroups of $\Z/n\Z$ meeting in the unique lineal structure. Every equivalence class contains an action on a tree or on a point. This description of $\mathcal H((\Z/n\Z)\wr \Z)$ was first obtained in \cite{Qp}. See Section \ref{sec:lamplighterposet} for a new, much shorter proof using Theorem \ref{thm:main}.
\end{ex}

\begin{ex}
Consider a solvable Baumslag-Solitar group $BS(1,n)=\langle t, a : tat^{-1}=a^n\rangle$ for $n\in \Z\setminus \{-1,0,1\}$, and let $n=\pm  p_1^{n_1}\cdots p_r^{n_r}$ be the prime factorization of $n$. Then $\mathcal H(BS(1,n))$ is as pictured in Figure \ref{fig:bs1n}: there is a unique lineal structure, a copy of the poset $2^{\{1,\ldots,r\}}$ of subsets of $\{1,\ldots,r\}$ containing the single lineal structure, and a single additional structure containing an action on the hyperbolic plane $\H^2$. Every element of $2^{\{1,\ldots,r\}}$ contains an action on a tree.  This description of $\mathcal H(BS(1,n))$ was first obtained in \cite{AR}. See Section \ref{sec:lamplighterposet} for a new, much shorter proof using Theorem \ref{thm:char=min}.
\end{ex}

\begin{figure}[h]
\centering

\begin{tabular}{c c c}

\begin{subfigure}[t]{0.33\textwidth}

\centering

\begin{tikzpicture}[scale=0.8]

\node[circle, draw, minimum size=0.1cm] (triv) at (0,-3) {};
\node[circle, draw, minimum size=0.1cm] (lin) at (0,0) {};

\node[circle, draw, minimum size=0.1cm] (10) at (1,2) {};
\node[circle, draw, minimum size=0.1cm] (01) at (3,2) {};
\node[circle, draw, minimum size=0.1cm] (11) at (1,3) {};
\node[circle, draw, minimum size=0.1cm] (02) at (3,3) {};
\node[circle, draw, minimum size=0.1cm] (12) at (2,4) {};

\node[circle, draw, minimum size=0.1cm] (ba) at (-3,2) {};
\node[circle, draw, minimum size=0.1cm] (ab) at (-1,2) {};
\node[circle, draw, minimum size=0.1cm] (bb) at (-3,3) {};
\node[circle, draw, minimum size=0.1cm] (ac) at (-1,3) {};
\node[circle, draw, minimum size=0.1cm] (bc) at (-2,4) {};

\node (p+) at (3.5,3.5) [red] {$\mathcal P_+$};

\node (p-) at (-0.5,3.5) [blue] {$\mathcal P_-$};

\draw[thick] (triv) -- (lin);

\draw[thick, red] (lin) -- (10);
\draw[thick, red] (lin) -- (01);
\draw[thick, red] (10) -- (11);
\draw[thick, red] (01) -- (11);
\draw[thick, red] (01) -- (02);
\draw[thick, red] (02) -- (12);
\draw[thick, red] (11) -- (12);

\draw[thick, blue] (lin) -- (ba);
\draw[thick, blue] (lin) -- (ab);
\draw[thick, blue] (ba) -- (bb);
\draw[thick, blue] (ab) -- (bb);
\draw[thick, blue] (ab) -- (ac);
\draw[thick, blue] (ac) -- (bc);
\draw[thick, blue] (bb) -- (bc);

\end{tikzpicture}
\caption{{$\gamma=\begin{pmatrix} 2 & 0 & 0 \\ 0 & 3 & 1 \\ 0 & 0 & 3 \end{pmatrix}$}}
\label{fig:matrix1}
\end{subfigure}

&

\begin{subfigure}[t]{0.33\textwidth}
\centering

\begin{tikzpicture}[scale=0.8]

\node[circle, draw, minimum size=0.1cm] (triv) at (0,-3) {};
\node[circle, draw, minimum size=0.1cm] (lin) at (0,0) {};

\node[circle, draw, minimum size=0.1cm] (low1) at (1,1) {};
\node[circle, draw, minimum size=0.1cm] (low2) at (2,1) {};
\node[circle, draw, minimum size=0.1cm] (low3) at (3,1) {};
\node[circle, draw, minimum size=0.1cm] (low4) at (4,1) {};

\node[circle, draw, minimum size=0.1cm] (med12) at (0.75,2) {};
\node[circle, draw, minimum size=0.1cm] (med13) at (1.5,2) {};
\node[circle, draw, minimum size=0.1cm] (med14) at (2.25,2) {};
\node[circle, draw, minimum size=0.1cm] (med23) at (3,2) {};
\node[circle, draw, minimum size=0.1cm] (med24) at (3.75,2) {};
\node[circle, draw, minimum size=0.1cm] (med34) at (4.5,2) {};

\node[circle, draw, minimum size=0.1cm] (up123) at (1,3) {};
\node[circle, draw, minimum size=0.1cm] (up124) at (2,3) {};
\node[circle, draw, minimum size=0.1cm] (up134) at (3,3) {};
\node[circle, draw, minimum size=0.1cm] (up234) at (4,3) {};

\node[circle, draw, minimum size=0.1cm] (up) at (2.5,4) {};

\node[circle, draw, minimum size=0.1cm] (left1) at (-1,2) {};
\node[circle, draw, minimum size=0.1cm] (left2) at (-2,2) {};

\node[circle, draw, minimum size=0.1cm] (left) at (-1.5,4) {};

\node (p+) at (4.5,3.5) [red] {$\mathcal P_+$};

\node (p-) at (-0.5,3.5) [blue] {$\mathcal P_-$};

\draw[thick] (triv) -- (lin);
\draw[thick, red] (lin) -- (low1);
\draw[thick, red] (lin) -- (low2);
\draw[thick, red] (lin) -- (low3);
\draw[thick, red] (lin) -- (low4);
\draw[thick, red] (low1) -- (med12);
\draw[thick, red] (low1) -- (med13);
\draw[thick, red] (low1) -- (med14);
\draw[thick, red] (low2) -- (med12);
\draw[thick, red] (low2) -- (med23);
\draw[thick, red] (low2) -- (med24);
\draw[thick, red] (low3) -- (med13);
\draw[thick, red] (low3) -- (med23);
\draw[thick, red] (low3) -- (med34);
\draw[thick, red] (low4) -- (med14);
\draw[thick, red] (low4) -- (med24);
\draw[thick, red] (low4) -- (med34);
\draw[thick, red] (med12) -- (up123);
\draw[thick, red] (med12) -- (up124);
\draw[thick, red] (med13) -- (up123);
\draw[thick, red] (med13) -- (up134);
\draw[thick, red] (med14) -- (up124);
\draw[thick, red] (med14) -- (up134);
\draw[thick, red] (med23) -- (up123);
\draw[thick, red] (med23) -- (up234);
\draw[thick, red] (med24) -- (up124);
\draw[thick, red] (med24) -- (up234);
\draw[thick, red] (med34) -- (up134);
\draw[thick, red] (med34) -- (up234);
\draw[thick, red] (up123) -- (up);
\draw[thick, red] (up124) -- (up);
\draw[thick, red] (up134) -- (up);
\draw[thick, red] (up234) -- (up);
\draw[thick, blue] (lin) -- (left1);
\draw[thick, blue] (lin) -- (left2);
\draw[thick, blue] (left1) -- (left);
\draw[thick, blue] (left2) -- (left);

\end{tikzpicture}
\caption{{$\gamma=\begin{pmatrix} 0 & 0 &-210 \\ 1 & 0 & -1 \\ 0 & 1 & 0 \end{pmatrix}$}}
\label{fig:matrix2}
\end{subfigure}

&

\begin{subfigure}[t]{0.33\textwidth}

\begin{tikzpicture}[scale=0.8]

\node[circle, draw, minimum size=0.1cm] (triv) at (0,-3) {};
\node[circle, draw, minimum size=0.1cm] (lin) at (0,0) {};

\node[circle, draw, minimum size=0.1cm] (100) at (-1,2) {};
\node[circle, draw, minimum size=0.1cm] (010) at (-2,2) {};
\node[circle, draw, minimum size=0.1cm] (001) at (-3,2) {};
\node[circle, draw, minimum size=0.1cm] (110) at (-1,3) {};
\node[circle, draw, minimum size=0.1cm] (101) at (-2,3) {};
\node[circle, draw, minimum size=0.1cm] (011) at (-3,3) {};
\node[circle, draw, minimum size=0.1cm] (111) at (-2,4) {};

\node[circle, draw, minimum size=0.1cm] (tree) at (2,2.5) {};

\node (p+) at (2,3.5) [red] {$\mathcal P_+$};

\node (p-) at (-0.5,3.5) [blue] {$\mathcal P_-$};

\draw[thick] (triv) -- (lin);
\draw[thick, red] (lin) -- (tree);

\draw[thick, blue] (lin) -- (100);
\draw[thick, blue] (lin) -- (010);
\draw[thick, blue] (lin) -- (001);
\draw[thick, blue] (100) -- (110);
\draw[thick, blue] (100) -- (101);
\draw[thick, blue] (010) -- (110);
\draw[thick, blue] (010) -- (011);
\draw[thick, blue] (001) -- (101);
\draw[thick, blue] (001) -- (011);
\draw[thick, blue] (110) -- (111);
\draw[thick, blue] (101) -- (111);
\draw[thick, blue] (011) -- (111);

\end{tikzpicture}
\caption{{$\gamma=\begin{pmatrix} 0 & 0 &0 & 2 \\ 1 & 0 & 0 & 0 \\ 0 & 1 & 0 & 0 \\ 0 & 0 & 1 & 0 \end{pmatrix}$}}
\label{fig:matrix3}
\end{subfigure}

\end{tabular}

\caption{Posets $\mathcal H(G(\gamma))$ for various admissible $\gamma$. Every equivalence class to the left in one of the figures contains an action on a Heintze group, while every class to the right contains an action on a tree.}
\label{fig:posetexamples}
\end{figure}
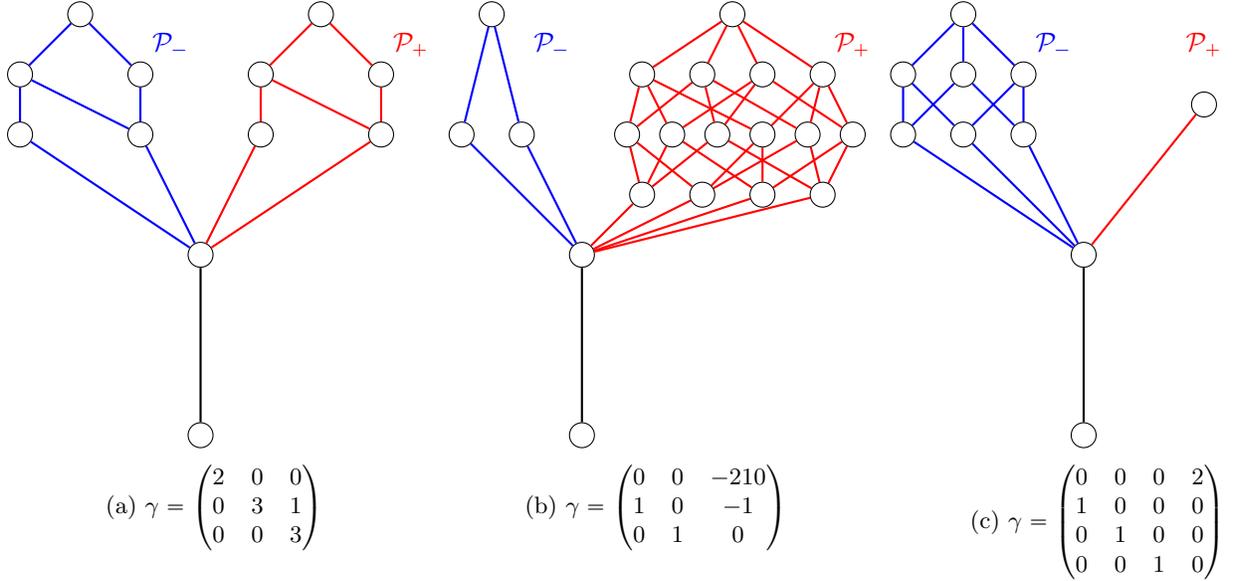

\begin{ex}
Consider a group $G(\gamma)=\langle \Z^n , t : tzt^{-1}=\gamma z \text{ for all } z\in \Z^n\rangle$, where $\gamma\in M_n(\Z)$ is one of the following matrices: \[\gamma = \begin{pmatrix} 2 & 0 & 0 \\ 0 & 3 & 1 \\ 0 & 0 & 3 \end{pmatrix}, \ \ \gamma=\begin{pmatrix} 0 & 0& -210 \\ 1 & 0 & -1\\ 0 & 1 & 0 \end{pmatrix}, \text{ or } \gamma= \begin{pmatrix} 0 & 0 & 0 & 2 \\ 1 & 0 & 0 & 0 \\ 0 & 1 & 0 & 0 \\ 0 & 0 & 1 & 0 \end{pmatrix}.\]  Theorem \ref{thm:char=min} can be used to show that $\mathcal H(G(\gamma))$ is as pictured in Figure \ref{fig:matrix1}, \ref{fig:matrix2}, or \ref{fig:matrix3}, respectively. See Section \ref{sec:otherabcyclicposet} for details.
\end{ex}

 It is not immediately clear how common admissible matrices are.
The following corollaries show that Theorem~\ref{thm:char=min} applies to an abundant family of examples.  Moreover, in the special cases covered by these corollaries, we  give a simple, explicit description of $\mc P_-(G)$.

\begin{restatable}{cor}{polynomialring}
\label{cor:polynomialring}
Let $p(x)\in \Z[x]$ be monic and irreducible with all roots outside the unit disk in $\C$. Set $R=\Z[x]/(p)$, $\gamma=x+(p)$, and $G=G(R,\gamma)$. Then $\mathcal H(G)$ is described by Theorem \ref{thm:char=min} and, moreover, $\mc P_-(G)$ is isomorphic to $\Div(1,\ldots,1)$ where the number of $1$'s between the parentheses is the number of real roots of $p$ plus half the number of complex roots of $p$.
\end{restatable}

If the constant term of $p$ is square-free, then the problem of computing the prime factorization of $p$ in $\Z[[x]]$ is straightforward, leading to:

\begin{cor}
\label{cor:squarefreeconstant}
Suppose, in addition to to the hypotheses of Corollary \ref{cor:polynomialring}, that the constant term of $p$ is square-free. Then $\mathcal H(G)$ is described by Theorem \ref{thm:char=min} and, moreover,
\begin{enumerate}
\item $\mc P_+(G)$ is isomorphic to $\Div(1,\ldots,1)$ where the number of $1$'s between the parentheses is the number of prime factors of the constant term of $p$; and
\item $\mc P_-(G)$ is isomorphic to $\Div(1,\ldots,1)$ where the number of $1$'s between the parentheses is the number of real roots of $p$ plus half the number of complex roots of $p$.
\end{enumerate}
\end{cor}

Part of the description in Theorem \ref{thm:char=min}  holds in greater generality. In particular, we can drop the requirement that $\Z^n$ is a cyclic module.

\begin{thm}
\label{thm:P-description}
Let $G$ be an ascending HNN extension of $\Z^n$ by a matrix $\gamma\in M_n(\Z)$ which is expanding and whose characteristic and minimal polynomials are equal.  The poset $\mathcal H(G)$ consists of the following structures: a single elliptic structure, which is dominated by a single lineal structure, and two subposets $\mc P_-(G)$ and $\mathcal P_+(G)$ that intersect in the single lineal structure. The poset $\mc P_-(G)$ is isomorphic to the poset of invariant subspaces of $\R^n$ under $\gamma$. Each element of $\mc P_-(G)$ contains an action on a quasi-convex subspace of a Heintze group. Moreover, $\mc P_+(G)$ is a lattice.
\end{thm}

We expect that the theory developed in this paper will lay the groundwork for future, more general classifications of hyperbolic actions of metabelian groups. A first direction would be to solve:

\begin{prob} Complete the classification of hyperbolic actions of abelian-by-cyclic groups defined by expanding integer matrices. \end{prob} 

This will necessitate describing $\mathcal P_+(G)$ in the case that $\Z^n$ is not a cyclic $\Z[x]$-module.
We anticipate that this can be done by completing $\Z^n$ \emph{as a module} and classifying the hyperbolic structures of $G$ in terms of closed sub-modules of the completion. A next direction would be:

\begin{prob} Classify hyperbolic actions of finitely generated metabelian groups with abelianizations of rank $>1$. \end{prob}  

The machinery developed in \cite{ABR} may be useful here. For a general metabelian group $G$ with (abelian) commutator subgroup $A=[G,G]$ and abelianization $Q=G/A$, the abelian group $A$ is a $\Z[Q]$-module. It may be possible to classify (certain) hyperbolic structures on $G$ in terms of closed sub-modules of certain completions of $A$.

A direction slightly tangential to the methods of this paper would be the following:

\begin{prob}
Classify hyperbolic actions of polycyclic groups.
\end{prob}

Some cases of this problem were solved in \cite{Largest}, and it is possible that the general case may be approached using invariant subspaces and possibly methods from algebraic number theory.

\subsection{About the proofs}

The proof of Theorem \ref{thm:main} utilizes the description of $G=G(R,\gamma)$ as a semidirect product $G=\gamma^{-1}R\rtimes \Z$, where $\gamma^{-1}R$ is the localization of $R$ with respect to the powers of $\gamma$.
This localization is a larger ring containing $R$ in which division by $\gamma$ is allowed. The proof uses the correspondence, first described in \cite{Amen},  of \emph{confining subsets} of $\gamma^{-1}R$ and certain hyperbolic actions of $G$; such subsets are attracting under multiplication by $\gamma$ and nearly closed under addition. 
Another crucial ingredient in the proof is the \emph{$(\gamma)$-adic completion} $\widehat R$ of $R$. This is the inverse limit of the quotients $R/(\gamma^n)$; see Section \ref{sec:completions} for details.

Theorem \ref{thm:main} is then proven by showing the following correspondences: \[\text{ideals of } \widehat R \leftrightarrow \text{confining subsets of } \gamma^{-1}R \text{ under } \gamma \leftrightarrow \text{valuations on } \gamma^{-1} R \leftrightarrow \text{actions of } G(R,\gamma) \text{ on trees.} \] 
For the first correspondence, we write an element of $\gamma^{-1} R$ with a (possibly infinite) base-$\gamma$ address \[a_{-k} \gamma^{-k} + a_{-k+1} \gamma^{-k+1} + \cdots + a_0 +a_1\gamma +\cdots.\]
The key step is to show that in a confining subset $Q$ of $\gamma^{-1} R$, the fractional parts $a_{-k}\gamma^{-k} + \cdots+a_{-1}\gamma^{-1}$ of elements are governed by an ideal $\frak a$ of $\widehat R$. In particular we show that, up to a small amount of ambiguity, the fractional part 
will agree with the first $k$ digits of an element \[a_{-k}+a_{-k+1}\gamma + \cdots+a_{-1}\gamma^{k-1} + b_k \gamma^k +b_{k+1} \gamma^{k+1} +\cdots \in \frak a,\] giving the correspondence between confining subsets of $\gamma^{-1} R$ and ideals of $\widehat R$.

\vspace{5pt}
The proof of Theorem \ref{thm:char=min} involves two major steps: classifying ideals in the completion of the ring $R=\Z[x]/(p)$ and classifying confining subsets of $\gamma^{-1}R$ under multiplication by $\gamma^{-1}$. For the first step, 
we show that the localization $\gamma^{-1} \widehat R$ is a product of rings with a simple structure: each ring in the product has a single finite chain of ideals, corresponding to a chain of prime power divisors $1, p_i, p_i^2,\ldots, p_i^{n_i}$ of $p$. This uses recent work of McDonough \cite{mcdonough} and Elliott \cite{elliott} on quotients and factorizations in formal power series rings.

For the second step,
we study invariant subspaces of $\gamma$ in $\R^n$. 
We construct a confining subset of $\gamma^{-1}R$ from such an invariant subspace by (essentially) looking at a neighborhood of the subspace. Conversely, to recover an invariant subspace from a confining subset we reverse engineer the construction by taking limits after repeatedly applying $\gamma^{-1}$ to a confining subset.

 For groups $G$ considered in Theorem \ref{thm:char=min}, we use valuations to show that every element of $\mc P_+(G)$ contains an actions on a tree.  There is a natural way to construct a tree associated to a valuation on a set, and this construction can be modified to incorporate a group action (see Section \ref{sec:valuations} and \cite{brown}).  
 For a more general class of metabelian groups, hyperbolic structures in general correspond to a certain form of generalized valuation on an abelian group that we call a \emph{pseudo-valuation}.  This correspondence,  described in Appendix \ref{sec:pseudoval},  could be useful for exploring actions of solvable groups with higher rank abelianizations.

\paragraph{Structure of the paper.} Section \ref{sec:prelims} gives the required preliminary information for this paper. Section \ref{sec:axioms} explains the conditions we consider on the element $\gamma$ and delves into the details of radix representations in base $\gamma$ as well as $(\gamma)-$adic completions. Section \ref{sec:confining} considers the structure of the poset of hyperbolic structures of a general abelian-by-cyclic group.  Section \ref{sec:examples} classifies the hyperbolic actions of many groups using Theorem \ref{thm:main}, some of which were obtained previously in other papers -- this is meant to demonstrate the power of the theorem, as well as provide some intuition to the reader regarding the technical results of the following sections.  Sections \ref{sec:pplusasideals} and \ref{sec:pminusasinvspaces} describe the subposets $\mathcal P_+(G)$ and $\mathcal P_-(G)$ in terms of ideals and invariant subspaces, respectively. Sections \ref{sec:vals} and \ref{sec:heintze} describe the elements of $\mathcal P_+(G)$ and $\mathcal P_-(G)$ via actions on trees and Heintze groups, respectively.   Appendix \ref{sec:appendix} deals with the proof of Theorem \ref{thm:char=min}. Lastly, Appendix \ref{sec:pseudoval} describes the correspondence between confining subsets and pseudo-valuations.

\paragraph{Acknowledgements.} The authors thank Mladen Bestvina, Yair Minsky, Manuel Reyes, and Daniel Erman for helpful conversations. The first author was supported by NSF grants DMS-1803368 and DMS-2106906. The second author was supported by the Deutsche Forschungsgemeinschaft (DFG, German Research Foundation) -Project-ID 427320536 – SFB 1442, as well as under Germany's Excellence Strategy EXC 2044 390685587, Mathematics Münster: Dynamics–Geometry–Structure. The third author was partially supported by NSF grants DMS-1840190 and DMS-2202986.

\section{Background}\label{sec:prelims}

This section describes some necessary background for this paper. We start with a more detailed explanation of our notation for posets and then delve into a discussion of $\Hl(G)$. 

\subsection{Posets}

For $n\in \Z_{\geq 0}$, let $[n]=\{0,1,\ldots,n-1\}$. The set $[n]$ is a poset with the usual partial order $\leq$ on $\Z$. For $n_1,\ldots,n_r\in \Z_{>0}$,  denote by $\Div(n_1,\ldots,n_r)$ the poset $\Div(n_1,\ldots,n_r)\vcentcolon=\prod_{i=1}^r [n_i+1]$ endowed with the product order. Thus, for $(i_1,\ldots,i_r), (j_1,\ldots,j_r)\in \Div(n_1,\ldots,n_r)$ we have $(i_1,\ldots,i_r)\leq (j_1,\ldots,j_r)$ if and only if $i_k\leq j_k$ for each $k$. The poset $\Div(n_1,\ldots,n_r)$ is isomorphic to its opposite poset via the isomorphism sending $(i_1,\ldots,i_r)\in \Div(n_1,\ldots,n_r)$ to $(n_1-i_1,\ldots,n_r-i_r)$. Moreover, $\Div(n_1,\ldots,n_r)$ is a lattice with the meet operation defined by taking the infimum in each entry and the join operation defined by taking the supremum in each entry.

\begin{ex}
The posets $\Div(n_1,\ldots,n_r)$ have many different guises.
\begin{itemize}
\item If $n\in \Z_{>0}$ and $n=p_1^{k_1}\cdots p_r^{k_r}$ is the prime factorization of $n$, then the poset of (positive) divisors of $n$ ordered by $m_1\leq m_2$ if $m_1$ divides $m_2$ is isomorphic to $\Div(k_1,\ldots,k_r)$.
\item If $n\in \Z_{> 0}$ and $\operatorname{Sub}(\Z/n\Z)$\ is the poset of subgroups of $\Z/n\Z$ ordered by inclusion, then $\operatorname{Sub}(\Z/n\Z)$ is isomorphic to the poset of (positive) divisors of $n$, hence isomorphic to $\Div(k_1,\ldots,k_r)$ where $n=p_1^{k_1}\cdots p_r^{k_r}$ is the prime factorization as in the last example.
\item If $r\in \Z_{> 0}$, $X$ is an $r$ element set, and $2^X$ denotes the power set of $X$ ordered by inclusion, then $2^X$ is isomorphic to $\Div(1,\ldots,1)$ (with $r$ entries between the parentheses). The isomorphism can be described as follows: denote by $x_1,\ldots,x_r$ the elements of $X$. The isomorphism $2^X\to \Div(1,\ldots,1)$ is given by sending a subset $Y\in 2^X$ to its \emph{indicator sequence} $(i_1,\ldots,i_r)$ where $i_j=1$ if $x_j\in Y$ and $i_j=0$ if $x_j\notin Y$.
\end{itemize}
\end{ex}

\subsection{Comparing generating sets and group actions} 
 
Throughout this paper, all group actions on metric spaces are assumed to be isometric. Given a metric space $X$, we denote by $d_X$ the distance function on $X$. The $X$ will frequently be dropped from $d_X$ if it is understood by context. Moreover, if $G$ is a group and $S$ is a generating set of $G$, then we denote by $\|\cdot\|_S$ the word norm on $G$ and by $d_S$ the word metric $d_S(g,h)=\|gh^{-1}\|_S$. 

\begin{defn}[{\cite[Definition 1.1]{ABO}}]\label{def-GG}
Let $S$, $T$ be two (possibly infinite) generating sets of a group $G$. We say that $S$ is \emph{dominated} by $T$, written $S\preceq T$, if the identity map on $G$ induces a Lipschitz map of metric spaces $(G, d_T)\to (G, d_S)$. This is equivalent to requiring $ \sup_{t\in T}\|t\|_S<\infty$.  The relation $\preceq$ is a preorder on the set of generating sets of $G$, and therefore it induces an equivalence relation in the standard way:
$$
S\sim T \;\; \Leftrightarrow \;\; S\preceq T \; {\rm and}\; T\preceq S.
$$
This is equivalent to the condition that the Cayley graphs $\Ga(G,S)$ and $\Ga(G, T)$ with respect to $S$ and $T$ are $G$--equivariantly quasi-isometric. We denote by $[S]$ the equivalence class of a generating set $S$, and by $\GG$ the set of all equivalence classes of generating sets of $G$. The preorder $\preceq$ induces a partial order $\preccurlyeq $ on $\GG$ in the standard way:
$$
[S]\preccurlyeq [T] \;\; \Leftrightarrow \;\; S\preceq T.
$$
\end{defn}

For example, all finite generating sets of a finitely generated group are equivalent and the equivalence class containing any finite generating set is the largest element of $\GG$. For every group $G$, the smallest element of $\GG$ is $[G]$. Note also that this order is ``inclusion reversing": if $S$ and $T$ are generating sets of $G$ such that $S\subseteq T$, then $T\preceq S$.

To define a hyperbolic structure on a group, we first recall the definition of a hyperbolic space. In this paper we employ the definition of hyperbolicity via the Rips condition. 

\begin{defn} A metric space $X$ is called \emph{$\delta$--hyperbolic} if it is geodesic and for any geodesic triangle $\Delta $ in $X$, each side of $\Delta $ is contained in the union of the closed $\delta$--neighborhoods of the other two sides.
\end{defn} 

\begin{defn}[{\cite[Definition 1.2]{ABO}}]
A \emph{hyperbolic structure} on $G$ is an equivalence class $[S]\in \GG$ such that the Cayley graph $\Gamma (G,S)$ with respect to $S$. Since hyperbolicity of geodesic metric spaces is quasi-isometry invariant, this definition is independent of the choice of the representative $S$. We denote the set of hyperbolic structures by $\HG$. It is a sub-poset of $\GG$ with the restriction of the partial order on $\GG$.  \end{defn}

The poset $\HG$ classifies the \emph{cobounded} hyperbolic actions of $G$ up to coarsely equivariant quasi-isometry, as we now summarize.
\begin{defn}
The action $G\curvearrowright X$ is \emph{cobounded} if for some (equivalently any) $x\in X$ there exists $R>0$ such that every point of $X$ is distance at most $ R$ from some point of the orbit $Gx$. Given two cobounded hyperbolic actions $G\curvearrowright X$ and $G\curvearrowright Y$, a map $f:X\to Y$ is \emph{coarsely equivariant} if for any $x\in X$ we have \[\sup_{g\in G} d_Y(f(gx),gf(x))<\infty.\] Given $C>0$, the map $f$ is \emph{$C$--coarsely Lipschitz} if \[d_Y(f(x),f(y))\leq Cd_X(x,y)+C\] for all $x,y\in X$. It is a \emph{$C$--quasi-isometry} if it is $C$--coarsely Lipschitz and also satisfies \[\frac{1}{C}d_X(x,y)-C\leq d_Y(f(x),f(y)).\]
\end{defn}

Given  a cobounded hyperbolic action $G\curvearrowright X$, there is an associated hyperbolic structure given by the following Schwarz-Milnor Lemma:

\begin{lem}[{\cite[Lemma 3.11]{ABO}}]\label{lem:MS}
Let $G\curvearrowright X$ be a cobounded hyperbolic action of $G$. Let $B\subseteq X$ be a bounded subset such that $\displaystyle \bigcup_{g\in G} gB=X$. Let $D=\operatorname{diam}(B)$ and let $x\in B$. Then $G$ is generated by the set \[S=\{g\in G: d_X(x,gx)\leq 2D+1\},\] and $X$ is $G$--equivariantly quasi-isometric to $\Gamma(G,S)$.
\end{lem}

 Thus, up to  equivariant quasi-isometries, hyperbolic actions of $G$ correspond to actions of $G$ on its hyperbolic Cayley graphs. There is an equivalence relation on hyperbolic actions given by $G$--coarsely equivariant quasi-isometry and a preorder on hyperbolic actions defined by $(G\curvearrowright X) \preceq (G \curvearrowright Y)$ if there is a coarsely equivariant coarsely Lipschitz map $Y\to X$. In this case we say that the action $G\curvearrowright X$ is \emph{dominated by} the action $G\curvearrowright Y$. With these relations, the set of \emph{equivalence classes of cobounded hyperbolic actions} of $G$ becomes a poset. This poset is isomorphic to $\HG$ \cite[Proposition~3.12]{ABO}.

   We will frequently speak about representatives of hyperbolic structures. If $[S]\in \Hl(G)$, we say that $[S]$ is represented by a cobounded hyperbolic action $G\curvearrowright X$ if $G\curvearrowright X$ is equivalent to the action $G\curvearrowright \Gamma(G,S)$; that is, if there is a coarsely $G$-equivariant quasi-isometry $X\to \Gamma(G,S)$.

\subsection{Actions on hyperbolic spaces}

Let $G$ be a group acting on a hyperbolic space $X$, and denote the Gromov boundary of $X$ by $\partial X$. In general, $X$ is not assumed to be proper, and its boundary is defined as the set of equivalence classes of sequences convergent at infinity. We also denote by $\Lambda (G)$ the set of limit points of $G$ on $\partial X$. That is, $\Lambda (G)\vcentcolon=\partial X\cap \overline{Gx},$ where $\overline{Gx}$ denotes the closure of a $G$--orbit in $X\cup \partial X$, for any choice of basepoint $x\in X$.  This definition is independent of the choice of $x\in X$. The following theorem summarizes the standard classification of group actions on hyperbolic spaces due to Gromov \cite[Section 8.2]{Gromov} and the results  \cite[Propositions 3.1 and 3.2]{Amen}.

\begin{thm}\label{ClassHypAct}
Let $G$ be a group acting on a hyperbolic space $X$. Then exactly one of the following conditions holds.
\begin{enumerate}
\item[1)] $|\Lambda (G)|=0$. Equivalently,  $G$ has bounded orbits. In this case the action of $G$ is called \emph{elliptic}.

\item[2)] $|\Lambda (G)|=1$. In this case the action of $G$ is called \emph{parabolic}. A parabolic action cannot be cobounded. 

\item[3)] $|\Lambda (G)|=2$. Equivalently, $G$ contains a loxodromic element and any two loxodromic elements have the same limit points on $\partial X$. In this case the action of $G$ is called \emph{lineal}. A lineal action $G \acts X$ is said to be \emph{orientable} if $G$ fixes its limit points on $\partial X$ \emph{pointwise}. 

\item[4)] $|\Lambda (G)|=\infty$. Then $G$ always contains loxodromic elements. In turn, this case breaks into two subcases.
\begin{enumerate}
\item[(a)] $G$ fixes a point of $\partial X$. Equivalently, any two loxodromic elements of $G$ have a common limit point on the boundary. In this case the action of $G$ is called \emph{quasi-parabolic} or \emph{focal}. 
\item[(b)] $G$ does not fix any point of $\partial X$. In this case the action of $G$ is said to be of \emph{general type}.
\end{enumerate}
\end{enumerate}
\end{thm}

It follows from the above that for any group $G$, $$\Hl(G)=\Hl_e(G)\sqcup \Hl_{\ell} (G)\sqcup \Hl_{qp} (G)\sqcup \Hl_{gt}(G)$$
where the sets of elliptic, lineal, quasi-parabolic, and general type hyperbolic structures on $G$ are denoted by $\Hl_e(G)$, $\Hl_{\ell} (G)$, $\Hl_{qp} (G)$, and $\Hl_{gt}(G)$ respectively. Namely, $[S]\in \mathcal H(G)$ lies in $\mathcal H_\ell(G)$ if the action $G\curvearrowright \Gamma(G,S)$ is lineal, and similar definitions hold for the posets $\Hl_e,\Hl_{qp},$ and $\Hl_{gt}$. We direct the reader to \cite[Section 4]{ABO} and \cite[Theorem 4.6]{ABO} for further explanation. We also note that $\Hl_{gt}(G) = \emptyset$ for a solvable group $G$,  as $G$ contains no free subgroups (which can be produced from a general type action using the Ping-Pong lemma). 

\paragraph{The Busemann pseudocharacter.} 
\label{sec:busemann}
A function $q\colon G\to \mathbb R$ is a \emph{quasi-character} (or \emph{quasi-morphism}) if there exists a constant $D$ such that $$|q(gh)-q(g)-q(h)|\le D$$ for all $g,h\in G$. We say that $q$ has \emph{defect at most $D$}. If, in addition, the restriction of $q$ to every cyclic subgroup of $G$ is a homomorphism, then $q$ is called a \emph{pseudocharacter} (or \emph{homogeneous quasi-morphism}). 

Given any action of a group $G$ on a hyperbolic space $X$ fixing a point on the boundary, one can associate a natural pseudocharacter $\beta$ called the \emph{Busemann pseudocharacter}. If $\beta$ is a homomorphism, then the action $G\curvearrowright X$ is called \emph{regular}. As we do not require the exact definition of $\beta$ in this paper, we refer the reader to \cite[Sec. 7.5.D]{Gromov} and \cite[Sec. 4.1]{Man} for details. An element $g\in G$ is loxodromic with respect to the action of $G$ on $X$ if and only if $\beta(g)\ne 0$.  In particular, $\beta$ is not identically zero whenever $G \curvearrowright X$ is quasi-parabolic or orientable lineal. 

Conversely, given a pseudocharacter on a group $G$, one can always construct an orientable lineal action:

\begin{lem}[{\cite[Lemma 4.15]{ABO}}] Let $p\colon G \to \R$ be a non-zero pseudocharacter. Let $C$ be any constant such that the defect of $p$ is at most $C/2$ and there exists a value of $p$ in the interval $(0,C/2)$. Let $$ X =X_{p,C} =\{g \in G: |p(g)| < C\}.$$ Then $X$ generates $G$ and the map $p \colon (G,d_X) \to \R$ is a quasi-isometry. In particular, $[X]$ is orientable lineal. \end{lem}

Further, using the Busemann pseudocharacter of a quasi-parabolic action in the above lemma yields the following relation between quasi-parabolic and lineal structures. 

\begin{lem}[{\cite[Corollary 4.26]{ABO}}] For any $[Y] \in \Hl_{qp}(G)$, there exists an (orientable) $[Z] \in \Hl_\ell(G)$ such that $[Z]\preccurlyeq [Y]$. In particular, if $\Hl_{qp}(G) \neq  \emptyset$, then $\Hl_\ell(G) \neq \emptyset$. \end{lem}

\paragraph{Relation to confining subsets.} Consider a group $G=H\rtimes_\alpha \Z$ where $\alpha \in Aut(H)$ and the generator $t\in \Z$ acts on $H$ by conjugation via $tht^{-1}=\alpha(h)$ for any $h\in H$.  Let $Q$ be a symmetric subset of $H$.  The following definition, developed by Caprace-Cornulier-Monod-Tessera in \cite[Section~4]{Amen}, forms the main tool used in the work in \cite{Qp,AR,Largest}. Here $Q\cdot Q$ denotes the set of products $\{g \cdot h \in H : g,h\in Q\}$. 

\begin{defn}\label{def:confining} Let $(H,\cdot)$ be a group, let $Q$ be a symmetric subset of $H$, and let $\alpha$ be an automorphism of $H$. The action of $\alpha$ is \textit{confining $H$ into $Q$} if it satisfies the following conditions$\colon$ 
\begin{itemize}
\item[(a)] $\alpha(Q) \subseteq Q$ 
\item[(b)] $H = \displaystyle \bigcup_{n \geq 0} \hspace{5pt} \alpha^{-n}(Q)$; and
\item[(c)] $\alpha^{k_0}(Q \cdot Q) \subseteq Q$ for some non-negative integer $k_0$. 
\end{itemize}
We also call the set $Q$ \emph{confining under $\alpha$}. If $\alpha(Q) \subsetneq Q$, then $Q$ is called \emph{strictly confining.} 
\end{defn}

The definition of a confining subset given in \cite{Amen} does not require symmetry of the subset $Q\subseteq H$. However, according to \cite[Theorem~4.1]{Amen}, to classify regular quasi-parabolic structures on the group $G=H\rtimes_\alpha \Z$, it suffices to consider only confining subsets which are symmetric; see also \cite[Proposition~2.6]{AR}.

\begin{prop}[{\cite[Proposition 4.6, Theorem 4.1]{Amen}}]\label{niceprop} Let $H$ be a group and  $\alpha$ an automorphism of $H$ which confines $H$
into some subset $Q \subseteq H$. Consider the group $G=H\rtimes_\alpha \Z$, and let $t$ denote a generator of $\Z$. Define $S=Q\cup \{t^{\pm 1}\}\subseteq G$. Then $\Gamma(G,S)$ is Gromov hyperbolic. If $Q$ is strictly confining then the action $G\curvearrowright \Gamma(G,S)$ is regular quasi-parabolic. If $Q$ is not strictly confining then the action $G\curvearrowright \Gamma(G,S)$ is lineal. 
\end{prop}

In general, confining subsets may not completely describe all the quasi-parabolic actions of the group $G$. However, under certain conditions, including the case when $H$ is abelian, there is a one-to-one correspondence between strictly confining subsets (considered up to a natural equivalence) and quasi-parabolic actions of $G$:

\begin{prop}[{\cite[Proposition 4.5]{Amen}},{\cite[Theorem 1.2]{ABR}}]
\label{prop:qctoconfining}
Let $G=H\rtimes_\alpha \Z$ be a semidirect product and let $G\curvearrowright X$ be a cobounded hyperbolic action with a fixed point on $\partial X$. Let $p:G\to \R$ be the Busemann pseudocharacter associated to this action and suppose that $p(H)=0$. Then there exists a subset $Q\subset H$ which is confining under the action of $\alpha$ or $\alpha^{-1}$ such that $X$ is $G$-equivariantly quasi-isometric to $\Gamma(G,Q\cup\{t^{\pm 1}\})$.
\end{prop}

\section{Axioms for commutative rings}\label{sec:axioms} 

In this section we consider groups associated to certain rings. We will introduce a convenient way of parametrizing the elements of such a ring $R$ and its completion $\widehat{R}$.

Consider a ring $R$ together with an element $\gamma\in R$. We assume the following properties, which will naturally lead to our parametrization:
\begin{enumerate}[({A}1)]
\item $\gamma$ is neither a unit nor a zero divisor in $R$;
\item $R$ is generated as a $\Z$-algebra by $\gamma$;
\item the ideal $(\gamma)$ generated by $\gamma$ is finite index in $R$;
\item the intersection of ideals $\bigcap_{n=1}^\infty (\gamma^n)$, is zero.
\end{enumerate}

 By the universal property of the polynomial ring $\Z[x]$, there is a unique homomorphism $\Z[x]\twoheadrightarrow R$ sending $x$ to $\gamma$. Therefore there is a natural isomorphism $R\cong \Z[x]/\mf a$ for some ideal $\mf a$ in $\Z[x]$, so $R$ is a commutative ring. As a quotient of the Noetherian ring $\Z[x]$, it is also necessarily Noetherian.

\begin{rem}
\label{rem:intdoms}
If $R$ happens to be an integral domain, then axiom (A4) follows from the other axioms; see \cite[Corollary 10.18]{atiyah_macdonald}. However, we will not generally assume that $R$ is a domain.
\end{rem}

Consider $R$ with its additive abelian group structure. Left multiplication by $\gamma$ defines an injective endomorphism \emph{of abelian groups} $\alpha\colon R\to R$. We define a group $G(R,\gamma)$ to be the ascending HNN extension of (the abelian group) $R$ by the endomorphism $\alpha$: \[G(R,\gamma)=\langle R, t: trt^{-1}=\gamma r \text{ for } r\in R\rangle.\] We will denote this group simply by $G$ in the case that $R$ and $\gamma$ are clear from context.  There is another description of this group in terms of localizations of rings. The powers $\{1,\gamma,\gamma^2,\ldots\}$ form a multiplicatively closed subset $S$ of $R$. Consider the \emph{localization} $S^{-1}R$ of $R$ by $S$,   denoted simply by $\gamma^{-1}R$ for short. Recall that $\gamma^{-1}R$ consists of the quotients $r/\gamma^k$ for $r\in R$ and $k\geq 0$. There is an equivalence relation on such quotients defined by $r/\gamma^k=s/\gamma^l$ if $r\gamma^l=s\gamma^k$ (a more complicated definition is required if $S$ contains zero divisors, but this is not necessary in our case). The multiplication-by-$\gamma$ endomorphism $\alpha\colon R\to R$ extends to an automorphism (of abelian groups) of $\gamma^{-1}R$ which we will continue to denote by $\alpha$.

The localization $\gamma^{-1}R$ can be used to give another description of $G(R,\gamma)$. We first state the following ubiquitously used and easy to prove lemma:

\begin{lem}[{\cite[Lemma 2.20]{AR}}]
\label{lem:wordrewrite}
Let $G=A\rtimes_\alpha \Z=\langle A, t : tat^{-1}=\alpha(a) \text{ for } a\in A\rangle$. Suppose that $Q\subset A$ is a subset with the property that $Q\cup \{t^{\pm 1}\}$ is a generating set of $G$ and $\alpha(Q)\subset Q$. Then any element $w\in G$ can be written as \[w=t^{-r} x_1\cdots x_m t^\ell\] where $r,\ell\geq 0$, $x_i\in Q$ for all $i$, and $r+\ell+m=\|w\|_{Q\cup \{t^{\pm 1}\}}$. Moreover, if $w\in A$ then $r=\ell$.
\end{lem}

\begin{lem}
\label{lem:semidirectprodstruct}
The ascending HNN extension $G(R,\gamma)$  is isomorphic to the semidirect product $\gamma^{-1}R\rtimes_\alpha \Z$.
\end{lem}

\begin{proof}
Let $s$ be a generator of $\Z$ in the semidirect product, so that $s$ acts on $\gamma^{-1}R$ by conjugation via multiplication by $\gamma$. As above, $t$ denotes the stable letter of $G=G(R,\gamma)$. Sending $t$ to $s$ and $R$ to $\gamma^{-1}R$ by the natural inclusion defines a homomorphism $F\colon G\to \gamma^{-1}R\rtimes_\alpha \Z$. To see that this homomorphism is surjective, it suffices to show that $\gamma^{-n} r\in \gamma^{-1}R$ (with $r\in R$) is in the image of $F$. We may alternatively write this as $s^{-n}rs^n$, which is clearly in the image of $F$ since $s$ and $r$ are. Applying Lemma \ref{lem:wordrewrite} with $A$ the normal subgroup of $G$ generated by $R$ and $Q=R$,  we may write an element of $G$ as $t^{-k}rt^l$ with $r\in R$ and $k,l\geq 0$. To prove injectivity, consider such an element $t^{-k}rt^l$ of $G$. The image under $F$ is $s^{-k}rs^l=(\gamma^{-k}r)s^{-k+l}$. This is equal to the identity if and only if $-k+l=0$ and $\gamma^{-k}r=0$. Since $\gamma$ is not a zero divisor, we have $r=0$. Thus $t^{-k}rt^l=t^{-k}rt^k$ is the identity of $G$.
\end{proof}

\begin{ex}\label{ex:BS1n,Lamplighter}
The following are examples of rings $R$ and elements $\gamma$ along with the resulting groups $G$.
\begin{itemize}
\item Taking $R=\Z$ and $\gamma=n\in \Z\setminus \{-1,0,1\}$ yields $G(R,\gamma)=\langle a,t:tat^{-1}=a^n\rangle=BS(1,n)\cong \Z[1/n]\rtimes \Z$. Note that $R$ and $\gamma$ satisfy the axioms (A1)--(A4) since $R$ contains no zero divisors, $R/(\gamma)\cong \Z/n\Z$, and $R$ is trivially generated as a $\Z$-algebra by $\gamma$. 
\item Taking $R=(\Z/n\Z)[x]$, $\gamma=x$ yields $G(R,\gamma)\cong (\Z/n\Z)[x^{\pm 1}]\rtimes \Z= (\Z/n\Z)\wr \Z$. Again, the reader may check that $R$ and $\gamma$ satisfy (A1)--(A4). 
\end{itemize}
\end{ex}

When discussing confining subsets, it will usually be convenient to think of $G$ using the semidirect product description given by Lemma \ref{lem:semidirectprodstruct}.

\subsection{Number systems and completions}\label{sec:completions}

In this subsection we parametrize elements of $R$. Specifically, we will write elements of $R$ in ``base-$\gamma$.'' Since the ideal $(\gamma)$ is finite index in $R$, we may choose a finite \emph{transversal} for the cosets of $(\gamma)$ in $R$. That is, $T$ contains exactly one element from each coset of $(\gamma)$. In fact, we will choose a standard transversal as follows:

\begin{lem}\label{lem:Tfinitesetofintegers}
There is a number $d\in \Z_{>0}$ such that $T=[d]=\{0,1,\ldots,d-1\}$ forms a transversal for $(\gamma)$.
\end{lem}

\begin{proof}
Recall that $R=\Z[x]/\frak a$. There is a surjective homomorphism of rings $\Z \cong \Z[x]/(x)\twoheadrightarrow R/(\gamma)$. If $d$ is the order of the cyclic group $R/(\gamma)$, then $[d]$ is a transversal.
\end{proof}

The transversal $T=[d]$ defines a base-$\gamma$ expansion for each element of $R$ as follows. Consider an element $r\in R$. There exists a unique element $a_0\in T$ such that $r-a_0\in (\gamma)$, and so  $r-a_0=r_1\gamma$ for some $r_1\in R$. Since $\gamma$ is not a zero divisor, the element $r_1$ is uniquely determined by $r-a_0$ and hence by $r$ itself. There is in turn a unique $a_1\in T$ with $r_1-a_1\in (\gamma)$, and so  $r_1-a_1=r_2\gamma$ for some $r_2\in R$. We have \[r=a_0+r_1\gamma=a_0+(a_1+r_2\gamma)\gamma=a_0+a_1\gamma+r_2\gamma^2.\] Suppose for induction that we have found $a_0, a_1, \ldots,a_{n-1}\in T$ and $r_n\in R$, each uniquely determined by $r$, such that \[r=a_0+a_1\gamma+\cdots+a_{n-1}\gamma^{n-1}+r_n\gamma^n.\] Then $r_n=a_n+r_{n+1}\gamma$ for some $a_n\in T$ and $r_{n+1}\in R$. Since $r_n$ is uniquely determined by $r$, so are $a_n$ and $r_{n+1}$. We have as before, \[r=a_0+\cdots+a_n\gamma^n+r_{n+1}\gamma^{n+1}.\]

We will call the formal power series $a_0+a_1\gamma+a_2\gamma^2+\cdots$ the \emph{$\gamma$-adic address} for $r$ or the base-$\gamma$ expansion of $r$. If $r,s\in R$ have the same $\gamma$-adic address $a_0+a_1\gamma+\cdots$, then they must in fact be equal, since for any $n$, there are $r_n,s_n\in R$ with \[r=a_0+a_1\gamma+\cdots+a_{n-1}\gamma^{n-1}+r_n\gamma^n, \ \ \ s=a_0+a_1\gamma+\cdots+a_{n-1}\gamma^{n-1}+s_n\gamma^n.\] We thus have $r-s=(r_n-s_n)\gamma^n\in (\gamma^n)$. Since $\bigcap_{n=1}^\infty (\gamma^n)=0$, this shows that $r=s$. To summarize:

\begin{lem}
Each $r\in R$ has a unique $\gamma$-adic address. If $r,s\in R$ have the same $\gamma$-adic address, then $r=s$.
\end{lem}

Such addresses lead us naturally to consider \emph{completions}. Completions will also be crucial to discussing hyperbolic structures on $G(R,\gamma)$. The $(\gamma)$-adic completion $\widehat{R}$  is the inverse limit $\varprojlim R/(\gamma^n)$. Here the morphisms $R/(\gamma^n)\to R/(\gamma^m)$ for $n\geq m$ are the natural quotient maps. The elements of $\widehat{R}$ are sequences $(x_n+(\gamma^n))_{n=0}^\infty$ with $x_n\in R$ for each $n$ and $x_n+(\gamma^m)=x_m+(\gamma^m)$ for $n\geq m$. Entry-by-entry addition and multiplication endow $\widehat{R}$ with well-defined ring operations. Furthermore, each element of $\widehat{R}$ has a well-defined $\gamma$-adic address. To see this, choose $x=(x_n+(\gamma^n))\in \widehat{R}$ and fix $m\geq 0$. The element $x_m$ has a $\gamma$-adic address $a_0+a_1\gamma+\cdots$. We have in particular that \[x_m+(\gamma^m)=a_0+a_1\gamma+\cdots+a_{m-1}\gamma^{m-1}+(\gamma^m).\] If $n\geq m$ and $b_0+b_1\gamma+\cdots$ is the $\gamma$-adic address of $x_n$ then
 \[b_0+\cdots+b_{m-1}\gamma^{m-1}+(\gamma^m)=x_n+(\gamma^m)=x_m+(\gamma^m)=a_0+\cdots+a_{m-1}\gamma^{m-1}+(\gamma^m).\] We must have $b_i=a_i$ for each $i\leq m-1$.

Consequently, if we choose $a_0,\ldots,a_n$ to agree with the coefficients of the first $n+1$ terms in the $\gamma$-adic address of $x_{n+1}$, then the sequence $a_0,a_1,\ldots$ is well-defined, independent of $n$, and we define $a_0+a_1\gamma+\cdots$ to be the $\gamma$-adic address of $x$. If $x=(x_n+(\gamma^n)),y=(y_n+(\gamma^n))\in \widehat{R}$ have the same $\gamma$-adic address, then we see by using addresses that $x_n+(\gamma^n)=y_n+(\gamma^n)$ for each $n$, so that $x=y$. Finally, given any sequence $a_0,a_1,\ldots$ of elements of $T$, the sequence $(a_0+\cdots+a_{n-1}\gamma^{n-1} +(\gamma^n))_{n=0}^\infty$ is an element of $\widehat{R}$. This yields:
 
\begin{lem}
Each $x\in \widehat{R}$ has a unique $\gamma$-adic address. If $x,y\in \widehat{R}$ have the same $\gamma$-adic address, then $x=y$. Finally, given any sequence $a_0,a_1,\ldots$ of elements of $T$, there is an element of $\widehat{R}$ with $\gamma$-adic address $a_0+a_1\gamma+a_2\gamma^2+\cdots$.
\end{lem}

The ring $\widehat{R}$ is endowed with the inverse limit topology. That is, as the inverse limit of the rings $R/(\gamma^n)$ (considered as discrete sets) it is a subset of the product $\prod_{n=0}^\infty R/(\gamma^n)$ and inherits the subspace topology. In this topology a sequence $\{a_0^i+a_1^i\gamma+a_2^i\gamma^2+\ldots\}_{i= 1}^\infty \subset \widehat R$ converges to $b_0+b_1\gamma+b_2\gamma^2+\ldots$ if and only if for each $j$, $a_j^i$ is eventually equal to $b_j$ (as $i \to +\infty$).

Finally we discuss the operations on $\widehat{R}$ in terms of $\gamma$-adic addresses. Consider $x=\sum_{i=0}^\infty a_i\gamma^i$ and $y=\sum_{i=0}^\infty b_i\gamma^i$. Then $x+y=\sum_{i=0}^\infty c_i\gamma^i$ where the $c_i$ are determined inductively as follows. First of all, $c_0$ is the unique element of $T$ in the same coset as $a_0+b_0$. We have $a_0+b_0=c_0+r_0\gamma$ for some $r_0\in R$. If $c_0,\ldots,c_n\in T$ and $r_0,\ldots,r_n\in R$ have been determined such that $a_i+b_i+r_{i-1}=c_i+r_i\gamma$ for each $i\leq n$, then $c_{n+1}\in T$ and $r_{n+1}\in R$ are determined uniquely by the equation $a_{n+1}+b_{n+1}+r_n=c_{n+1}+r_{n+1}\gamma$. This formula for $x+y$ may be verified by summing \[\left(\sum_{i=0}^n a_i\gamma^i + \left(\gamma^{n+1}\right)\right)+\left(\sum_{i=0}^n b_i \gamma^i +\left(\gamma^{n+1}\right)\right)\] for each $n\geq 0$ and verifying that this is equal to $\sum_{i=0}^n c_i \gamma^i +(\gamma^{n+1})$. Similarly, $xy=\sum_{i=0}^\infty d_i\gamma^i$ where the $d_i$ are determined inductively as follows. First, $a_0b_0=d_0+s_0\gamma$ with $d_0\in T$ and $s_0\in R$. If $d_0,\ldots,d_n\in T$ and $s_0,\ldots,s_n\in R$ have been determined such that \[\sum_{i=0}^k a_ib_{k-i}+s_{k-1}=d_k+s_k\gamma\] for $k\leq n$, then $d_{n+1}\in T$ and $s_{n+1}\in R$ are determined uniquely by \[\sum_{i=0}^{n+1} a_ib_{n+1-i}+s_n=d_{n+1}+s_{n+1}\gamma.\] This formula for $xy$ may  be verified analogously.

Finally, we introduce one more property of the pair $(R,\gamma)$ which will be crucial for investigating hyperbolic structures on $G(R,\gamma)$. It says that $\gamma$-adic addresses of elements of $R$ are ``well-behaved.'' The importance of this property is for proving Lemma \ref{lem:RsubsetQ}. We say that a $\gamma$-adic address $a_0+a_1\gamma+\cdots$ is \emph{eventually periodic} if there exists $N\geq 0$ and $k>0$ with $a_{n+k}=a_n$ for all $n\geq N$. That is, the $a_j$ eventually repeat in a sequence of period $k$. Recall that the transversal $T=[d]$ defines the $\gamma$-adic address of any element of $R$ or $\widehat R$. We introduce one further axiom for the pair $(R,\gamma)$:

\begin{enumerate}[({A}1)]
\setcounter{enumi}{4}
\item for every $r\in R$, the $\gamma$-adic address of $r$ is eventually periodic.
\end{enumerate}

The eventually periodic property of addresses occurs immediately and naturally in simple examples:

\begin{ex}
Consider the ring $R=\Z$ with $\gamma=2$. The $\gamma$-adic address of $-1$ is $1+2+2^2+2^3+\cdots$. This would normally be written in base 2 as $-1=111\cdots$. Thus the $\gamma$-adic address for $-1$ is eventually periodic.
In fact, writing $-1=a_0+a_12+a_22^2+\cdots$ we have $a_n=a_{n+1}$ so that the eventually periodic property holds with $k=1$ and for all values of $n$.
For a (slightly) more complicated example, the $\gamma$-adic address of $-13$ is \[1+2+2^4+2^5+2^6+2^7+\cdots.\] This is eventually periodic with $a_n=a_{n+1}$ for all $n\geq 4$.
\end{ex}

\begin{ex}
Consider the ring $R=\Z[x]/(x^2-2)$ with $\gamma=x+(x^2-2)$. This is isomorphic to $\Z[\sqrt{2}]$ via the isomorphism sending $\gamma$ to $\sqrt{2}$. A transversal is given by $[2]=\{0,1\}$. One may check that the $\gamma$-adic address of $-1$ in this ring is \[-1=1+\gamma^2+\gamma^4+\gamma^6+ \cdots\] (in fact this essentially follows from the address $-1=1+2+2^2+\cdots$ from the last example). This is eventually periodic with period $k=2$.

In the ring $R=\Z[x]/(x^3-x-2)$ with $\gamma=x+(x^3-x-2)$, we again have that $[2]$ is a transversal. The $\gamma$-adic address of $-1$ is \[-1=1+\gamma+\gamma^3+\gamma^4+\gamma^6+\gamma^7+\gamma^9+\cdots.\] This is eventually periodic with period $k=3$.
\end{ex}

\section{Generalities on confining subsets}
\label{sec:confining}

In this section we consider the hyperbolic structures of a general abelian-by-cyclic group $G=A\rtimes_\alpha \Z$. Under certain hypotheses about abelianizations, we describe the structure of $\mathcal H(G)$ which holds without any of the additional machinery  introduced later in the paper. We will use this structure extensively when proving our main theorems.

\begin{prop}
\label{prop:generalstructure}
Let $G=A\rtimes_\alpha \Z$ with $A$ abelian, and denote by $t$ a generator of $\Z$. Suppose that the abelianization of $G$ is virtually cyclic and that the same holds for its finite index subgroup $A\rtimes \langle t^2\rangle$. Then $\mathcal{H}(G)$ has the following structure: every hyperbolic structure on $G$ is either elliptic, lineal, or quasi-parabolic. The union $\mathcal H_{qp}(G) \cup \mathcal H_{\ell}(G)$ consists of two subposets $\mathcal{P}_-(G)$ and $\mathcal{P}_+(G)$ which are both lattices and intersect in the unique lineal structure of $G$. Every other element of $\mathcal{P}_-(G)\cup \mathcal{P}_+(G)$ is quasi-parabolic. The elements of $\mathcal{P}_-(G) \setminus \mathcal H_\ell (G)$ are incomparable to the elements of $\mathcal{P}_+(G) \setminus \mathcal H_\ell (G)$. Finally, the unique lineal structure of $G$ dominates the unique elliptic structure. See Figure \ref{fig:general}.
\end{prop}

The rather technical condition that the abelianizations of $G=A\rtimes \langle t \rangle$ and $A\rtimes \langle t^2 \rangle$ are both virtually cyclic is meant to rule out non-orientable lineal actions, as will become apparent in the proof. The proposition does not hold, e.g., when $G=\langle a, t : tat^{-1}=a^{-1}\rangle$ (the fundamental group of the Klein bottle), for which the finite index subgroup $\langle a \rangle \rtimes \langle t^2 \rangle$ is $\Z^2$.

\begin{figure}[h]
\centering

\begin{tikzpicture}[scale=0.35]

\node[circle, draw, minimum size=0.8cm] (triv) at (0,-3) {$*$};
\node[circle, draw, minimum size=0.8cm] (lin) at (0,0) {$\R$};

\node (low1) at (2,1) {};
\node (low2) at (3,1) {};
\node (low3) at (4,1) {};

\node (Low1) at (-2,1) {};
\node (Low2) at (-3,1) {};
\node (Low3) at (-4,1) {};

\draw[thick] (triv) -- (lin);
\draw[thick] (lin) -- (low1);
\draw[thick] (lin) -- (low2);
\draw[thick] (lin) -- (low3);

\draw[thick] (lin) -- (Low1);
\draw[thick] (lin) -- (Low2);
\draw[thick] (lin) -- (Low3);

\draw[thick, dotted, red, rotate=296.5] (-0.05, 3.4) ellipse (60pt and 140pt);

\draw[thick, dotted, blue, rotate=63.5] (-0.05, 3.4) ellipse (60pt and 140pt);

\node[red] (ideals) at (6,-1) {$\mathcal{P}_+(G)$};

\node[blue] (others) at (-6,-1) {$\mathcal{P}_-(G)$};

\end{tikzpicture}

\caption{The poset of hyperbolic structures for a general abelian-by-cyclic group $A\rtimes \Z$ under basic hypotheses on abelianizations. The subposets $\mathcal P_-(G)$ and $\mathcal P_+(G)$ are lattices.}
\label{fig:general}
\end{figure}
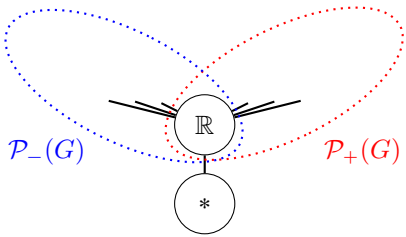

The remainder of this section is devoted to proving Proposition \ref{prop:generalstructure}. We show that we can  identify $\mc P_+(G)$ with the poset of equivalence classes of confining subsets of $A$ under $\alpha$, and, in  Corollary \ref{cor:lattice}, we show that this latter poset is a lattice.  An analogous argument holds for $\mc P_-(G)$ with $\alpha^{-1}$ in place of $\alpha$.  For the rest of the section let $G=A\rtimes_\alpha \Z$, where $A$ is abelian and $\alpha$ is an automorphism of $A$, and let $t$ denote the generator of the $\Z$ factor of $G$ satisfying $tat^{-1}=\alpha(a)$ for all $a\in A$. Finally,  assume that the abelianizations of $G$ and the subgroup $A\rtimes \langle t^2\rangle$ are virtually cyclic.  Recall that subsets which are confining under $\alpha$ are defined in Definition~\ref{def:confining}.

\begin{lem}
\label{lem:confiningpartialorder}
If $P,Q \subset A$ are confining under $\alpha$, then $P\cup \{t^{\pm 1}\} \preceq Q\cup \{t^{\pm 1}\}$ if and only if there exists $N\in \Z$ such that $\alpha^N(Q)\subset P$.
\end{lem}

\begin{proof}
If $\alpha^N(Q)\subset P$, then $t^Nqt^{-N}\in P$ for every $q\in Q$.  Thus $\|q\|_{P\cup \{t^{\pm 1}\}}\leq 2|N|+1$, which shows that $P\cup \{t^{\pm 1}\}\preceq Q\cup \{t^{\pm 1}\}$.

If $P\cup \{t^{\pm 1}\}\preceq Q\cup \{t^{\pm 1}\}$, then there exists $k>0$ such that $\|q\|_{P\cup \{t^{\pm 1}\}}\leq k$ for all $q\in Q$. By Lemma \ref{lem:wordrewrite} we may write any such $q$ as $q=t^{-a}u_1\cdots u_bt^a$ where $u_i\in P$ for all $i$ and $0\leq a,b\leq k$. Consequently $\alpha^a(q)=u_1\cdots u_b \in P^b\subset P^k$. Choose $d$ large enough that $\alpha^d(P^k)\subset P$; such $d$ exists by part (c) of Definition \ref{def:confining}. Then $\alpha^{a+d}(q)\in P$. We have $a+d\leq k+d$, so choosing $N=k+d$ completes the proof.
\end{proof}

Lemma \ref{lem:confiningpartialorder} implies that there is a preorder on confining subsets: $P\preceq Q$ if $\alpha^N(Q)\subset P$ for some $N\in \Z$. This defines an equivalence relation $P\sim Q$ if $P\preceq Q$ and $Q\preceq P$. Thus $P\preceq Q$ if and only if $[P\cup \{t^{\pm 1}\}]\preccurlyeq [Q\cup \{t^{\pm 1}\}]$ and $P\sim Q$ if and only if $[P\cup \{t^{\pm 1}\}]= [Q\cup \{t^{\pm 1}\}]$. We also define ``meet'' and ``join'' operations that descend to honest meet and join operations on the set of equivalence classes of confining subsets:
\begin{itemize} 
\item $P\vee Q\vcentcolon=P\cap Q$; and
\item $P\wedge Q\vcentcolon=P\cdot Q=\{q_1q_2 \colon q_1\in P, q_2\in Q\}$.
\end{itemize}

\begin{lem}
The sets $P\vee Q$ and $P\wedge Q$ are confining under $\alpha$.
\end{lem}

\begin{proof}
First we prove that $P\vee Q$ is confining. Since $\alpha(P)\subset P$ and $\alpha(Q)\subset Q$, we have $\alpha(P\cap Q) \subset \alpha(P) \cap \alpha(Q) \subset P\cap Q$. This verifies Definition \ref{def:confining}(a). If $g\in A$ then there exist $m,n\geq 0$ with $\alpha^m(g)\in P$ and $\alpha^n(g)\in Q$. Hence  $\alpha^{\max(m,n)}(g)\in P\cap Q$. This verifies Definition \ref{def:confining}(b). Finally there exist $k_0\geq 0$ and $l_0\geq 0$ with $\alpha^{k_0}(P\cdot P)\subset P$ and $\alpha^{l_0}(Q\cdot Q)\subset Q$. Hence $\alpha^{\max(k_0,l_0)}\left((P\cap Q)\cdot (P\cap Q)\right)\subset P\cap Q$, which verifies Definition \ref{def:confining}(c).

The fact that $P\wedge Q$ is confining can be similarly verified by using the following observations. First, $\alpha(P\cdot Q)= \alpha(P)\cdot \alpha(Q)\subset P\cdot Q$. Secondly, if $g\in A$ then there exists $m\geq 0$ with $\alpha^m(g)\in P$. Thus, $\alpha^m(g)\in P\cdot Q$. Finally, there exist $k_0\geq 0$ and $l_0\geq 0$ with $\alpha^{k_0}(P\cdot P)\subset P$ and $\alpha^{l_0}(Q\cdot Q)\subset Q$. If $q_1,q_2\in P$ and $u_1,u_2\in Q$, then as $A$ is abelian, \[\alpha^{\max(k_0,l_0)}(q_1u_1\cdot q_2u_2)=\alpha^{\max(k_0,l_0)}(q_1q_2)\alpha^{\max(k_0,l_0)}(u_1u_2)\in P\cdot Q.\] This shows that $\alpha^{\max(k_0,l_0)}\left((P\cdot Q) \cdot (P\cdot Q)\right)\subset P\cdot Q$.
\end{proof}

The next lemma shows that $\vee$ and $\wedge$ indeed descend to meet and join operations on equivalence classes:

\begin{lem}
If $P\sim P'$ and $Q\sim Q'$, then $P\vee Q\sim P'\vee Q'$ and $P\wedge Q\sim P'\wedge Q'$.
\end{lem}

\begin{proof}
It suffices to prove this for the case that $P=P'$ and $Q\sim Q'$.  Lemma \ref{lem:confiningpartialorder} provides $N\geq 0$ large enough so that $\alpha^N(Q)\subset Q'$ and $\alpha^N(Q')\subset Q$.

Then \[\alpha^N(P\vee Q)= \alpha^N(P)\cap \alpha^N(Q)\subset P\cap Q'=P\vee Q'.\]  Similarly $\alpha^N(P\vee Q')\subset P\vee Q$.  Therefore  $P\vee Q\sim P\vee Q'$ by Lemma~\ref{lem:confiningpartialorder}.    The proof that $P\wedge Q\sim P\wedge Q'$ is analogous. 
\end{proof}

Let $\mathcal{Q}$ denote the set of equivalence classes of confining subsets under the relation $\sim$ defined above. The preorder $\preceq$ induces a partial order on $\mathcal{Q}$. We denote by $[P]$ the equivalence class containing a particular confining subset $P$.

\begin{lem}
 The equivalence class $[Q_1\wedge Q_2]$ is a greatest lower bound for $[Q_1]$ and $[Q_2]$, and $[Q_1\vee Q_2]$ is a least upper bound for $[Q_1]$ and $[Q_2]$.
\end{lem}

\begin{proof}
Suppose that $P\preceq Q_1$ and $P\preceq Q_2$. Choose $N\geq 0$ large enough such that $\alpha^N(Q_1)\subset P$, $\alpha^N(Q_2)\subset P$, and $\alpha^N(P\cdot P)\subset P$. Then   $\alpha^{2N}(Q_1\cdot Q_2)=\alpha^N\left(\alpha^N(Q_1)\cdot \alpha^N(Q_2)\right)\subset \alpha^N(P\cdot P)\subset P,$ and so $P\preceq Q_1\wedge Q_2$. 

If $Q_1\preceq P$ and $Q_2\preceq P$, then choose $N\geq 0$ large enough that $\alpha^N(P)\subset Q_1$ and $\alpha^N(P)\subset Q_2$. It follows immediately that $\alpha^N(P)\subset Q_1\cap Q_2$, and so $Q_1\vee Q_2\preceq P$.
\end{proof}

\begin{cor}
\label{cor:lattice}
The poset $\mathcal{Q}$ is a lattice.
\end{cor}

We will call the poset $\mathcal Q$ the \emph{lattice of confining subsets under $\alpha$}. Finally, we prove Proposition \ref{prop:generalstructure}.

\begin{proof}[Proof of Proposition \ref{prop:generalstructure}]
Since $G$ is solvable, every hyperbolic structure is elliptic, lineal, or quasi-parabolic.  The abelianization of $G$ is virtually cyclic, and so the natural homomorphism $G\to \Z$ is the unique non-trivial homomorphism $G\to \R$ up to scaling. Moreover, by amenability, every pseudocharacter (homogeneous quasi-morphism) on $G$ is a homomorphism to $\R$ (\cite{banach}, see also \cite[Theorem 6.16]{ghys}).

We first consider \emph{orientable} lineal actions $G\curvearrowright X$. That is, we consider actions where $X$ is quasi-isometric to a line and $G$ fixes both points of $\partial X$. The Busemann pseudocharacter $\rho$ defined by this action is a non-trivial homomorphism $\rho\colon G\to \R$ that vanishes on $A$. Caprace-Cornulier-Monod-Tessera show in \cite[Proposition~4.5]{Amen} that the action $G\curvearrowright X$ is equivalent to the action $G\curvearrowright \Gamma(G,Q\cup \{t^{\pm 1}\})$, where $Q$ is some confining subset of $A$ under $\alpha$ or $\alpha^{-1}$.  See also the discussion in  \cite[Section 3.2]{ABR}.  We note that both \cite[Proposition~4.5]{Amen} and \cite[Section 3.2]{ABR} consider only quasi-parabolic actions, but the discussion and proofs go through verbatim for lineal actions: both require only that $G\curvearrowright \partial X$ has (at least one) global fixed point.  By the discussion after \cite[Theorem~4.1]{Amen} (see also \cite[Theorem 1.1]{ABR}), the confining subset $Q$ cannot be strictly confining since $G\curvearrowright X$ is not quasi-parabolic.  Thus $Q=A$. We conclude that $G\curvearrowright X\sim G\curvearrowright \Gamma(G,A\cup \{t^{\pm 1}\})$ and hence $G$ has a unique orientable lineal structure. 

Now we consider the case of a general lineal action $G\curvearrowright X$. That is, we do not assume that $G$ fixes $\partial X$ pointwise. We will show that in fact the action must be orientable. To do this, we will first show that $A$ must act elliptically. Suppose for contradiction that $A$ does not act elliptically. Then the restricted action $A\curvearrowright X$ is a cobounded lineal action. Since $A$ is abelian, the action $A\curvearrowright X$ is orientable (see \cite[Example 4.23]{ABO}). Then $A$ and $t^2$ both fix $\partial X$ pointwise and therefore the subgroup that they generate, $G_0=A\rtimes \langle t^2\rangle$, also fixes $\partial X$ pointwise. By hypothesis, the abelianization of $G_0$ is virtually cyclic. Hence, we may apply the reasoning of the previous paragraph to the group $G_0$ in place of $G$, which shows that $A$ acts elliptically, contradicting our assumption. In any case, $A$ acts elliptically. Therefore the action $G\curvearrowright X$ is dominated by $G\curvearrowright \Gamma(G,A\cup \{t^{\pm 1}\})$. But two distinct lineal actions are either equivalent or incomparable (see \cite[Corollary 4.12]{ABO}). Thus, we have $(G\curvearrowright X) \sim (G\curvearrowright \Gamma(G,A\cup \{t^{\pm 1}\}))$. In other words, $[A\cup \{t^{\pm 1}\}]$ is the unique lineal structure on $G$.

Every quasi-parabolic structure on $G$ dominates this single lineal structure (see the discussion in Section \ref{sec:busemann}). By the first paragraph,  the Busemann pseudocharacter $\rho$ of a quasi-parabolic structure is proportional to the natural homomorphism $G\to \Z$, and so the element $t$ acts loxodromically. Moreover, the global fixed point of $G$ in this quasi-parabolic structure is either the attracting fixed point of $t$ (if $\rho(t)<0$) or the repelling fixed point of $t$ (if $\rho(t)>0$). Thus we may divide $\mathcal H_\ell(G) \cup \mathcal H_{qp}(G)$ into two  subposets: structures for which $G$ fixes the attracting fixed point of $t$, which we denote $\mathcal P_-(G)$,  and structures for which $G$ fixes the repelling fixed point of $t$, which we denote $\mathcal P_+(G)$. These subposets meet in the single lineal structure and are otherwise incomparable to each other. To see this last point, note, for instance, that if an element of $\mathcal P_+(G)$ dominates an element of $\mathcal P_-(G)$, then the element of $\mathcal P_-(G)$ must have \emph{both} the attracting and repelling fixed points of $t$ as global fixed points. This element can only be the single lineal structure.

By \cite[Theorems 1.1 \&  1.2]{ABR}, elements of $\mathcal H_{qp}(G)\cup \mathcal H_\ell(G)$ coincide with hyperbolic structures $[Q\cup \{t^{\pm 1}\}]$ where $Q$ is confining under $\alpha$ (if $\rho(t)<0$) or under $\alpha^{-1}$ (if $\rho(t)>0$). We have shown that, equivalently, the elements of $\mathcal P_-(G)$ coincide with the hyperbolic structures $[Q\cup \{t^{\pm 1}\}]$ where $Q$ is confining under $\alpha^{-1}$, and the elements of $\mathcal P_+(G)$ coincide with the hyperbolic structures $[Q\cup\{t^{\pm 1}\}]$ where $Q$ is confining under $\alpha$. By Lemma \ref{lem:confiningpartialorder} the partial order on $\mathcal P_-(G)$ is given by $[P\cup \{t^{\pm 1}\}]\preccurlyeq [Q\cup \{t^{\pm 1}\}]$ if and only if there exists $N\geq 0$ such that $\alpha^{-N}(Q)\subset P$. By Corollary \ref{cor:lattice}, this poset is a lattice (namely, isomorphic to the lattice of confining subsets under $\alpha^{-1}$). By the same reasoning interchanging $\alpha^{-1}$ with $\alpha$, $\mathcal P_+(G)$ is also a lattice (isomorphic to the lattice of confining subsets under $\alpha$).
\end{proof}

We pull out one fact from the proof of Proposition~\ref{prop:generalstructure} which will be needed in later sections.
\begin{prop}\label{prop:Ppmconfining}
Under the assumptions and notation of Proposition~\ref{prop:generalstructure}, the elements of $\mathcal P_-(G)$ coincide with the hyperbolic structures $[Q\cup \{t^{\pm 1}\}]$ where $Q$ is confining under $\alpha^{-1}$, and the elements of $\mathcal P_+(G)$ coincide with the hyperbolic structures $[Q\cup\{t^{\pm 1}\}]$ where $Q$ is confining under $\alpha$.
\end{prop}

\section{Examples}\label{sec:examples}

Before turning to the proofs of Theorems \ref{thm:main} and \ref{thm:char=min}, we first illustrate their utility by classifying the hyperbolic actions of a number of groups.  
In addition to describing the poset $\mc H(G)$ for several new groups, we also show how the previous work of the authors in \cite{AR,Qp} fits into the more general framework of this paper, recovering the main results of \cite{AR} and \cite{Qp} with significantly shorter proofs.

Theorem \ref{thm:main} requires a bit of explanation. Consider the group $G=G(R,\gamma)$ where $R$ and $\gamma$ satisfy axioms (A1)--(A5), and consider the completion $\widehat R$. There is a preorder $\leq$ on ideals of $\widehat R$ defined by $\frak a\leq \frak b$ if for every $x\in \frak a$, there exists $i\in \Z_{\geq 0}$ with $\gamma^i x\in \frak b$. This induces an equivalence relation $\sim$ in the usual way: $\frak a \sim \frak b$ if $\frak a \leq \frak b$ and $\frak b\leq \frak a$, and a partial order $\preccurlyeq$ on the set of equivalence classes. The resulting poset of equivalence classes under $\sim$ with the partial order $\preccurlyeq$ is what we call the \emph{poset of ideals of $\widehat R$ up to multiplication by $\gamma$}.  Theorem~\ref{thm:main} states that  $\mc P_+(G)$ is isomorphic to this poset.

\subsection{Lamplighter groups}
\label{sec:lamplighterposet}

We first use Theorem~\ref{thm:main} to give a streamlined proof of the structure  of the poset of hyperbolic actions of the lamplighter group $(\Z/n\Z)\wr \Z$, for $n\geq 2$.

\begin{thm}[{\cite[Theorem 1.4]{Qp}}]
\label{thm:lamplighter}
If $G=(\Z/n\Z) \wr \Z$ is a lamplighter group for $n \geq 2$, then $\mathcal H(G)$ is as pictured in Figure \ref{fig:lamplighter}.
\end{thm}

The lamplighter group $(\Z/n\Z)\wr \Z$ may be expressed as $G(R,\gamma)$ for $R=(\Z/n\Z)[x]$ and $\gamma=x$, as described in Example~\ref{ex:BS1n,Lamplighter}.
For this $R$ and $\gamma$, the axioms (A1)--(A4) are automatic. Every element of $R$ has a finite base-$\gamma$ address with respect to the transversal $[n]=\{0,1,\ldots,n-1\}$ for $(\gamma)$, which  verifies (A5). We begin by using ideals of $\widehat{R}=(\Z/n\Z)[[x]]$ to classify $\mathcal P_+(G)$ and then later discuss $\mathcal P_-(G)$.

An element $m\in \Z/n\Z$ generates the ideal $(m)=m(\Z/n\Z)[[x]]$, and two such ideals $(m_1)$ and $(m_2)$ are equal exactly when $m_1$ and $m_2$ generate the same (additive) subgroup of $\Z/n\Z$.

\begin{lem}
\label{lem:units}
Let $f\in (\Z/n\Z)[[x]]$ with $f(x)=a_0+a_1x+a_2x^2+\cdots$.  If $a_0$ is a unit in the ring $\Z/n\Z$, then $f$ is a unit in $(\Z/n\Z)[[x]]$.
\end{lem}

\begin{proof}
Suppose that $f(x)$ is as given with $a_0$ a unit. We will  find a power series $g(x)=b_0+b_1x+b_2x^2+\cdots$ such that 
\[a_0b_0+(a_0b_1+a_1b_0)x+(a_0b_2+a_1b_1+a_2b_0)x^2+\cdots =f(x)g(x)=1.\] To this end, choose $b_0\in \Z/n\Z$ with $a_0b_0=1$, which is possible since $a_0$ is a unit. The other coefficients $b_i$ will be defined inductively.  
Suppose $b_0,\ldots,b_{k-1}$ have been chosen so that  $a_0b_0=1$ and the coefficients of $x,x^2,\dots,x^{k-1}$ in the product are zero, and set  $b_k=b_0(-a_1b_{k-1}-\cdots-a_kb_0)$. Then  \[a_0b_k+a_1b_{k-1}+a_2b_{k-2}+\cdots+a_kb_0=0,\] and so the coefficient of $x^k$ is zero, as desired.
\end{proof}

\begin{lem}\label{lem:x^k}
Let $f \in (\Z/n\Z)[[x]]$.  If  the coefficients of  $f$ generate $\Z/n\Z$ as a group, then  $x^k$ lies in the ideal generated by $f$ for some $k$.
\end{lem}

\begin{proof}
The proof is by induction on the total number of factors in the prime factorization of $n$. The base case is when $n$ is prime. Then since every nonzero element of the ring $\Z/n\Z$ is a unit, for some $k$ we have that $f(x)=a_kx^k+a_{k+1}x^{k+1} +\cdots=x^k(a_k+a_{k+1}x+\cdots)$, where $a_k$ is a unit in $\Z/n\Z$. By Lemma \ref{lem:units}, there exists an element $g(x)\in (\Z/n\Z)[[x]]$ with $(a_k+a_{k+1}x+\cdots)g(x)=1$. Thus, $f(x)g(x)=x^k$, and so  $x^k$ lies in the ideal generated by $f$.

Now suppose that the prime decomposition of $n$ is $n=p_1^{k_1}\cdots p_r^{k_r}$. Suppose for induction that the lemma has been proven for all numbers $m$ with at most $k_1+\cdots+k_r-1$ prime factors. Re-ordering if necessary,  suppose that $p_r$ is the smallest number in the set $\{p_1,\ldots,p_r\}$.

Consider the power series $p_1^{k_1}\cdots p_{r-1}^{k_{r-1}}p_r^{k_r-1}f(x)$. Its coefficients all lie in the additive subgroup of $\Z/n\Z$ generated by $p_1^{k_1}\cdots p_r^{k_r-1}$. Furthermore, this power series is nonzero, since some coefficient of $f$ is not divisible by $p_r$ (because the coefficients of $f$ generate $\Z/n\Z$ as an additive group). Note that as a set, the subgroup of $\Z/n\Z$ generated by $p_1^{k_1}\cdots p_r^{k_r-1}$ is given by \[\left\{0, p_1^{k_1}\cdots p_r^{k_r-1}, 2p_1^{k_1}\cdots p_r^{k_r-1}, 3p_1^{k_1}\cdots p_r^{k_r-1}, \ldots, (p_r-1)p_1^{k_1}\cdots p_r^{k_r-1}\right\}.\]

Hence we may write \[p_1^{k_1}\cdots p_r^{k_r-1}f(x)=p_1^{k_1}\cdots p_r^{k_r-1}x^l(a_0+a_1x+\cdots)\] for some $l\geq 0$, where $a_0\neq 0$ and $0\leq a_i\leq p_r-1$ for all $i$. In particular, since $a_0$ is not divisible by any $p_i$, it is a unit in $\Z/n\Z$. Thus by Lemma \ref{lem:units}, there is a power series $g(x)$ with $(a_0+a_1x+\cdots)g(x)=1$ and \[(p_1^{k_1}\cdots p_r^{k_r-1}f(x))g(x)=p_1^{k_1}\cdots p_r^{k_r-1}x^l.\]

Reduce $f(x)$ modulo $p_1^{k_1}\cdots p_r^{k_r-1}$, and denote the result by $\overline{f}(x)\in (\Z/p_1^{k_1}\cdots p_r^{k_r-1}\Z)[[x]]$. Since the coefficients of $f(x)$ generate $\Z/n\Z$ as a group, the coefficients of $\overline{f}(x)$ generate $\Z/p_1^{k_1}\cdots p_r^{k_r-1}\Z$ as a group. Thus, by the induction hypothesis, there exists a power series $\overline{h}(x)\in (\Z/p_1^{k_1}\cdots p_r^{k_r-1}\Z)[[x]]$ with $\overline{f}(x)\overline{h}(x)=x^m$ for some $m\geq 0$.

Choose a power series $h(x) \in (\Z/n\Z)[[x]]$ which reduces to $\overline{h}(x)$ modulo $p_1^{k_1}\cdots p_r^{k_r-1}$. Then \[f(x)h(x)=x^m + p_1^{k_1}\cdots p_r^{k_r-1}u(x)\] for some power series $u(x)$. Since $p_1^{k_1}\cdots p_r^{k_r-1}x^l$ is in the ideal generated by $f$, so is $p_1^{k_1}\cdots p_r^{k_r-1}x^lu(x)$. Thus, both $f(x)h(x)x^l=x^{m+l}+p_1^{k_1}\cdots p_r^{k_r-1}x^lu(x)$ and $ p_1^{k_1}\cdots p_r^{k_r-1}x^lu(x)$ lie in the ideal generated by $f$, and so their difference $x^{m+l}$ lies in $ (f)$, as well. This completes the inductive step.
\end{proof}

The following fact is probably well known and we leave its proof to the reader.

\begin{lem}
\label{lem:cyclicgroupgenerators}
If $a_1,\ldots,a_l\in \Z/n\Z$  generate the subgroup $\langle m\rangle$, then there exist elements $b_1,\ldots,b_l\in\Z/n\Z$  that  generate $\Z/n\Z$ and such that $a_j=b_jm$ for each $j$.
\end{lem}

From this we immediately derive:

\begin{cor}
\label{monomial}
Let $f(x)\in (\Z/n\Z)[[x]]$ and $m\in \Z/n\Z$ be such that the coefficients of $f$ generate the  subgroup $\langle m \rangle \leq \Z/n\Z$. Then $mx^k$ lies in the ideal of $(\Z/n\Z)[[x]]$ generated by $f$, for some $k\in \Z_{\geq 0}$.
\end{cor}

\begin{proof}
Applying Lemma \ref{lem:cyclicgroupgenerators} to the coefficients of $f$, we may write $f(x)=mg(x)$ where the coefficients of $g$ generate $\Z/n\Z$. By Lemma~\ref{lem:x^k}, there exist $h(x)$ and $k\geq 0$ with $g(x)h(x)=x^k$. Thus $mx^k\in (f)$.
\end{proof}

\begin{lem}
\label{lem:lamplighterideals}
If $\mathfrak{a}$ is an ideal of $(\Z/n\Z)[[x]]$, then there exists $m\in \Z/n\Z$ such that $\mathfrak a$ is equivalent to the ideal $m(\Z/n\Z)[[x]]$ generated by $m$.
\end{lem}

\begin{proof}
Let $\mathfrak{a}$ be an arbitrary ideal. Let $\langle m \rangle\leq \Z/n\Z$ be the smallest subgroup containing the coefficients of \emph{every} power series in $\mathfrak a$. Define $\mathfrak{b}=m(\Z/n\Z)[[x]]$. Then $\frak a \subset \frak b$, and we claim that $\mathfrak{a}\sim \mathfrak{b}$. By the definition of the equivalence relation, it suffices to show that $mx^l\in \mathfrak{a}$ for some $l$.

To see this, choose any $f_1\in \mathfrak{a}$.  Its coefficients generate a subgroup $\langle m_1\rangle$ of $\langle m \rangle \leq \Z/n\Z$. By Corollary \ref{monomial}, we have that $m_1x^{k_1}\in \mathfrak{a}$ for some $k_1$. If $\langle m_1 \rangle \neq \langle m \rangle$, then since $\langle m\rangle$ is the \emph{smallest} subgroup containing the coefficients of $\mathfrak a$, there exists some $f_2\in \mathfrak{a}$ whose coefficients generate a subgroup $\langle m_2 \rangle \leq \Z/n\Z$ which is \textit{not} contained in $\langle m_1 \rangle$. We have $m_2x^{k_2} \in \mathfrak{a}$ for some $k_2$. Similarly, if $\langle m_1,m_2\rangle \neq \langle m\rangle$ then there exists some $f_3\in \mathfrak{a}$ whose coefficients generate a subgroup $\langle m_3 \rangle\leq \Z/n\Z$ with $\langle m_3 \rangle \not \leq \langle m_1,m_2\rangle$. We continue this process inductively. Since $\langle m\rangle$ is finite, we eventually find numbers $m_1,\ldots,m_r\in \Z/n\Z$ with $\langle m_1,\ldots,m_r \rangle = \langle m \rangle$ and elements $m_1x^{k_1},\ldots, m_rx^{k_r}\in \mathfrak{a}$. Setting $l=\max\{k_1,\ldots,k_r\}$ we have $m_1x^l,\ldots,m_rx^l \in \mathfrak{a}$. Since $\langle m_1,\ldots,m_r\rangle =\langle m\rangle$, we have $mx^l \in \mathfrak{a}$, as claimed.
\end{proof}

\begin{proof}[Proof of Theorem \ref{thm:lamplighter}]
By Lemma \ref{lem:lamplighterideals}, any ideal $\frak a\subset (\Z/n\Z)[[x]]$ is equivalent to an ideal $m(\Z/n\Z)[[x]]$ generated by some $m\in \Z/n\Z$. The ideal $m(\Z/n\Z)[[x]]$ satisfies the property that if $f\in (\Z/n\Z)[[x]]$ and $x^i f \in m(\Z/n\Z)[[x]]$ for some $i\geq 0$, then $f\in m(\Z/n\Z)[[x]]$. Thus we see that $m(\Z/n\Z)[[x]]\leq  m'(\Z/n\Z)[[x]]$ if and only if $m(\Z/n\Z)[[x]]\subset m'(\Z/n\Z)[[x]]$. Thus $m(\Z/n\Z)[[x]]\sim m'(\Z/n\Z)[[x]]$ if and only if $m(\Z/n\Z)[[x]]= m'(\Z/n\Z)[[x]]$. Moreover, $m(\Z/n\Z)[[x]]\subset m'(\Z/n\Z)[[x]]$ exactly if the subgroup $m(\Z/n\Z)$ is contained in $m'(\Z/n\Z)$. It follows that the poset of ideals of $(\Z/n\Z)[[x]]$ up to equivalence is isomorphic to $\operatorname{Sub}(\Z/n\Z)$, the poset of subgroups of $\Z/n\Z$ with inclusion. Thus $\mathcal P_+(G)$ is isomorphic to the opposite of $\operatorname{Sub}(\Z/n\Z)\cong \Div(k_1,\ldots,k_r)$, where $n=p_1^{k_1}\cdots p_r^{k_r}$ is the prime factorization. Additionally, $\operatorname{Sub}(\Z/n\Z)\cong \Div(k_1,\ldots,k_r)$ is isomorphic to its own opposite. 

To fully describe $\mathcal H(G)$, it remains to describe the lattice $\mathcal P_-(G)$. For this we use confining subsets and the partial order described in Section \ref{sec:confining}. Recall the automorphism $\alpha$ of $(\Z/n\Z)[x^{\pm 1}]$ defined by $\alpha(p(x))=xp(x)$.  The poset of confining subsets under $\alpha$ (that is, $\mathcal P_+(G)$) is the poset induced by the preorder defined by $Q\preceq P$ if $\alpha^N(P)\subset Q$ for some $N\geq 0$. The poset $\mathcal P_-(G)$ is isomorphic to the poset of confining subsets under $\alpha^{-1}$ induced by the preorder $Q\preceq P$ if $\alpha^{-N}(P)\subset Q$ for some $N\geq 0$. There is an isomorphism from $\mathcal P_+(G)$ to $\mathcal P_-(G)$. Namely, there is a unique automorphism of the ring $(\Z/n\Z)[x^{\pm 1}]$  defined by $x\mapsto x^{-1}$. It is straightforward to see that this automorphism sends confining subsets under $\alpha$ to confining subsets under $\alpha^{-1}$ and vice versa. Moreover, this automorphism respects the preorders, and so $\mc P_+(G)$ is isomorphic to $\mc P_-(G)$. This completes the classification of $\Hl(G)$.
\end{proof}

\subsection{Torsion-free finitely presented abelian-by-cyclic groups}

In this section we consider the class of torsion-free finitely presented abelian-by-cyclic groups. Recall that  such groups are ascending HNN extensions of free abelian groups $\Z^n$ by matrices $\gamma\in M_n(\Z)$ with non-zero determinant  \cite{finpres}. That is, a torsion-free finitely presented abelian-by-cyclic group $G$ has the form \[G=\langle \Z^n, t : txt^{-1}=\gamma x \text{ for } x\in \Z^n \rangle.\] We will denote $G$ by $G(\gamma)$ to make the matrix $\gamma$ explicit when necessary. We assume that $\gamma$ is admissible, that is:

\begin{itemize}
\item[(i)] $\gamma$ is expanding (i.e. all of its eigenvalues lie outside the unit disk in the complex plane);
\item[(ii)] $\Z^n$ is a cyclic $\Z[x]$-module (with the structure induced by letting $x$ act on $\Z^n$ by $\gamma$).
\end{itemize}

Consider the prime factorization $p=up_1^{n_1} \cdots p_r^{n_r}$ in $\Z[[x]]$, where $u\in \Z[[x]]$ is a unit and the $p_i$ are distinct irreducible power series in $\Z[[x]]$. We recall Theorem~\ref{thm:char=min} for the convenience of the reader.

\charmin*

In the next three subsections we will apply Theorem \ref{thm:char=min} to classify the hyperbolic actions of solvable Baumslag-Solitar groups and other torsion-free finitely presented abelian-by-cyclic groups. Theorem \ref{thm:char=min} will then be proved in the remaining sections of the paper.

\subsubsection{Solvable Baumslag-Solitar groups}
\label{sec:baumslagposet}

We now use Theorem \ref{thm:char=min} to give a streamlined proof of the structure of the poset of hyperbolic actions of the Baumslag-Solitar group $BS(1,n)$ for $n\in \Z\setminus \{-1,0,1\}$:

\begin{thm}[{\cite[Theorem 1.1]{AR}}]
\label{thm:baumslag}
Let $G=BS(1,n)$ be a Baumslag-Solitar group with $n \notin \{-1,0,1\}$, and let $n=p_1^{k_1} \cdots p_r^{k_r}$ be the prime factorization of $n$.  The poset $\mathcal H(G)$ is as pictured in Figure \ref{fig:bs1n}. 
\end{thm}

Recall that $BS(1,n)\cong \Z[\frac{1}{n}]\rtimes \Z$, so $BS(1,n)=G(\gamma)$ where $\gamma$ is the $1\times 1$ matrix $\begin{pmatrix} n\end{pmatrix}$ acting on $\R=\R^1$. The matrix $\gamma$ is expanding with characteristic polynomial $-n+x$, and $\Z$ is a cyclic $\Z[x]$-module. Hence we may apply Theorem \ref{thm:char=min}. There are only two  subspaces of $\R$ invariant under $\gamma$: 0 and $\R$ itself. It then remains to consider the prime factorization of the monic polynomial $-n+x$ in the formal power series ring $\Z[[x]]$. Equivalently, we will factor the polynomial $n-x$.

We start by stating some basic results on prime factorization in (the unique factorization domain) $\Z[[x]]$. These results are standard and straightforward to prove, so we omit their proofs. Note that since $\Z[[x]]$ is a unique factorization domain, an element is prime if and only if it is irreducible.

\begin{lem}
\label{lem:powerseriesunits}
Let $f(x)=a_0+a_1x+a_2x^2+\cdots \in \Z[[x]]$. Then $f$ is a unit if and only if $a_0=\pm 1$.
\end{lem}

The following allows us to factor any formal power series whose constant term  is not a prime power.

\begin{lem}
Let $f(x)=a_0+a_1x+a_2x^2+\cdots \in \Z[[x]]$. If $a_0=b_0c_0$ with $b_0,c_0\in \Z$ relatively prime, then $f$ may be factored as $f=gh$ where $g(x)=b_0+b_1x+\cdots$ and $h(x)=c_0+c_1x+\cdots$.
\end{lem}

An immediate corollary is the following:

\begin{cor}
\label{cor:seriesfactorization}
Let $f(x)=a_0+a_1x+a_2x^2+\cdots \in \Z[[x]]$. If $a_0=(-1)^\delta p_1^{k_1}\cdots p_r^{k_r}$ is the prime factorization of $a_0$ then $f$ may be factored as $f=(-1)^\delta f_1\cdots f_r$ where the constant term of $f_i$ is   $p_i^{k_i}$.
\end{cor}

However, there is no guarantee that the factors $f_i$ given by Corollary \ref{cor:seriesfactorization} are themselves prime. The following gives a necessary condition for a power series to be prime. However this is not sufficient: there are reducible power series with constant term a power of a prime (e.g., $(2+x)^2=4+4x+x^2$).

\begin{cor}
\label{cor:primepowerconstant}
Let $f(x)=a_0+a_1x+a_2x^2+\cdots \in \Z[[x]]$. If $f$ is prime, then $a_0=\pm p^n$ where $p$ is prime and $n\geq 1$.
\end{cor}

\begin{lem}
\label{lem:relprimeseries}
Let $f(x)=a_0+a_1x+\cdots$ and $g(x)=b_0+b_1x+\cdots$ be two formal power series in $\Z[[x]]$. Then $f$ and $g$ are relatively prime if and only if $a_0$ and $b_0$ are relatively prime.
\end{lem}

The following lemmas give a partial converse to Corollary \ref{cor:primepowerconstant}.

\begin{lem}
\label{lem:seriesprimality}
Consider a formal power series $f(x)=a_0+a_1x+\cdots \in \Z[[x]]$. If $a_0$ is (up to sign) a power of a prime $p\in \Z$ and $a_1$ is relatively prime to $p$, then $f$ is prime. 
\end{lem}

\begin{proof}
Multiply by $-1$ if necessary to assume $a_0>0$, and write $a_0=p^n$ for some $n\geq 1$. Suppose that $f=gh$ where $g=b_0+b_1x+\cdots$ and $h=c_0+c_1x+\cdots$. Up to multiplying both $g$ and $h$ by $-1$, we have $b_0=p^k$ and $c_0=p^l$ where $k,l\geq 0$. Moreover, $a_1=b_0c_1+b_1c_0$.  Since $a_1$ is not divisible by $p$, it must be the case  that $b_0=1$ or $c_0=1$, and so either $g$ or $h$ is a unit.
\end{proof}

\begin{lem}
\label{lem:primeconstantterm}
Consider a formal power series $f(x)=a_0+a_1x+\cdots \in \Z[[x]]$. If $a_0$ is (up to sign) a prime $p\in \Z$, then $f$ is prime.
\end{lem}

\begin{proof}
Multiply by $-1$ if necessary to assume that $a_0=p$. If $f=gh$ where $g=b_0+b_1x+\cdots$ and $h=c_0+c_1x+\cdots$, then $b_0c_0=p$. Up to swapping $g$ and $h$, $p$ divides $b_0$ while $c_0=\pm 1$. Thus $h$ is a unit by Lemma \ref{lem:powerseriesunits}.
\end{proof}

We are now ready to apply these results to the polynomial $n-x$.

\begin{lem}
\label{lem:baumslagfactorization}
Let $n\in \Z\setminus \{-1,0,1\}$, and consider the polynomial $f(x)=n-x$. Let $(-1)^\delta p_1^{k_1}\cdots p_r^{k_r}$ be the prime factorization of $n$.  The prime factorization of $f$ is given by $f(x)=(-1)^\delta f_1\cdots f_r$, where
 the polynomials $f_i$ are prime;
 the polynomials $f_i$ are pairwise relatively prime; and
 the constant term of $f_i$ is $p_i^{k_i}$.
\end{lem}

\begin{proof}
By Corollary $\ref{cor:seriesfactorization}$ we may factor $f$ as $(-1)^\delta f_1 \cdots f_r$, where the constant term of $f_i$ is $p_i^{k_i}$ for each $i$. We claim that each of these factors is prime. Write $f_i=p_i^{k_i} + a^i_1x +a^i_2x^2+\cdots$. Then the coefficient of $x$ in $f_1\cdots f_r$ is \[\sum_{j=1}^r (p_1^{k_1} \cdots \widehat{p_j^{k_j}} \cdots p_r^{k_r}) a^j_1,\] where $\widehat{\cdot}$ denotes omission of the corresponding factor. If some $a_1^i$ is divisible by $p_i$, then each term of this sum is divisible by $p_i$. However, we have \[\sum_{j=1}^r (p_1^{k_1} \cdots \widehat{p_j^{k_j}} \cdots p_r^{k_r}) a^j_1 = \pm 1.\] Thus $a_1^i$ is relatively prime to $p_i$ for each $i$. By Lemma \ref{lem:seriesprimality}, this proves that $f_i$ is prime. The power series $f_i$ are relatively prime to each other by Lemma \ref{lem:relprimeseries}.
\end{proof}

Finally, we can deduce the structure of $\mathcal H(BS(1,n))$ from Theorem \ref{thm:char=min}. Recall that $G=G(\gamma)$ where $\gamma=\begin{pmatrix} n \end{pmatrix}$. The matrix $\gamma$ is admissible. Hence, we have $\mathcal P_+(G)\cong \Div(1,\ldots,1)\cong 2^{\{1,\ldots,r\}}$ where there are $r$ 1's between the parentheses and where the number of distinct prime factors of $n$ is $r$. There are two invariant subspaces of $\R$: 0 and $\R$ itself. Thus $\mathcal P_-(G)$ is isomorphic to the lattice $\Div(1)$ with two elements, one being greater than the other. This matches the description given in Figure \ref{fig:bs1n}, proving Theorem~\ref{thm:baumslag}.

\subsubsection{Other abelian-by-cyclic groups}
\label{sec:otherabcyclicposet}

In this subsection we classify the hyperbolic actions of certain more exotic abelian-by-cyclic groups $G(\gamma)$. The methods used here generalize easily to classify the hyperbolic actions of other abelian-by-cyclic groups. As in the case of $BS(1,n)$, one must compute the prime factorization of the characteristic polynomial in $\Z[[x]]$ and the invariant subspaces of the matrix $\gamma$.

\begin{ex}
Consider the matrix $\gamma=\begin{pmatrix} 2 & 0 & 0 \\ 0 & 3 & 1 \\ 0 & 0 & 3 \end{pmatrix}$. Its characteristic polynomial is $p(x)=(2-x)(3-x)^2$. The reader may check that $\Z^3$ is a cyclic $\Z[x]$-module generated by $(1, 0, 1)$.
Thus Theorem \ref{thm:char=min} applies to the group $G=G(\gamma)=\langle \Z^n , t : txt^{-1}=\gamma x \text{ for } x\in \Z^n\rangle$.

The prime factorization of $p$ in $\Z[[x]]$ is immediate: it is $p(x)=(2-x)(3-x)^2$ since the factors $2-x$ and $3-x$ are prime in $\Z[[x]]$ by Lemma \ref{lem:primeconstantterm}. Thus by Theorem \ref{thm:char=min} we have $\mathcal P_+(G)\cong \Div(1,2)$.

The invariant subspaces of $\gamma$ are the direct sums of the generalized eigenspaces $\ker(\gamma-\lambda I)^i$ for $i\geq 0$ and $\lambda$ an eigenvalue of $\gamma$. The possible direct summands are \[\ker(\gamma-2I)=\operatorname{Span}\left\{\begin{pmatrix} 1 \\ 0 \\ 0\end{pmatrix}\right\}, \quad \ker(\gamma-3I)=\operatorname{Span}\left\{\begin{pmatrix} 0 \\ 1 \\ 0\end{pmatrix}\right\}, \quad \textrm{and } \ker(\gamma-3I)^2=\operatorname{Span}\left\{\begin{pmatrix} 0 \\ 1 \\ 0\end{pmatrix},\begin{pmatrix} 0 \\ 0 \\ 1\end{pmatrix}\right\},\] with the second kernel being contained in the third. Thus the poset of invariant subspaces is $\Div(1,2)$ and this is $\mathcal P_-(G)$; see Figure \ref{fig:matrix1}.
\end{ex}

\begin{ex}
Consider the matrix $\gamma=\begin{pmatrix} 0 & 0 & -210 \\ 1 & 0 & -1 \\ 0 & 1 & 0 \end{pmatrix}$. This is the companion matrix to the irreducible polynomial $p(x)=x^3+x+210$ (see Section \ref{section:preliminaries}), and so 
 $p$ is the characteristic and minimal polynomial of $\gamma$. It has two complex conjugate roots and one real root all lying outside the unit disk in $\C$, so $\gamma$ is expanding. That $\Z^3$ is a cyclic $\Z[x]$-module  follows automatically from the fact that $\gamma$ is the companion matrix to a polynomial. A cyclic vector for $\gamma$ is $(1,0,0)$. Thus Theorem \ref{thm:char=min} applies.

First we calculate the prime factorization of $p$ in $\Z[[x]]$. The constant term $210$ of $p$ factors as $2\cdot 3\cdot 5\cdot 7$. By Corollary \ref{cor:seriesfactorization}, we may factor $p$ in $\Z[[x]]$ as $(2+\cdots)(3+\cdots)(5+\cdots)(7+\cdots)$. Each of these factors is prime by Lemma \ref{lem:primeconstantterm}, so $\mathcal P_+(G)\cong \Div(1,1,1,1)$ by Theorem \ref{thm:char=min}.

Now we calculate the invariant subspaces. Either complex eigenvector gives rise to its real and imaginary parts which span a plane in $\R^3$. This plane is one invariant subspace for $\gamma$. The real eigenvector of $\gamma$ spans a line in $\R^3$ which is another invariant subspace for $\gamma$. The other invariant subspaces are 0 and $\R^3$, which are direct sums of some sub-collection of the invariant plane and the invariant line; see Section \ref{sec:invarsubspaces} for more details on this. Thus, the poset of invariant subspaces is $\Div(1,1)$ and this is $\mathcal P_-(G)$; see Figure \ref{fig:matrix2}.
\end{ex}

\begin{ex}
Consider the matrix $\gamma=\begin{pmatrix}  0 & 0 & 0 & 2 \\ 1 & 0 & 0 & 0  \\ 0 & 1 & 0 & 0 \\ 0 & 0 & 1 & 0 \end{pmatrix}$. This is the companion matrix to the irreducible polynomial $p(x)=x^4-2$, which is thus the characteristic and minimal polynomial of $\gamma$. It has two complex conjugate roots $\pm i \sqrt[4]{2}$ and two real roots $\pm \sqrt[4]{2}$, all outside the unit disk. Again, $\Z^4$ is a cyclic $\Z[x]$-module with generator $v=(1,0,0,0)$, and so Theorem \ref{thm:char=min} applies.

The polynomial $p$ is prime in $\Z[[x]]$ by Lemma \ref{lem:primeconstantterm}, so $\mathcal P_+(G)\cong \Div(1)$. Now we calculate the invariant subspaces. Again, either complex eigenvector gives rise to an $\gamma$-invariant plane in $\R^4$ by taking its real and imaginary parts. The real eigenvectors of $\gamma$ each give rise to a $\gamma$-invariant eigenline. Every invariant subspace is a direct sum of some collection of these three subspaces (that is, the plane and two lines); again, see Section \ref{sec:invarsubspaces} for more details. Thus the poset of invariant subspaces $\mathcal P_-(G)$ is $\Div(1,1,1)$; see Figure \ref{fig:matrix3}.
\end{ex}

\section{Structure of $\mc P_+(G)$: ideals of $\widehat R$}\label{sec:pplusasideals}

In this section we study a group $G=G(R,\gamma)$ where $R$ and $\gamma\in R$ satisfy axioms (A1)--(A5). Recall that $G(R,\gamma)$ is isomorphic to $\gamma^{-1}R\rtimes_\alpha \Z$ and we denote a generator of $\Z$ by $t$. In order to describe $\mathcal H(G)$, we wish to apply Proposition \ref{prop:generalstructure} to the group $G=G(R,\gamma)$. To do this, we need to understand  the abelianizations of both $G$ and $\gamma^{-1}R \rtimes \langle t^2\rangle$.

 \begin{prop}
 \label{prop:Grgammaabelianization}
Let $(R,\gamma)$ satisfy axioms (A1)--(A5), and set $G=G(R,\gamma)$. Then the abelianizations of $G$ and $\gamma^{-1}R \rtimes \langle t^2\rangle$ are virtually cyclic.
 \end{prop}

 \begin{proof}
 We will prove the statement for $G=\gamma^{-1}R \rtimes \langle t \rangle$ first, developing some tools that will apply to $\gamma^{-1}R \rtimes \langle t^2 \rangle$ along the way. We can see from the presentation \[G=\langle R, t: trt^{-1} = \gamma r \text{ for } r \in R\rangle\] that the abelianization of $G$ is $H\times \Z$ where $H$ is the image of $R$. In particular, $H$ is the quotient of $R$ (as an abelian group) by the normal subgroup generated by the elements $r-\gamma r$ for $r\in R$. This is simply the quotient (of rings) $R/(1-\gamma)$. Similarly, $\gamma^{-1} R \rtimes \langle t^2 \rangle$ is the HNN extension \[\langle R, t: trt^{-1}=\gamma^2 r \text{ for } r \in R\rangle\] and hence its abelianization is $K\times \Z$ where $K=R/(1-\gamma^2)$. We will show that both $H$ and $K$ are finite.
 
 For use later in the proof, we first note that $1-a\gamma$ is not a zero divisor in $R$ for any $a\in R$. Indeed, suppose that $(1-a\gamma)r=0$. Then we have that \[r=a\gamma r=a^2\gamma^2r=\ldots\] lies in $\bigcap_{k=0}^\infty (\gamma^k)=0$, and hence $r =0$.
 
 Consider  the quotient $H=R/(1-\gamma)$; we will show that $H$ is finite cyclic. Let $\Z[x]\to R$ be the natural homomorphism sending $x$ to $\gamma$, and let $\frak a$ be the kernel. Then there are surjective ring homomorphisms  \[\Z[x]\to\Z[x]/(1-x)\to \Z[x]/(\frak a + (1-x))\to R/(1-\gamma).\] The restriction of the homomorphism $\Z[x]\to R/(1-\gamma)$ to $\Z$ is surjective, and this shows that $R/(1-\gamma)$ is cyclic. If the image of some $m\in \Z\setminus \{0\}$ in $R/(1-\gamma)$ is zero, then this will show that $R/(1-\gamma)$ is finite. Note that the image of any $f(x)\in \Z[x]$ in $R/(1-\gamma)$ is the same as the image of $f(1)$ and that the image of any $f\in \frak a\subset \Z[x]$ is zero. Thus it suffices to show that $f(1)\neq 0$ for some $f\in \frak a$.

 Suppose for contradiction that $f(1)=0$ for every $f\in \frak a$. Then every element of $\frak a$ is divisible by the linear polynomial $1-x$, and so  $\frak a \subset (1-x)$. Choose a non-zero $f\in \frak a$. We may factor $f$ as $f(x)=(1-x)^kg(x)$ where $k\geq 1$ and $g(1)\neq 0$, so that $g\notin \frak a$. In $R$ we have $0=f(\gamma)=(1-\gamma)^k g(\gamma)$ and $g(\gamma)\neq 0$. Set $r=g(\gamma)$. The binomial theorem implies that $(1-\gamma)^k=1-a\gamma$ for some $a\in R$, and so $(1-a\gamma)r=0$. However, this is a  contradiction, as we have shown that   $1-a\gamma$ is not a zero divisor. Therefore $H=R/(1-\gamma)$ is finite cyclic.
 
 Now consider $K=R/(1-\gamma^2)$. There is a short exact sequence \[0 \to (1-\gamma)/(1-\gamma^2) \to R/(1-\gamma^2) \to R/(1-\gamma)\to 0.\] By our previous work, $R/(1-\gamma)$ is finite cyclic. On the other hand there is an isomorphism of $R$-modules $R/(1+\gamma)\to (1-\gamma)/(1-\gamma^2)$ defined by $r+(1+\gamma) \mapsto r(1-\gamma)+(1-\gamma^2)$. By the same reasoning as for $R/(1-\gamma)$, $R/(1+\gamma)\cong (1-\gamma)/(1-\gamma^2)$ is also finite cyclic. Thus, $R/(1-\gamma^2)$ is finite, as desired. 
 \end{proof}
 
Thus Proposition \ref{prop:generalstructure} applies and we conclude that $\mc H_{qp}(G)\cup \mc H_{\ell}(G)=\mc P_-(G)\cup \mc P_+(G)$, where $\mc P_+(G)$ is the poset of subsets of $\gamma^{-1}R$ which are confining under the action of $\alpha$, while $\mc P_-(G)$ is the poset of confining subsets under $\alpha^{-1}$.  The intersection $\mc P_+(G)\cap \mc P_-(G)$ is the unique lineal structure of $\mc H(G)$.

We now proceed to show that there is a one-to-one correspondence between elements of $\mc P_+(G)$ and ideals of $\widehat R$ up to a certain equivalence relation. 
In Section \ref{sec:conffromideals}, we show how to construct a confining subset of $\gamma^{-1}R$ from an ideal in $\widehat R$.  In Section \ref{sec:idealfromconf}, we show how to construct an ideal in $\widehat R$ from a confining subset of $\gamma^{-1}R$ and analyze how these two constructions relate to each other.   Finally, in Section \ref{sec:P_+isotoideals} we prove our main result (Theorem~\ref{thm:I(G)}), which shows that there is an order-reversing isomorphism between the poset of ideals of $\widehat R$ up to equivalence and $\mc P_+(G)$.  

 The ideas and proofs in the rest of this section are generalizations of  the corresponding facts from \cite{AR}, and the proofs presented here mostly follow those in \cite{AR}.  We include full proofs here even when they are in essence identical to those in \cite{AR} because of the (significant) change of notation involved and to keep the paper self-contained.

We first introduce some new notation for elements of $\gamma^{-1}R$ and $\widehat R$ which will be useful for the sake of brevity in this section. Recall that $T=[d]=\{0,\ldots,d-1\}$  denotes a transversal for the ideal $(\gamma)$ in $R$, where $d$ is a positive integer multiple of 1, provided by Lemma \ref{lem:Tfinitesetofintegers}. An element $s\in R$ can be written in base-$\gamma$ as $s=a_0+a_1\gamma+\cdots,$ where $a_i \in T$ for each $i\geq 0$. An element $r\in \gamma^{-1}R$ has the form $s/\gamma^k$ for some $s\in R$ and some $k\geq 0$. This yields that every element $r\in  \gamma^{-1}R$ can be written as
\[
r=a_{-k}\gamma^{-k}+\cdots +a_{-1}\gamma^{-1}+a_0+a_1\gamma+a_2\gamma^2+\cdots ,
\]
for some $a_i\in T$ and some $k \geq0$.

In this section, and this section only, it will sometimes be convenient to use decimal notation for the $\gamma$--adic addresses of  elements of $\gamma^{-1}R$ and $\widehat R$.  If $r=a_{-k}\gamma^{-k}+\cdots +a_{-1}\gamma^{-1}+a_0+a_1\gamma+a_2\gamma^2+\cdots\in \gamma^{-1}R$, we write $r=\dots a_2a_1a_0.a_{-1}\dots a_{-k}$.  We say $\dots a_2a_1a_0$ is the \emph{integral part} of $r$ and $0.a_{-1}\dots a_{-k}$ is the \emph{fractional part} of $r$. Similarly, if $x=\sum_{i=0}^\infty x_i\gamma^i\in \widehat{R}$, we write $x=\dots x_2x_1x_0$.  We will use both the summation notation and the decimal notation to represent elements of $\gamma^{-1}R$ and $\widehat R$ interchangeably throughout the section. Note that multiplication by $\gamma$ shifts the decimal point one position to the right: if  $r=\ldots a_1 a_0.a_{-1} a_{-2} \ldots a_{-k}$, then $\gamma r=\ldots a_1 a_0 a_{-1}. a_{-2} \ldots a_{-k}.$

We begin by showing that $R$ (as a subset of $\gamma^{-1}R$)  is confining under $\alpha$.
\begin{lem}\label{lem:Rconfining}
$R\subseteq \gamma^{-1}R$ is confining under $\alpha$.
\end{lem}

\begin{proof}
Since $R$ is a ring containing $\gamma$, it is closed under addition and under multiplication by $\gamma$ (which is the automorphism $\alpha$ of $\gamma^{-1}R$). This verifies Definition \ref{def:confining}(a) and Definition \ref{def:confining}(c). Given any $x$ in $\gamma^{-1}R$, we may write $x=\gamma^{-\ell}r$ for some $\ell\in \Z_{\geq 0}$ and $r\in R$, and thus $\alpha^\ell(x)\in R$, verifying Definition \ref{def:confining}(b).
\end{proof}

The following lemma is crucial for identifying confining subsets with ideals of $\widehat{R}$. It is the only place axiom (A5), eventually periodic addresses, is used. Recall that the preorder $\preceq$ on confining subsets under $\alpha$ was defined in Section \ref{sec:confining}.

\begin{lem}\label{lem:RsubsetQ}
For any $Q\subseteq \gamma^{-1}R$ which is confining under $\alpha$, we have $Q\preceq R$ in the preorder $\preceq$ on confining subsets.
\end{lem}

\begin{proof}
We will show that every element of $R$ has uniformly bounded word length with respect to the generating set $Q\cup \{t^{\pm1}\}$, which is equivalent to $Q\preceq R$.   As the fixed transversal  $T$ is finite, there is some constant $K$ such that $\alpha^{K}(f)\in Q$ for all $f\in T$.  Let $\overline Q=Q\cup \bigcup_{i\geq 0}\alpha^i{T}$. By \cite[Lemma~4.9 \& Corollary~4.10]{Qp} (see also \cite[Lemma~3.2]{AR}), $\overline Q$ is confining under $\alpha$ and $Q\sim \overline{Q}$.  Thus it suffices to show that elements of $R$ have bounded word length with respect to $\overline{Q}\cup \{t^{\pm 1}\}$. Note that $\overline Q$ contains all terms of the form $a \gamma^j$ for $a \in T$ and $j \geq 0$.

 We divide the proof into two cases and first consider only those elements of $R$ whose $\gamma$--adic addresses are finite.

\noindent {\bf Case 1:}  Let $R_0$ be the subset of $R$ consisting of the elements whose $\gamma$--adic addresses are finite.  We will show that every element of $R_0$ has uniformly bounded word length with respect to $\overline Q\cup\{t^{\pm 1}\}$. 

\begin{claim}\label{claim:bddlength}
Every element of $\alpha^{k_1}(R_0)$ has uniformly bounded word length $k_1$ in $\overline Q$, where $k_1>0$ is large enough that $\alpha^{k_1}(\overline Q+\overline Q)\subseteq \overline Q$ (see Definition \ref{def:confining}(c)).
\end{claim}

Assuming the claim, the statement about all elements of $R_0$ can be deduced as follows.  Let $x\in R_0$, and suppose $x$ has degree $n_x$, so that 
\begin{equation}\label{eqn:eltofR}
x=\sum_{i=0}^{n_x} a_i\gamma^i=\sum_{i=0}^{k_1-1} a_i\gamma^i + \sum_{i=k_1}^{n_x} a_i\gamma^i,
\end{equation}
where $a_i\in T$. The first term in \eqref{eqn:eltofR} has length $\leq k_1$ in $\overline{Q}$.  The second term in \eqref{eqn:eltofR} is an element of $\alpha^{k_1}(R_0)$ and so has length $\leq k_1$ in $\overline{Q}$ if we prove the claim.  Therefore $x$ will have length at most $2k_1$ with respect to $\overline Q\cup \{t^{\pm 1}\}$. 
We now prove the claim.
\begin{proof}[Proof of Claim~\ref{claim:bddlength}]  Fix $y\in \alpha^{k_1}(R_0)$, and let $n_y$ be the degree of $y$, so that 
\[
y=\sum_{i=k_1}^{n_y} a_i\gamma^i.
\] 
We will prove the statement by induction on $n_y$. For the base case, suppose $n_y\in  [k_1,2k_1)$.  Then 
\[
y=\sum_{i=k_1}^{2k_1-1} a_i\gamma^i,
\]
where $a_i\in T$ for each $i$ and $a_i=0$ for each $n_y<i\leq 2k_1-1$ (if $n_y<2k_1-1$). Since each term $a_i \gamma^i$ is an element of $\overline{Q}$, it follows that $y$ has length at most $k_1$ in $\overline{Q}$.

For the induction step, assume that $y$ has length at most $k_1$ in $\overline Q$ whenever $n_y\in [k_1,(l-1)k_1)$, and suppose $n_y\in [k_1,l k_1)$.  Then we may write 
\[
y=\sum_{i=k_1}^{l k_1-1} a_i\gamma^i= \sum_{i=k_1}^{2k_1-1} a_i\gamma^i + \sum_{i=2k_1}^{l k_1 -1} a_i\gamma^i = g + f,
\]
where $g,f$ are the two summands after the second equals sign. The element $g$ has $k_1$ terms, and  $$\displaystyle g = \alpha^{k_1} \left(\sum_{i=0}^{k_1 -1 } b_i\gamma^i \right),$$ where $b_i=a_{i+k_1}\in T$.
Moreover, $$\displaystyle f = \alpha^{k_1} \left(\sum_{i=k_1}^{(l-1)k_1 -1 } c_i\gamma^i \right) = \alpha^{k_1}(f'),$$ where $c_i\in T$ and $f' \in R_0$ has degree at most $(l-1)k_1 -1$. By the induction hypothesis, $f'$ has length at most $k_1$ in $\overline{Q}$, and so $f = \alpha^{k_1}(f_1 + f_2 + \cdots + f_r)$ with $r \leq k_1$ and $f_j \in \overline{Q}$. Thus 
\begin{align*} y & = \alpha^{k_1} \left(\sum_{i=0}^{k_1 -1 } b_i\gamma^i \right) + \alpha^{k_1}(f_1 + \cdots + f_r) \\
 & = \alpha^{k_1}(b_0 + f_1) + \alpha^{k_1}(b_1 \gamma + f_2) + \cdots + \alpha^{k_1}(b_{r-1} \gamma^{r-1} + f_r) + \alpha^{k_1}(b_r\gamma^r) + \cdots + \alpha^{k_1}(b_{k_1-1} \gamma^{k_1 -1})
\end{align*}
As each $b_i \gamma^i$ is an element of $ \overline{Q}$ and each $f_j$ is an element of $ \overline{Q}$, each term is contained in $\overline Q$ since $\alpha^{k_1}(\overline{Q}+ \overline{Q}) \subset \overline{Q}$ and $\overline Q$ is closed under $\alpha$. Thus $y$ has length at most $k_1$ in $\overline{Q}$, which completes the induction.
\end{proof}

\noindent {\bf Case 2:} We now consider elements of $R$ whose  $\gamma$--adic addresses are infinite.  Let $y\in R$ be such an element.  By axiom (A5), the $\gamma$--adic address of $y$ is eventually periodic, and so $y=y_1 + y_2$, where $y_1$ has a finite $\gamma$--adic address and $y_2$ has a periodic $\gamma$--adic address. Case 1 implies that $\|y_1\|_{\overline Q\cup\{t^{\pm 1}\}}\leq 2k_1$.  As $y_2\in R$, Definition \ref{def:confining}(b) ensures that there is some non-negative integer $N$ such that $\alpha^i(y_2)\in \overline Q$ for all $i\geq N$.  Choose $i\geq N$ so that $i$ is a multiple of the period of $y_2$. This ensures that $z=y_2-\alpha^i(y_2)$ has a finite $\gamma$--adic address, and so    $\|z\|_{\overline Q\cup\{t^{\pm 1}\}}\leq 2k_1$ by Case 1.  Thus $y_2=\alpha^i(y_2)+z$, and, using that fact that $\alpha^i(y_2)\in \overline Q$ by our choice of $i$, we see that $\|y_2\|_{\overline Q\cup\{t^{\pm 1}\}}\leq 2k_1 +1$.  Therefore  $\|y\|_{\overline Q\cup \{t^{\pm 1}\}}\leq 4k_1+1$.
\end{proof}

The following lemma will be useful in the next several subsections.

\begin{lem}\label{lem:idealtopclosed}
Let $\frak a$ be an ideal of $\widehat{R}$. Then $\frak a$ is closed topologically in $\widehat{R}$.
\end{lem}

\begin{proof}
By \cite[Theorem 10.26]{atiyah_macdonald}, $\widehat{R}$ is Noetherian, so every  ideal $\frak a$ of $\widehat R$ is finitely generated. The ring $\hat{R}$ is sequentially compact, and every finitely generated ideal in a sequentially compact ring is closed: if $\frak a=(f_1,\ldots,f_k)$ and $x_i=a_1^if_1 + \cdots +a_k^if_k\in \frak a$ with $x_i\to x\in \widehat{R}$, then we may pass to a subsequence to assume that the sequences $\{a_j^i\}_{i=1}^\infty$ all converge in $\widehat{R}$, say $a_j^i\to a_j\in \widehat{R}$. Then $x_i\to a_1f_1+\cdots+a_kf_k$, and so $x=a_1f_1+\cdots+a_kf_k\in \frak a$.
\end{proof}

\subsection{Constructing confining subsets from ideals}\label{sec:conffromideals}
Given an ideal $\frak a\subseteq \widehat R$, we first show how to construct a confining subset of $\gamma^{-1}R$.  Roughly, given any element $x\in \frak a$, the associated confining subset contains all elements of $\gamma^{-1}R$ whose fractional part coincides with (some number of) the final digits of $x$.  In particular, infinitely many elements of the confining subset can be constructed from a single element of the ideal.  Formally, we have the following definition.

\begin{defn}
Given an ideal $\frak a\subseteq \widehat R$, define 
\[
\mathcal C(\frak a)=\{\dots a_2a_1a_0.a_{-1}\dots a_{-k}\in \gamma^{-1} R: \exists x\in \frak a \textrm{ such that } 
x=\dots x_1x_0a_{-1}\dots a_{-k}  \textrm{ for some } x_i\in T\}.
\] 
\end{defn} 

It is important to note here that the string of non-negative digits $\ldots a_2a_1a_0$ is not arbitrary but must define an element of $R$. This is necessary in order for $\mc C(\frak a)$ to be a subset of $\gamma^{-1}R$.

\begin{ex}
If $R=\Z$ and $\gamma=2$, then $\gamma^{-1}R=\Z[\frac12]$ and $\widehat R=\Z_2$, the ring of $2$--adic integers.  Consider the ideal $\frak a=2^4\Z_2$ in $\widehat R=\Z_2$ and the element $x=\dots 101010000\in \frak a$.  Then, for example, the elements $1111.10101(=1111.101010000), 1.1 (=1.10000), $ and $0.101 (=0.1010000)$ are all elements of $\mc C(\frak a)$. 
\end{ex}

\begin{rem}\label{rem:R in C(frak a)}
Since $0\in \frak a$ for every ideal $\frak a\subseteq \widehat R$, we have that $R\subseteq \mc C(\frak a)$ for every ideal $\frak a$. 
\end{rem}

\begin{lem}
\label{lem:idealtoconfining}
If $\frak a\subseteq \widehat R$ is an ideal, then $\mathcal C(\frak a)$ is confining under $\alpha$.
\end{lem}

\begin{proof}
We will check that Definition \ref{def:confining} holds. 
For condition (a), let $r=\ldots a_2a_1a_0.a_{-1}\ldots a_{-k}\in \mc C(\frak a)$.
  By the definition of $C(\frak a)$, there exists $x=\ldots c_1c_0a_{-1}a_{-2}\ldots a_{-k} \in \frak a$ for some $c_i\in T$.  Thus $
  \alpha(r)=\gamma r= \ldots a_2a_1a_0a_{-1}.a_{-2}\ldots a_{-k}\in \mc C(\frak a),$ and so 
   $\alpha(\mc C(\frak a))\subseteq \mc C(\frak a)$, verifying (a). By Remark \ref{rem:R in C(frak a)}, $R\subseteq \mc C(\frak a)$, and so $\bigcup_{i=0}^\infty \alpha^{-i}(\mc C(\frak a))=\gamma^{-1}R$, verifying (b).
  
  To show condition (c) holds, let $r_1=\dots a_2a_1a_0.a_{-1}\dots a_{-k}$ and $r_2=\dots b_2b_1b_0.b_{-1}\dots b_{-\ell}$ be elements of $\mc C(\frak a)$.   We may assume without loss of generality that $k\geq \ell$.  The sum $r_1+r_2$ is calculated as follows:
 \begin{center} \begin{tabular}{cccccccccc}
 & $\dots$ & $a_1$ & $a_0.$ & $a_{-1}$ &  $\dots$ &  $a_{-\ell}$ & $a_{-\ell-1}$ & $\dots$ & $a_{-k}$\\
 $+$ &$\dots$ & $b_1$ & $b_0.$ & $b_{-1}$ & $\dots$  & $b_{-\ell}$ &&& \\\hline
 = & $\dots$ & $c_1$ & $c_0.$ & $c_{-1}$ & $\dots$ & $c_{-\ell}$ & $a_{-\ell-1}$ & $\dots$ & $a_{-k}$.
  \end{tabular}
 \end{center}
 That is, 
 \begin{equation}\label{eqn:r1+r2}
 r_1+r_2=\dots c_1c_0.c_{-1}\dots c_{-\ell}a_{-\ell-1}\dots a_{-k}.
 \end{equation}
Since $r_1,r_2\in \mc C(\frak a)$, there exist $x_1,x_2\in \frak a$ with $x_1=\dots d_2d_1d_0a_{-1}\dots a_{-k}$ and $x_2=\dots f_2f_1f_0b_{-1}\dots b_{-\ell}$ for some $d_i,f_i\in T$.  Since $\frak a$ is an ideal, 
 $x_1+\gamma^{k-\ell}x_2$ is also in $ \frak a$.  We see that  
\begin{align*}
x_1+\gamma^{k-\ell}x_2&=\left(\dots d_2d_1d_0a_{-1}\dots a_{-k}\right) + (\dots f_2f_1f_0b_{-1}\dots b_{-\ell}\underbrace{0\dots\dots 0}_{k-\ell \textrm{ times}}) \\
	& = \dots z_1z_0c_{-1}\dots c_{-\ell} a_{-\ell-1}\dots a_{-k}
\end{align*}
 for some $z_i\in T$.  Comparing this with \eqref{eqn:r1+r2}, we see that $r_1+r_2\in \mc C(\frak a)$.  Therefore (c) holds with $k_0=0$.
\end{proof}

In fact, the structure of $\mc C(\frak a)$ is even nicer than suggested by Lemma \ref{lem:idealtoconfining}.

\begin{lem}
\label{lem:idealtosubring}
If $\frak a\subseteq \widehat R$ is an ideal, then the confining subset $\mathcal C(\frak a)$ is a subring of $\gamma^{-1}R$ containing $R$.
\end{lem}

\begin{proof}
By the proof of Lemma \ref{lem:idealtoconfining}, the set $\mathcal C(\frak a)$ is closed under addition and contains $R$. We will prove that it is also closed under multiplication.

Any element of $\gamma^{-1}R$ can be written as $a\gamma^{-k} +x$, where $a=a_0+a_1\gamma+\dots +a_{k-1}\gamma^{k-1}$ and $x\in R$, for any sufficiently large $k$.  Here, $x$ is the ``integral part'' of the element (involving only non-negative powers of $\gamma$) and  $a\gamma^{-k}$ is the ``fractional part'' (involving only negative powers).  Note that we are free to take any sufficiently large $k$ because the fractional part of the element can be made to have as many additional digits as desired simply by adding zero digits.

In particular, given any $r,s\in \mc C(\frak a)$,  we may choose $k$ large enough to write $r=a\gamma^{-k}+x$ and $s=b\gamma^{-k}+y$ with $x,y\in R$ and $a=a_0+a_1\gamma+\cdots+a_{k-1}\gamma^{k-1}$ and $b=b_0+b_1\gamma+\cdots+b_{k-1}\gamma^{k-1}$. Since $r,s\in \mc C(\frak a)$, by definition there must exist elements $a_0+a_1\gamma+\cdots+a_{k-1}\gamma^{k-1}+d_k\gamma^k+\cdots=a+z\gamma^k$  and  $b_0+b_1\gamma+\cdots+b_{k-1}\gamma^{k-1}+e_k\gamma^k+\cdots=b+w\gamma^k$ in $\frak a$, where $z,w\in \widehat{R}$. We will show that $rs\in \mc C(\frak a)$.  

We have \[rs=(a\gamma^{-k}+x)(b\gamma^{-k}+y)=(ab+ay\gamma^k+bx\gamma^k)\gamma^{-2k}+xy.\] Since $xy\in R$, the fractional part of $rs$ is contained in $(ab+ay\gamma^k+bx\gamma^k)\gamma^{-2k}$.
Hence it suffices to show that the first $2k$ digits of $ab+ay\gamma^k+bx\gamma^k$ agree with the first $2k$ digits of some element of $\frak a$.

As $a+z\gamma^k\in \frak a$ and $b+w\gamma^k\in \frak a$, the ideal $\frak a$ contains the element 
\[(a+z\gamma^k)(b+w\gamma^k)+(a+z\gamma^k)(-w\gamma^k+y\gamma^k)+(b+w\gamma^k)(-z\gamma^k+x\gamma^k) = ab+ay\gamma^k+bx\gamma^k+\gamma^{2k}(zy+wx-zw).\] 
  The first $2k$ digits of this element agree with the first $2k$ digits of $ab+ay\gamma^k+bx\gamma^k$, completing the proof.
\end{proof}

\subsection{Constructing ideals from confining subsets}\label{sec:idealfromconf}

We now show how to construct an ideal of $\widehat R$ from a confining subset of $\gamma^{-1}R$.   Roughly speaking, the ideal is the set of all elements of $\widehat R$ with the property  that every terminal string of digits appears as the fractional part of \emph{some} element of the confining subset. Alternatively, it may be thought of as a kind of limit set of the fractional parts of elements of $Q$. 
\begin{defn}
Given a subset $Q\subset \gamma^{-1}R$ that is confining under $\alpha$, define $\mathcal I(Q)$ by 
\[
\mathcal I(Q)\vcentcolon=\{\ldots a_2a_1a_0\in \widehat{R}: \forall t\geq 0 \,\exists q\in Q \textrm{ such that } q=\dots b_2b_1b_0.a_t\dots a_0
\textrm{ for some } b_i\in T\}.
\]
\end{defn}

\begin{lem} 
If $Q\subset \gamma^{-1}R$ is confining under $\alpha$, then $\mathcal I(Q)$ is an ideal of $\widehat R$.
\end{lem}

\begin{proof}	
We first show that $\mc I(Q)$ is closed under addition.  Let $x=\dots x_2x_1x_0$ and $y=\dots y_2y_1y_0$ be elements of $\mc I(Q)$.  Then $x+y=z=\dots z_2z_1z_0$ for some $z_i\in T$.  Let $k_0$ be large enough that $\alpha^{k_0}(Q+Q)\subset Q$.  By the definition of $\mc I(Q)$, for any $t\geq 0$, there exist $r=\dots a_2a_1a_0.x_{t+k_0}\dots x_0$ and $ s=\dots b_2b_1b_0.y_{t+k_0}\dots y_0$ in $Q$  for some $a_i,b_i\in T$.  Thus $
r+s=\dots c_2c_1c_0.z_{t+k_0}\dots z_0$
for some $c_i\in T$, and
\[
\alpha^{k_0}(r+s)=\dots c_2c_1c_0z_{t+k_0}\dots z_{t+1}.z_t\dots z_0\in Q.
\]
Since $t$ is arbitrary, it follows that $z\in \mc I(Q)$.

We next show that $\mc I(Q)$ is closed under multiplication by elements of $\widehat R$. We do this in three steps. First, since $\mc I(Q)$ is closed under addition, $\mc I(Q)$ is closed under multiplication by any positive integer multiple  $z=z\cdot 1$ of $1$ in $\widehat{R}$: if $x\in \mc I(Q)$, then 
\[zx=\underbrace{x+\cdots +x}_{z \text{ times}} \in \mc I(Q).\] In particular, since $T=[d]=\{0,\dots, d-1\}$, it follows that $\mc I(Q)$ is closed under multiplication by elements of $T$.
Second, we show that $\mc I(Q)$ is closed under multiplication by $\gamma$.  If $x=\dots x_2x_1x_0 \in \mc I(Q)$, then for any $s\geq 0$ there exists $r=\dots a_2a_1a_0.x_s\dots x_0\in Q$, for some $a_i\in T$.  But \[r=\dots a_2a_1a_0.x_s\dots x_0=\dots a_2a_1a_0.x_s\dots x_00,\] and since $
\gamma x= \dots x_2x_1x_00$,
 this shows that $\gamma x\in \mc I(Q)$.
 
Third, we show that $\mc I(Q)$ is (topologically) closed in $\widehat R$. Fix $x=\dots x_2x_1x_0\in \overline{ \mc I(Q)}$.  Then for any $s\geq 0$, there exists $y=\dots y_2y_1y_0\in \mc I(Q)$ with $y_i=x_i$ for all $i\leq s$.  By the definition of $\mc I(Q)$, there exists $r\in Q$ with $r=\dots a_2a_1a_0.y_s\dots y_0$ for some $a_i\in T$.  But then we also have $r=\dots a_2a_1a_0.x_s\dots x_0$.  Since $s$ was arbitrary, this shows that $x\in \mc I(Q)$.

We can now put these three steps together to show that $\mc I(Q)$ is closed under multiplication by elements of $\widehat R$.  Let $x=\dots x_2x_1x_0\in\mc I(Q)$ and $r=\dots r_2r_1r_0\in \widehat R$, where $r_i\in T$.   For any $s\geq 0$, we have 
\[
\left(\sum_{i=0}^s r_i\gamma^i\right)\cdot x=\sum_{i=0}^s r_i\gamma^ix \in \mc I(Q).
\]  Since $\mc I(Q)$ is topologically closed, this implies that 
\[
rx=\lim_{s\to\infty}\left(\sum_{i=0}^s r_i\gamma^i\right)\cdot x =\lim_{s\to\infty}\sum_{i=0}^s r_i\gamma^i x\in \mc I(Q),\] as desired.  Therefore $\mc I(Q)$ is an ideal of $\widehat R$.
\end{proof}

\subsection{Saturated ideals and the poset $\mc P_+(G)$}\label{sec:P_+isotoideals}

In this subsection, we complete our discussion of the relationship between ideals and the poset $\mathcal P_+(G)$. We first define a preorder on ideals which leads to a notion of equivalence for ideals.

\begin{defn}
Define a preorder on ideals of  $\widehat R$ by $\frak a\leq \frak b$ if for all $x\in \frak a$, there exists $i\in \Z_{\geq 0}$ such that $\gamma^ix\in \frak b$. This induces an equivalence relation $\sim$ on the ideals of $\widehat R$ by $\frak a \sim \frak b$ whenever $\frak a\leq\frak b$ and $\frak b\leq \frak a$.
 \end{defn}
 
 The preorder $\leq$ induces a partial order $\preccurlyeq$ on the resulting set of equivalence classes of ideals.
 
 \begin{rem}
 Since ideals in $\widehat R$ are finitely generated, $\frak a \leq \frak b$ exactly when $\gamma^i \frak a\subset \frak b$ for some $i\geq 0$, i.e., there is a uniform power $i$ such that $\gamma^i x\in \frak b$ for all $x\in \frak a$, even though this is not required by the definition.
 \end{rem}
 
The following lemma shows that this relation on ideals of $\widehat R$ interacts well with the way we constructed confining subsets from ideals.
  \begin{lem}
 \label{lem:idealconfiningorder}
 If $\frak a\leq \frak b$, then $\mathcal C(\frak a)\subset\mathcal C(\frak b)$. If $\frak a \sim \frak b$, then $\mc C(\frak a)=\mc C(\frak b)$.
 \end{lem}
 
 \begin{proof}
 Assume $\frak a\leq \frak b$. By the definition of $\mc C(\frak a)$, elements of the ideal $\frak a$ determine only the coefficients of the negative powers of $\gamma$ in elements of $\mc C(\frak a)$. That is, for each $a=\ldots a_1 a_0\in \frak a$, $n\geq 0$, and $z\in R$, there is an element $a_0\gamma^{-n}+\cdots +a_{n-1}\gamma^{-1} + z$ in $\mathcal C(\frak a)$. Moreover, if $a=\dots a_1 a_0$ lies in $\frak a$, then $a$ and $\gamma^i a=\dots a_n\dots a_0\underbrace{0\dots 0}_{i \textrm{ times}}$ determine the same elements of $\mc C(\frak a)$.  Since there exists $i$ such that $\gamma^i\frak a\subseteq \frak b$, this implies that $\mc C(\frak a)\subseteq \mc C(\frak b)$. The second sentence of the lemma follows immediately.
 \end{proof}

We now choose a canonical representative of each equivalence class $[\frak a]$ of ideals of $\widehat R$ under $\sim$. Since $\gamma$ is not a zero divisor in $\widehat R$, there is a natural injection from $\widehat R$ to the localization $\gamma^{-1} \widehat R$, which we denote by $f$.
If $\frak a$ is an ideal of $\widehat R$, its image $f(\frak a)$ generates an ideal denoted by $\frak a^e$. The ideal $f^{-1}(\frak a^e)$ contains $\frak a$ and is called the \emph{saturation} of $\frak a$. An ideal is \emph{saturated} if $\frak a$ is its own saturation: $\frak a=f^{-1}(\frak a^e)$. Equivalently, $\frak a\subset \widehat R$ is saturated if and only if whenever $r\in \widehat R$ and $\gamma^i r\in \frak a$ for some $i\geq 0$, we have $r\in \frak a$. 
Recall that an element of a poset is \emph{largest} when it is greater than or equal to every other element of the poset.

\begin{lem}
The saturation of an ideal $\frak a\subseteq \widehat R$ is the unique largest element in the equivalence class $[\frak a]$ with respect to the partial order given by inclusion.
\end{lem} 

\begin{proof}
Let  $\frak m$  be the saturation of $\frak a$. 
Since $\frak a \subset \frak m$, it follows that $\frak a\leq \frak m$. For any $r\in \frak m$, we have $\gamma^i r\in \frak a$ for some $i\geq 0$, and so  $\frak m \leq \frak a$.  Thus $\frak m \sim \frak a$. To see that $\frak m$ is largest in $[\frak a]$, let $\frak b\in [\frak a]$.  Since $\frak b\sim \frak a$,  for every $b\in \frak b$, there is an $i\in \Z_{\geq 0}$ such that $\gamma^i b \in \frak a \subset \frak m$. As $\frak m$ is saturated, this implies that $b\in \frak m$, and so $\frak b\subseteq \frak m$.
\end{proof}

We may form the poset of saturated ideals of $\widehat R$  with the partial order given by inclusion. 
The following general lemma shows that this poset is isomorphic to the poset of (all) ideals of $\gamma^{-1} \widehat{R}$ with inclusion.

\begin{lem}[{\cite[Ch. 2 Sec. 2 No. 4 Proposition 10]{bourbaki_comm}}]
\label{lem:localideals}
Let $A$ be a ring, $S\subset A$ a multiplicatively closed subset of $A$, and $S^{-1}A$ the localization. Let $f:A\to S^{-1}A$ be the natural homomorphism. Then the map sending an ideal $\frak b\subset S^{-1}A$ to its preimage $f^{-1}(\frak b)$ is an isomorphism from the poset of ideals of $S^{-1}A$ with inclusion to the poset of saturated ideals of $A$ with inclusion.
\end{lem}

Our goal is to prove the following theorem. 
 \begin{thm}\label{thm:I(G)}
 In the notation and terminology of Theorem~\ref{thm:main}, $\mc P_+(G)$ is isomorphic to the opposite of the poset of saturated ideals of $\widehat R$ and hence to the opposite of the poset of ideals of the localization $\gamma^{-1} \widehat R$.
\end{thm}

The next two lemmas show the relationship between the order on saturated ideals and the order on $\mc P_+(G)$. Again we will investigate the structure of the preorder $\preceq$ on confining subsets under $\alpha$.

\begin{lem}\label{lem:incomparable}
If $\frak a,\frak b$ are saturated ideals of $\widehat R$ such that $\frak a\not\leq \frak b$ and $\frak b \not\leq \frak a$, then $\mc C(\frak a)$ and $\mc C(\frak b)$ are incomparable with respect to the preorder $\preceq$ on confining subsets.
\end{lem}

\begin{proof}
Since $\frak a \not\leq \frak b$, we have that $\frak a\not\subseteq \frak b$.  Thus there exists $x=\dots x_2x_1x_0\in \frak a$ such that $x\not\in \frak b$.  Since $x\not\in\frak b$, there exists some non-negative integer $M$ such that no element $y=\dots y_2y_1y_0\in \frak b$ satisfies $y_i=x_i$ for all $i\leq M$.  To see this, notice that if such an element of $\frak b$ existed for each $M\in \Z_{\geq 0}$, then the limit of these elements would be equal to $x$.  Since $\frak b$ is topologically closed by Lemma \ref{lem:idealtopclosed}, this would contradict $x\notin \mathfrak b$. 

For any $K\in \Z_{\geq 0}$, there is an element $q\in \mc C(\frak a)$ such that $q=\dots q_2q_1q_0.x_K\dots x_0$. Assume $K\geq M$, and consider the smallest $u\in \Z$ such that $\alpha^u(q)\in \mathcal C(\frak b)$. 
We will show that $u>K-M$, so suppose for contradiction that $u\leq K-M$. Then the digits of $\alpha^u(q)$ to the right of the decimal point are $x_{K-u} x_{K-u-1} \ldots x_0$ and $K-u\geq M$.  These digits are uniquely determined by $\alpha^u(q)$ up to possibly adding or deleting zeroes all the way to the right of the expression. Since $\alpha^u(q)\in \mc C(\frak b)$, there is an element $y\in \frak b$ determining the same element of $\gamma^{-1} R$. That is, $y=\ldots y_{M+2} y_{M+1} x_M \ldots x_0 0 \ldots  0$ with some number of zeroes at the end. Since $\frak b$ is saturated, we have $y'=\ldots y_{M+2} y_{M+1} x_M \ldots x_0\in \frak b$ (with no extra zeroes). This contradicts our choice of $M$, and so $u>K-M$. Thus $\mc C(\frak b)\not\preceq \mc C(\frak a)$, since $K$ is arbitrarily large. A symmetric argument shows that $\mc C(\frak a)\not\preceq \mc C(\frak b)$, so that the confining subsets are incomparable.
\end{proof}

\begin{lem}\label{lem:orderrev}
If $\frak a,\frak b$ are two saturated ideals of $\widehat R$ such that $\frak a\lneq \frak b$, then $\mc C(\frak a) \succeq C(\frak b)$ and $\mc C(\frak a) \not\sim \mc C(\frak b)$.
\end{lem}

\begin{proof}
Since $\frak a\leq \frak b$, we have $\mc C(\frak a)\subset \mc C(\frak b)$ by Lemma \ref{lem:idealconfiningorder}. Thus $\mc C(\frak a)\succeq \mc C(\frak b)$. 
Since $\frak a \lneq \frak b$, an argument similar to the proof of Lemma~\ref{lem:incomparable} produces elements of $q\in \mc C(\frak b)$ with $\inf\{ u:\alpha^u(q)\in \mc C(\frak a)\}$ arbitrarily large, which shows that $\mc C(\frak b) \not\sim \mc C(\frak a)$.
\end{proof}

The final lemmas necessary to prove Theorem~\ref{thm:I(G)} show that, up to equivalence of hyperbolic structures, $\mathcal C(\mathcal I(Q))$ is the same as $Q$ for any confining subset $Q$. First we show that $Q\preceq \mathcal C (\mathcal I(Q))$.

\begin{lem}\label{lem:CIQinalphaKQ}
If $Q\subseteq \gamma^{-1}R$ is confining under $\alpha$, then $Q \preceq \mc C(\mc I(Q))$.
\end{lem}

\begin{proof}
It suffices to find $K\geq 0$ with $\mc C(\mc I(Q))\subset \alpha^{-K}(Q)$ by Lemma~\ref{lem:confiningpartialorder}. Let $a=\dots a_2a_1a_0.a_{-1}\dots a_{-k}\in \mc C(\mc I(Q))$.  By definition of $\mc C(\mc I(Q))$, there exists an element $x\in \mc I(Q)$ such that $x=\dots c_2c_1c_0a_{-1}\dots a_{-k}$, where $c_i\in T$.  By definition of $\mc I(Q)$, there is an element $q\in Q$ such that $q=\dots b_2b_1b_0.a_{-1}\dots a_{-k}$ for some $b_i\in T$.   Let $r=a-q\in R$.  

By Lemma \ref{lem:RsubsetQ}, there is a constant $M\geq 0$ such that $\alpha^M(R)\subset Q$, and so $\alpha^M(a)=\alpha^M(q)+\alpha^M(r)\in Q+Q$.  Let $k_0$ be large enough that $\alpha^{k_0}(Q+Q)\subseteq 	Q$.  Then  $\alpha^{M+k_0}(a)=\alpha^{k_0}(\alpha^M(a))\in \alpha^{k_0}( Q+Q)\subseteq Q$, as desired. Setting $K=M+k_0$ concludes the proof.
\end{proof}

We now prove the equivalence of $Q$ and $\mathcal C(\mathcal I(Q))$.

\begin{lem}\label{lem:surj}
If $Q\subseteq \gamma^{-1}R$  is confining under  $\alpha$, then $\mc C(\mc I(Q))\sim Q$.
\end{lem}

\begin{proof}
It follows from Lemma \ref{lem:CIQinalphaKQ} that $\mc C(\mc I(Q))\succeq Q$. We will show that $\mc C(\mc I(Q))\preceq Q$.  Suppose this is not the case.  Then $Q\not\subseteq \alpha^{-k}(\mc C(\mc I(Q)))$ for any $k\in \Z$ by Lemma~\ref{lem:confiningpartialorder}, and so there are elements $a\in Q$ with $m_a\vcentcolon=\inf\{k: \alpha^k(a)\in \mc C(\mc I(Q))\}$  arbitrarily large.  Choose such an element $a=\dots a_0.a_{-1}\dots a_{-\ell}\in Q$ with $m=m_a > k_0$, where $k_0$ is as in Definition \ref{def:confining}(c), which holds for $Q$. 

Choose $s\leq \ell$ largest such that no element of $\mc I(Q)$ has the form $\dots a_{-s}\dots a_{-\ell}$, that is, such that no element of $\mc I(Q)$ agrees with the last $\ell-s+1$ digits of $a$.  We must have $s\geq m$,  for $\alpha^s(a)=\dots a_0a_{-1}\dots a_{-s}.a_{-s-1}\dots a_{-\ell}\in \mc C(\mc I(Q))$, and if $s< m$, then this contradicts the definition of $m$ as an infimum.

The lemma will follow from the following  claim, whose proof we defer for the moment.

\begin{claim}\label{claim:1}
For $u\in \Z_{\geq 0}$ arbitrarily large, there is an element $d=\dots d_0.d_{-1}\dots d_{-u}\in Q$ with the property that there does not exist an element of the form $\dots d_{-u}\in \mc I(Q)$.
\end{claim}

Assuming Claim \ref{claim:1}, there is a sequence $u_i\to\infty$ and a sequence $\{d^i\}_{i=1}^\infty$ of elements of $Q$ with $d^i=\dots d^i_0.d^i_{-1}\dots d^i_{-u_i}$ with the property that there are no elements of the form $\dots d^i_{-u_i}$ in $\mc I(Q)$.  By passing to a subsequence, we may assume that the sequence of elements $d^i_{-1}\dots d^i_{-u_i}\in R\subseteq \widehat{R}$ converges to an element $\dots e_2e_1e_0\in \widehat R$.

 Given any $v\geq 0$ and any sufficiently large $i$, we have 
\[
d^i_{-1}\dots d^i_{-u_i}=d^i_{-1}\dots d^i_{-u_i+v+1}e_v\dots e_0.
\]  Thus 
\[
\alpha^{u_i-v-1}(d^i)=\dots d^i_0d^i_{-1}\dots d^i_{-u_i+v+1}.d^i_{-u_i+v}\dots d^i_{-u_i}=\dots d^i_0d^i_{-1}\dots d^i_{-u_i+v+1}.e_v\dots e_0 \in Q.
\]
This proves that $\dots e_1 e_0\in \mc I(Q)$.  However, this is a contradiction, 
as $d^i_{-u_i}=e_0$ for sufficiently large $i$, but there does not exist an element of the form $\dots d^i_{-u_i}$ in $\mc I(Q)$ by Claim \ref{claim:1}.
\end{proof}

We now prove  Claim \ref{claim:1}.
\begin{proof}[Proof of Claim \ref{claim:1}]
Consider an element $a=\dots a_0.a_{-1}\dots a_{-\ell}\in Q$ as in the first paragraph. If $s=\ell$, then by the definition of $s$ there is no element of the form $\dots a_{-\ell}$ in $\mc I(Q)$.  In this case, we take $d=a$.

On the other hand, if $s<\ell$, then by the definition of $s$ there exists 
\[
x=\dots x_2x_1x_0a_{-s-1}\dots a_{-\ell}\in \mc I(Q).
\]
Let $y=\dots y_2y_1y_0\in \mc I(Q)$ be the additive inverse of $x$, so that $x+y=0$.

By the definition of $\mc I(Q)$, there exists $b=\dots b_2b_1b_0.y_{\ell-1}\dots y_0\in Q$.  Then $c\vcentcolon=a+b\in Q+Q$, so  $\alpha^{k_0}(c)\in Q$.  The element $c=\dots c_0.c_{-1}\dots c_{-s}$ is given by
 \begin{center} \begin{tabular}{cccccccccc}
&$\dots $ & $a_1$ & $a_0.$ & $a_{-1}$& $\dots$ & $a_{-s}$ & $a_{-s-1}$ & $\dots$ & $a_{-\ell}$ \\
$+$& $\dots$ & $b_1$ & $b_0.$ & $y_{\ell-1}$ & $\dots$ & $y_{\ell-s}$ & $y_{\ell-s-1}$ & $\dots $ & $y_0$
  \\\hline
$=$ &$\dots$ & $c_1$ & $c_0.$ & $c_{-1}$ & $\dots$ & $c_{-s}$ & $0$ & $\dots$ & $0$ 
  \end{tabular}
 \end{center}
Recall that $s \geq m > k_0$. Therefore $\alpha^{k_0}(c)=\dots c_0c_{-1}\dots c_{-k_0}.c_{-k_0-1}\dots c_{-s}$. We will show that there does not exist an element $z\in \mc I(Q)$ of the form $\dots c_{-s}$.  This will show that we can take $d=\alpha^{k_0}(c)$ in this case.  To see this, suppose there were such an element $z=\dots z_2z_1z_0c_{-s}\in \mc I(Q)$.  Then since $s<\ell$, we have 
 \[
 \gamma^{\ell-s}z=\dots z_2z_1z_0c_{-s}\underbrace{0 \dots\dots 0}_{\ell-s \textrm{ times}}\in \mc I(Q).
 \]  
 By the definition of $c$, we see that $\gamma^{\ell-s}z-y\in \mc I(Q)$ is given by 
 \begin{center} \begin{tabular}{cccccc}
& $\dots$ & $c_{-s}$ & $0$ & $\dots$ & $0$ \\
$-$ & $\dots$ & $y_{\ell-s}$ & $y_{\ell-s-1}$ & $\dots$ & $y_0$ 
  \\\hline
$=$ & $\dots$ & $a_{-s}$ & $a_{-s-1}$ & $\dots$ & $a_{-\ell}$,
  \end{tabular}
 \end{center}
 which contradicts the definition of $s$. Since $s\geq m$, the quantity  $s-k_0$ may be chosen arbitrarily large. This  completes the proof of Claim \ref{claim:1} and the proof of the lemma.
\end{proof}

We are now ready to prove Theorem \ref{thm:I(G)}.

\begin{proof}[Proof of Theorem \ref{thm:I(G)}]
By Proposition \ref{prop:Grgammaabelianization}, the conditions of Proposition \ref{prop:generalstructure} are satisfied.  By Proposition~\ref{prop:Ppmconfining},  we may identify $\mc P_+(G)$ with the poset of equivalence classes of confining subsets of $\gamma^{-1}R$ under $\alpha$.  
Let $\mc S$ denote the poset of saturated ideals of $\widehat R$, and define a map $\phi\colon \mc S\to \mc P_+(G)$ by $\phi(\frak a)= \mc [C(\frak a)\cup \{t^{\pm 1}\}]$.

Lemmas \ref{lem:incomparable} and \ref{lem:orderrev} show that $\phi$ is an injective, order-reversing map of posets.  
 Fix $[T]\in \mc P_+(G)$. Again by Proposition \ref{prop:Ppmconfining}, there is a subset $Q\subseteq \gamma^{-1}R$ which is confining under $\alpha$ such that $[Q\cup\{t^{\pm 1}\}]=[T]$.  Lemma~\ref{lem:surj} then implies that $\phi(\mc I(Q))=[\mc C(\mc I(Q))\cup \{t^{\pm 1}\}]=[Q\cup\{t^{\pm 1}\}]=[T]$.  Therefore $\phi$ is surjective.
\end{proof}

\section{Elements of $\mc P_+(G)$: valuations and actions on trees}\label{sec:vals} 
\label{sec:valuations}

In this section, we complete the proof of Theorem \ref{thm:main}. Let $G=G(R,\gamma)$, where $R$ and $\gamma$ satisfy (A1)--(A5). The remaining step of the proof is to show that each element of $\mathcal P_+(G)$ is represented by an action on a tree.

\begin{thm}\label{thm:P+tree}
Let $G=G(R,\gamma)$.  Every element of $\mc P_+(G)$ contains an action on a tree. 
\end{thm}

This theorem follows quickly from the following:

\begin{prop}[{\cite[Proposition 3.14]{AR}}]
\label{prop:treeschwarzmilnor}
Let $G$ be a group which may be expressed as an ascending HNN extension \[A *_A=\langle A, s : sas^{-1}=\phi(a) \text{ for all } a\in A\rangle,\] where $A$ is a group and $\phi$ is an injective endomorphism of $A$.  The action of $G$ on the Bass-Serre tree associated to this HNN extension is equivalent to its action on $\Gamma(G,A\cup \{s^{\pm 1}\})$.
\end{prop}

\begin{proof}[Proof of Theorem \ref{thm:P+tree}]
An element of $\mathcal P_+(G)$ has the form $[\mathcal C(\frak a)\cup \{t^{\pm 1}\}]$ by Lemma \ref{lem:surj}. One may check that $G$ is isomorphic to the ascending HNN extension \[\langle \mathcal C(\frak a), t : tzt^{-1}=\gamma z \text{ for } z\in \mathcal C(\frak a)\rangle\] (see, e.g., the proof of \cite[Lemma 3.13]{AR}). There is an action of $G$ on the resulting Bass-Serre tree, and by Proposition \ref{prop:treeschwarzmilnor}, this action is a representative for the hyperbolic structure $[\mathcal C(\frak a)\cup \{t^{\pm 1}\}]$.
\end{proof}

\subsection{Proof of Theorem \ref{thm:main}}

The proof of Theorem \ref{thm:main} now follows immediately by combining our previous results.

\begin{proof}[Proof of Theorem \ref{thm:main}]
Let $G=G(R,\gamma)$. By Propositions \ref{prop:generalstructure} and   \ref{prop:Grgammaabelianization}, the poset $\mathcal H(G)$ splits into two lattices
$\mathcal P_+(G)$ and $\mathcal P_-(G)$ meeting in a unique lineal structure that dominates the unique elliptic structure.
By Theorem \ref{thm:I(G)}, $\mathcal P_+(G)$ is isomorphic to the poset of ideals of $\widehat R$ up to multiplication by $\gamma$ (equivalently the poset of \emph{saturated} ideals of $\widehat R$ with inclusion, equivalently the poset of \emph{all} ideals of $\gamma^{-1}\widehat R$ with inclusion). Moreover, the elements of $\mathcal P_+(G)$ are represented by actions on trees by Theorem \ref{thm:P+tree}.
\end{proof}

A downside of the proof of Theorem \ref{thm:P+tree} is that it is a priori quite difficult to visualize the Bass-Serre trees representing the elements of $\mathcal P_+(G)$. However, very explicit pictures of the Bass-Serre trees and their group actions may be made in practice using valuations on rings. We devote the rest of this section to describing this explicit construction; see Proposition \ref{prop:isomorphictree}. The construction of trees from valuations is analogous to that of \cite[Section I.3]{alp_bass} and \cite{brown}. We refer the reader to \cite{alp_bass} and \cite{brown} for further background.

We initially work in a more general setting and construct a tree from a valuation on an abstract set. In this level of generality, there is no group action. In  Section \ref{sec:groupvaluations}, we describe how to use this construction to incorporate a group action. 

Consider a set $X$ endowed with the discrete topology and a function $v\colon X\times X\to \R\cup \{\infty\}$ satisfying the following properties:
\begin{enumerate}[(V1)]
\item $v(x,y)=v(y,x)$;
\item $v(x,x)=\infty$; and 
\item $v(x,z)\geq \min\left\{v(x,y),v(y,z)\right\}$ for any $x, y, z \in X$, with equality unless $v(x,y)= v(y,z)$.
\end{enumerate}

We produce a tree $T$ from such a set $X$ and function $v$ as follows. 
\begin{defn}\label{def:tree}
Equipped with the product topology, the set $X\times \R$  is homeomorphic to a disjoint union of lines. Define an equivalence relation $\sim$ on $X\times \R$ by $(x,h)\sim (y,h)$ if $h\leq v(x,y)$. Note that $\sim$ is indeed  an equivalence relation, since  $(x,h)\sim(y,h)\sim(z,h)$ implies $h\leq \min\{v(x,y),v(y,z)\}\leq v(x,z)$. Finally,  consider the quotient space  $T=(X\times \R)/\sim$, endowed with the quotient topology, and the quotient map $\pi\colon X\times \R \to T$. If we want to emphasize the function $v$ which gave rise to $T$, we will write $T_v$. See Figure \ref{fig:easytree} for an easy example of $T_v$.
\end{defn}

\begin{figure}[h]
\begin{center}
\def\svgwidth{0.7\textwidth}
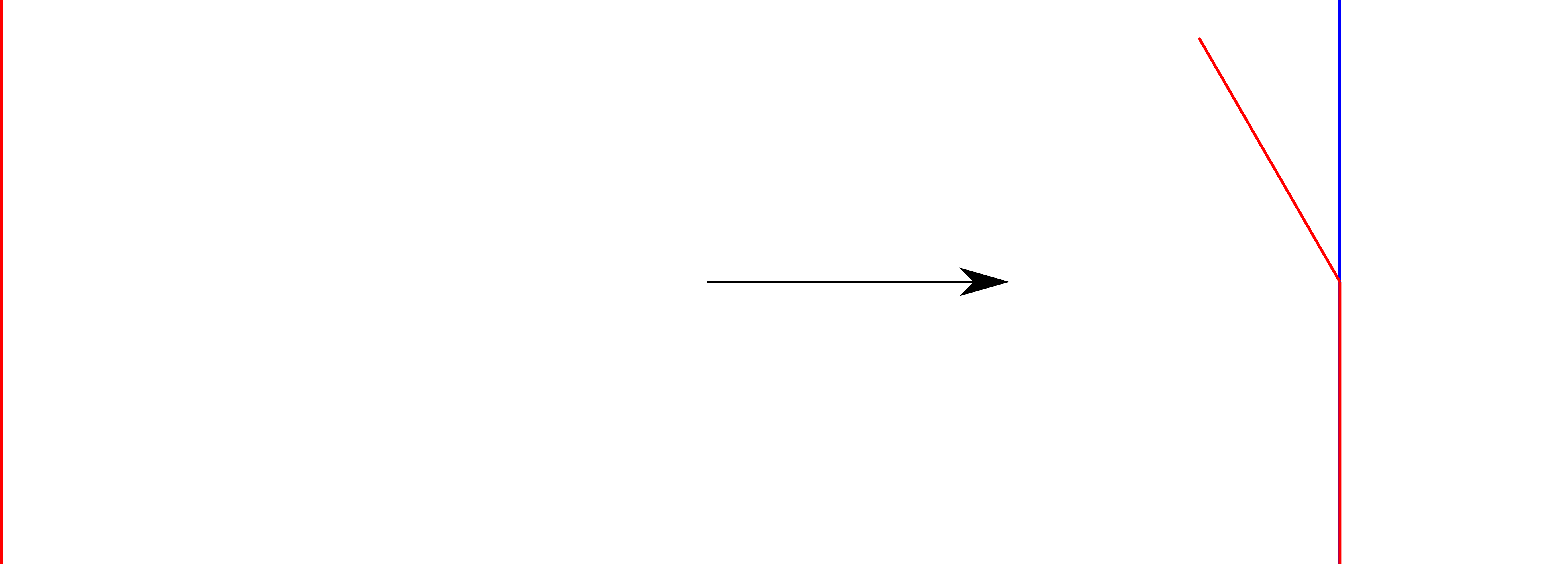
\end{center}
\caption{The tree $T_v$ for the function $v\colon \{0,1,2\}^2\to \R\cup \{\infty\}$ defined by $v(0,0)=v(1,1)=v(2,2)=\infty$, $v(0,1)=v(1,0)=1$, $v(0,2)=v(2,0)=0$, $v(1,2)=v(2,1)=0$.}
\label{fig:easytree}

\end{figure}

Recall that an \emph{arc} is a topological embedding of $[0,1]$ and that an \emph{$\R$-tree} is a metric space for which there is a unique arc joining any pair of points, which is a geodesic. We now introduce a metric on $T_v$. Consider two points $p_1=\pi(x_1,h_1)$ and $p_2=\pi(x_2,h_2)$ in $T_v$. We may assume without loss of generality that $h_2\leq h_1$. If $(x_1,h_2)\sim (x_2,h_2)$ then $\pi(\{x_1\} \times [h_2,h_1])$ is an arc between $p_1$ and $p_2$. Otherwise $(x_1,h_2)\not\sim (x_2,h_2)$ and one may check that the quantity $v(x_1,x_2)$ is well-defined, independent of representatives of $p_1$ and $p_2$. In this case $\pi(\{x_1\}\times [v(x_1,x_2),h_1]) \cup \pi(\{x_2\}\times [v(x_1,v_2),h_2])$ is an arc between $p_1$ and $p_2$. We define: \[d(p_1,p_2)=\begin{cases} |h_2-h_1| & \textrm{ if }(x_1,h_2)\sim (x_2,h_2) \textrm{ or } (x_1,h_1)\sim (x_2,h_1) \\ |h_2-v(x_1,x_2)|+|h_1-v(x_1,x_2)| & \textrm{ otherwise}\end{cases}.\]

\begin{thm}\label{thm:Tatree}
The space $T_v$ equipped with the metric $d(\cdot,\cdot)$ is an $\R$-tree.
\end{thm}

This construction and proof is closely related to that of \cite[Theorem 3.9]{alp_bass}, so we omit the proof here and refer the reader to that paper.

\subsection{Valuations and actions of $G(R,\gamma)$ on trees}
\label{sec:groupvaluations}
Thus far we have constructed trees associated to valuations on abstract sets. In this section, we will use valuations to construct new actions of the groups $G=G(R,\gamma)$ on simplicial trees. The construction in this section works for any group $H\rtimes_\alpha \Z^n$ where $H$ is abelian, and  we work in this level of generality. The result about $G=G(R,\gamma)$ is then a special case of this construction.

We begin by defining a valuation on an abelian group $H$. 

\begin{defn}\label{def:valuation} A \emph{valuation} on an abelian group $H$ is a function $\overline{v}\colon H\to \R\cup \{\infty\}$ satisfying:
\begin{enumerate}[(a)]
\item $\overline{v}(0)=\infty$; and
\item $\overline{v}(x+y)\geq \min \left\{\overline{v}(x),\overline{v}(y)\right\}$ with equality unless $\overline{v}(x)=\overline{v}(y)$.
\end{enumerate}
\end{defn}

\noindent Note that (a) and (b) imply that $\overline{v}(-x)=\overline{v}(x)$ for all $x\in H$. This follows since if $\overline v(x)\neq \overline{v}(-x)$, then \[\overline{v}(0)=\min\{\overline{v}(x),\overline{v}(-x)\}<\infty,\] which contradicts (a). Thus, the function $v\colon H\times H\to \R\cup \{\infty\}$ defined by  $v(x,y)=\overline{v}(x-y)$ satisfies conditions (V1)--(V3) listed in the previous section. Let $T_{\overline{v}}= (H\times \R)/\sim$ be the tree defined by $v$ as in Definition~\ref{def:tree}.

Consider the group $H\rtimes_\alpha \Z^n$, where $H$ is abelian, and fix a homomorphism $\rho\colon \Z^n\to \R$.
\begin{defn}
 A valuation $\overline{v}$ on $H$ is \emph{subordinate to $\rho$} if $\overline{v}(\alpha(g)(x))=\overline{v}(x)+\rho(g)$ for all $x\in H$ and $g\in \Z^n$.
\end{defn}

Define an action of $H\rtimes_\alpha \Z^n$ on $H\times \R$ as follows: the group $H$ acts on $H\times \R$ by left translation on itself and trivially on the $\R$--factor, and $\Z^n$ acts on $H\times \R$ by $g\cdot (x,h)=(\alpha(g)(x), h+\rho(g))$.  For certain valuations, this descends to an action of $H\rtimes_\alpha \Z^n$ on the associated tree $T_{\overline{v}}$.

\begin{lem} If $\ol v$ is subordinate to $\rho$, then the action described above defines an action of $H\rtimes_\alpha \Z^n$ on $T_{\overline{v}}$.
\end{lem}

\begin{proof}  It is clear that  $H\curvearrowright H\times \R$ descends to an action of $H$ on the quotient $T_{\overline{v}}$.  The action of $\Z^n$ also descends to an action  $\Z^n\curvearrowright T_{\overline{v}}$, since $$ (x,h)\sim (y,h) \Leftrightarrow h\leq \overline{v}(x-y) \Leftrightarrow h+\rho(g)\leq \overline{v}(\alpha(g)(x)-\alpha(g)(y)) \Leftrightarrow (\alpha(g)(x),h+\rho(g))\sim (\alpha(g)(y),h+\rho(g)).$$ Moreover, these define an action of $H\rtimes_{\alpha} \Z^n$ on $T_{\overline{v}}$ since $\alpha(g)(x)\cdot (y,h)=(y+\alpha(g)(x),h)$, whereas \[gxg^{-1}\cdot(y,h)=gx\cdot (\alpha(g)^{-1}(y),h-\rho(g))=g\cdot (\alpha(g)^{-1}(y)+x,h-\rho(g))=(y+\alpha(g)(x),h).\qedhere\]
\end{proof}

In this paper, we are most interested in groups $G(R,\gamma)= \gamma^{-1}R\rtimes_\alpha \Z$, where $R$ and $\gamma$ satisfy  axioms (A1)--(A5). However, valuations in the more general setting above are also of interest. We give two examples.

\begin{ex}Consider  the group $(\Z/n\Z)\wr \Z^2=(\Z/n\Z)[x^{\pm 1},y^{\pm 1}]\rtimes_\alpha \Z^2$.
\begin{enumerate}[(i)]
 \item Fix a homomorphism $\rho\colon \Z^2\to \R$, and  define a valuation $\overline{v}$ on $(\Z/n\Z)[x^{\pm 1},y^{\pm 1}]$ subordinate to $\rho$ as follows. Let $\overline{v}(0)=\infty$, and if $p(x,y)$ is a non-zero Laurent polynomial in $(\Z/n\Z)[x^{\pm 1},y^{\pm 1}]$, then let \[\overline{v}(p)=\min \{\rho(k,l) : ax^ky^l \text{ appears as a term in } p \text{ for some } a\neq 0\}.\] Then $\overline{v}$ satisfies the definition of a valuation since if $p,q\in (\Z/n\Z)[x^{\pm 1},y^{\pm 1}]$ and $ax^ky^l$ and $bx^uy^v$ are terms of $p$ and $q$, respectively, such that $\overline{v}(p)=\rho(k,l)$ and $\overline{v}(q)=\rho(u,v)$, then every monomial $cx^iy^j$ in $p+q$ satisfies $\rho(i,j)\geq \min\{\rho(k,l),\rho(u,v)\}$. Moreover, if $\rho(k,l)\neq \rho(u,v)$, say $\rho(k,l)<\rho(u,v)$, then $ax^ky^l$ is a monomial of $p+q$ with $\rho(k,l)$ minimal, and therefore \[\overline{v}(p+q)=\overline{v}(p)=\min\{\overline{v}(p),\overline{v}(q)\}.\] The tree $T_{\overline{v}}$ will be simplicial exactly if the image of $\rho$ is discrete in $\R$.
\item We now describe another valuation on this group coming from a ring homomorphism. There is a ring homomorphism $(\Z/n\Z)[x^{\pm 1},y^{\pm 1}]\to (\Z/n\Z)[z^{\pm 1}]$ defined by $x\mapsto z$ and $y\mapsto z^{-1}$. This induces a homomorphism of groups $(\Z/n\Z)[x^{\pm 1},y^{\pm 1}]\rtimes_\alpha \Z^2 \to (\Z/n\Z)[z^{\pm 1}]\rtimes \Z$ defined by sending the generators of $\Z^2$ to $1$ and $-1$ in $\Z$, respectively. The valuation $\overline{w}$ on $(\Z/n\Z)[z^{\pm 1}]$ defined by \[\overline{w}(p(z))=\inf \{ k : z^k \text{ appears in } p \text{ with non-zero coefficient} \} \text{ and } \overline{w}(0)=\infty\] induces a valuation $\overline{v}$ on $(\Z/n\Z)[x^{\pm 1},y^{\pm 1}]$ by $\overline{v}(p(x,y))=\overline{w}(p(z,z^{-1}))$. The valuation $\overline{v}$ is subordinate to the homomorphism $\rho\colon\Z^2\to \R$ that sends the generators to 1 and $-1$, respectively. The resulting tree of $(\Z/n\Z)[x^{\pm 1},y^{\pm1 }]$ induced by $\overline{v}$ is isomorphic to the tree of $(\Z/n\Z)[z^{\pm 1}]$ induced by $\overline{w}$ and the action of $(\Z/n\Z)\wr \Z^2=(\Z/n\Z)[x^{\pm 1},y^{\pm 1}]\rtimes_\alpha \Z^2$ is obtained by pulling back the action of $(\Z/n\Z)[z^{\pm 1}]\rtimes \Z$.
\end{enumerate}
\end{ex}

Suppose $R$ satisfies the axioms (A1)--(A5), and consider the corresponding group $G=G(R,\gamma)=\gamma^{-1}R \rtimes_\alpha \Z$.  We can associate a valuation to any ideal of $\widehat R$.
\begin{lem}\label{lem:overlinev}
For any ideal $\frak a$ of $\widehat R$, the function $\overline{v}\colon \gamma^{-1}R\to \R\cup\{\infty\}$ defined by \[ \overline{v}(x)=\begin{cases}\infty & \text{ if } \gamma^k x\in \mathcal C(\frak a) \text{ for all } k \\ -\inf\{k\in \Z:\gamma^kx\in \mathcal C(\frak a)\} & \textrm{ else}\end{cases}\]
is a valuation.
 
\end{lem}

\begin{proof}
By Lemma \ref{lem:idealtosubring}, the set $\mathcal C(\frak a)$ is a subring of $\gamma^{-1}R$ containing $R$.
Thus  for every $x\in \gamma^{-1}R$,  we have $\gamma^k x\in \mathcal C(\frak a)$ for some $k$, and so $\overline v(x)$ has range in $\Z\cup \{\infty\}$. That $\ol v(0) = \infty$ also follows from the definition. To see that $\overline{v}$ is a valuation on $\gamma^{-1}R$, note that if $\gamma^kx, \gamma^ly\in \mathcal C(\frak a)$,  then $\gamma^{\max\{k,l\}}(x+y)\in \mathcal C(\frak a)$. Moreover, if $\ol v(x) = -k$ and $ \ol v(y) =-l$ with $k<l$,  then $\gamma^i(x+y)\notin \mathcal C(\frak a)$ for $k\leq i<l$, since $\gamma^ix\in \mathcal C(\frak a)$ whereas $\gamma^iy\notin \mathcal C(\frak a)$. Hence $\inf\{i:\gamma^i(x+y)\in \mathcal C(\frak a)\}=l$, and so $\ol v(x +y) = \op{min}\{\ol v(x), \ol v(y)\}$.
\end{proof}

The Bass-Serre trees of $G$  in Theorem \ref{thm:P+tree} can now be described explicitly using valuations.

\begin{prop}
\label{prop:isomorphictree}
Let $\frak a$ be an ideal of $\widehat R$. Let $\overline v$ be the valuation on $\gamma^{-1}R$ associated to $\frak a$ in Lemma \ref{lem:overlinev}. Then the Bass-Serre tree of $G$ as an HNN extension of $\mathcal C(\frak a)$ is $G$-equivariantly isomorphic to $T_{\overline v}$. Hence every element of $\mathcal P_+(G)$ is represented by an action on a tree $T_{\overline v}$.
\end{prop}

\begin{proof}
Fix as a basepoint $b= \pi (0,0)\in T_{\overline v}=(\gamma^{-1} R\times \R)/\sim$. An element $rt^k$ with $r\in \gamma^{-1}R$ fixes $b$ if only if $k=0$ and $\overline{v}(r)\geq 0$. Thus the stabilizer of $b$ is $\mathcal C(\frak a)$. Since $G$ acts transitively on $\gamma^{-1}R \times \Z$, there is a single orbit of vertices. The vertices adjacent to and directly above $b$ in $T_{\overline{v}}$ are exactly the equivalence classes of the pairs $(r,1)$ with $r\in \mathcal C(\frak a)$. Since $\mathcal C(\frak a)$ acts transitively on these vertices, $T_{\overline{v}}$ also has a single orbit of edges. The stabilizer of the vertex $(0,1)$ is exactly the subgroup of $\gamma^{-1} R$ of elements $r\in \gamma^{-1} R$ satisfying $\overline{v}(r)\geq 1$ or, equivalently, satisfying $\inf \{k : \gamma^k r \in \mathcal C(\frak a)\}\leq -1$. This is exactly $\mathcal \gamma C(\frak a)$. Thus, $T_{\overline{v}}$ is the Bass-Serre tree for $G$ corresponding to the expression of $G$ as an ascending HNN extension of $\mathcal C(\frak a)$, which is glued to itself via the endomorphism $\alpha$ (which is multiplication by $\gamma$).
\end{proof}

Before moving on to the next section, we use valuations to describe Bass-Serre tree representatives for hyperbolic structures of $(\Z/n\Z)\wr \Z$ and $BS(1,n)$ and draw  explicit pictures of these trees. The posets $\mathcal H((\Z/n\Z)\wr \Z)$ and $\mathcal H(BS(1,n))$ are described in Section \ref{sec:examples}.

\subsubsection{Bass-Serre trees for lamplighter groups}

Consider the lamplighter group $G=(\Z/n\Z)\wr \Z$, where $n\geq 2$. We use Theorem \ref{thm:P+tree} to describe trees representing the elements of $\mathcal P_+(G)$. As $\mathcal P_-(G)$ is isomorphic to $\mathcal P_+(G)$ in this case, the construction for $\mathcal P_-(G)$ is analogous.

 A saturated ideal $m(\Z/n\Z)[[x]]$ for $m\in \Z/n\Z$  corresponds to the hyperbolic structure $[\mathcal C(m(\Z/n\Z)[[x]])\cup \{t^{\pm 1}\}]$. The ring $\mathcal C(m(\Z/n\Z)[[x]])$ consists of exactly the Laurent polynomials for which every coefficient of a negative power of $x$ is in the subgroup $m(\Z/n\Z)$. By Proposition \ref{prop:isomorphictree}, this hyperbolic structure is represented by the tree $T_{\overline{v}}$, where \[\overline{v}(p(x))=-\inf\{k : x^k p(x)\in \mathcal C\left(m(\Z/n\Z)[[x]]\right)\}.\] In other words, $\overline{v}(p(x))$ is the infimum of all $i$ such that the coefficient of $x^i$ in $p(x)$ doesn't lie in $m(\Z/n\Z)$. The largest such hyperbolic structure is $[\mathcal C(0)\cup \{t^{\pm 1}\}]$, which corresponds to the valuation $\overline{v}_0(p(x))$, which measures the smallest $i$ such that $x^i$ appears in $p(x)$.

We may choose the generator $m$ of $\langle m \rangle$ to divide $n$. There is an obvious quotient $\Z/n\Z\to \Z/m\Z$ which induces a quotient $\pi\colon (\Z/n\Z)\wr \Z\to (\Z/m\Z)\wr \Z$. We see that the valuation $\overline{v}$ induced by $\mathcal C(m(\Z/n\Z)[[x]])$ is the pullback of the main valuation $\overline{v}_0$ on $(\Z/m\Z)[x^{\pm 1}]$. That is, $\overline{v}(p(x))=\overline{v}_0(\pi(p(x)))$. Thus, the tree $T_{\overline{v}}$ is isomorphic to the tree $T_{\overline{v}_0}$ for $(\Z/m\Z)\wr \Z$ and the action is induced by the quotient homomorphism. Hence all of the trees representing the hyperbolic structures of $(\Z/n\Z)\wr \Z$ are obtained by pulling back the actions of quotients $(\Z/m\Z)\wr \Z$ on their main Bass-Serre trees. So in order to draw pictures of these actions it suffices to draw  the main Bass-Serre tree of $(\Z/m\Z)\wr \Z$ corresponding to the valuation, $\overline{v}_0$.

\begin{ex}
The main Bass-Serre tree for $(\Z/2\Z)\wr \Z$ is pictured in Figure \ref{fig:lamplightertree}. The main Bass-Serre tree for $(\Z/n\Z)\wr \Z$ is analogous, being an action on an $(n+1)$-regular tree with vertices labeled by equivalence classes of Laurent polynomials in $(\Z/n\Z)[x^{\pm 1}]$.
\end{ex}

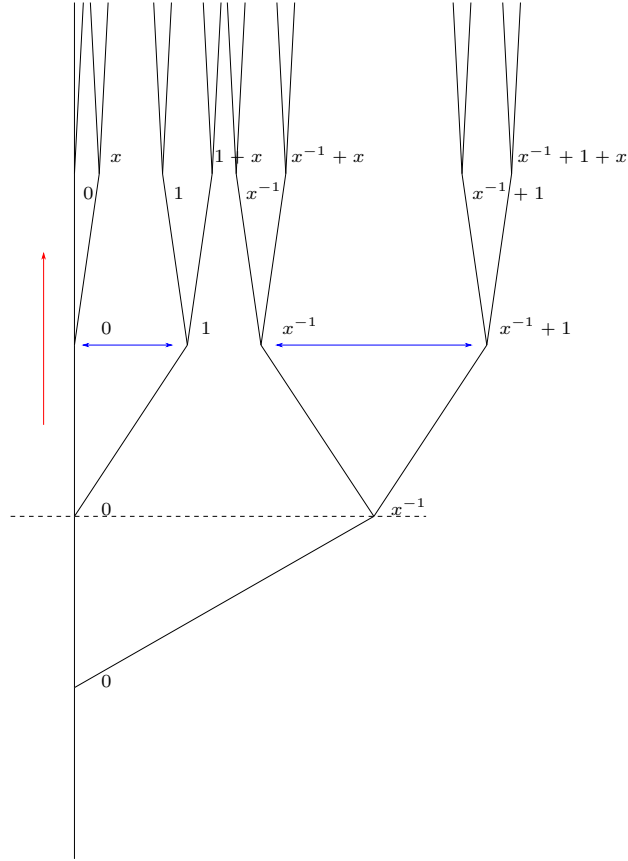
\begin{figure}[h]

\begin{center}
\begin{overpic}[width=0.43\textwidth,percent]{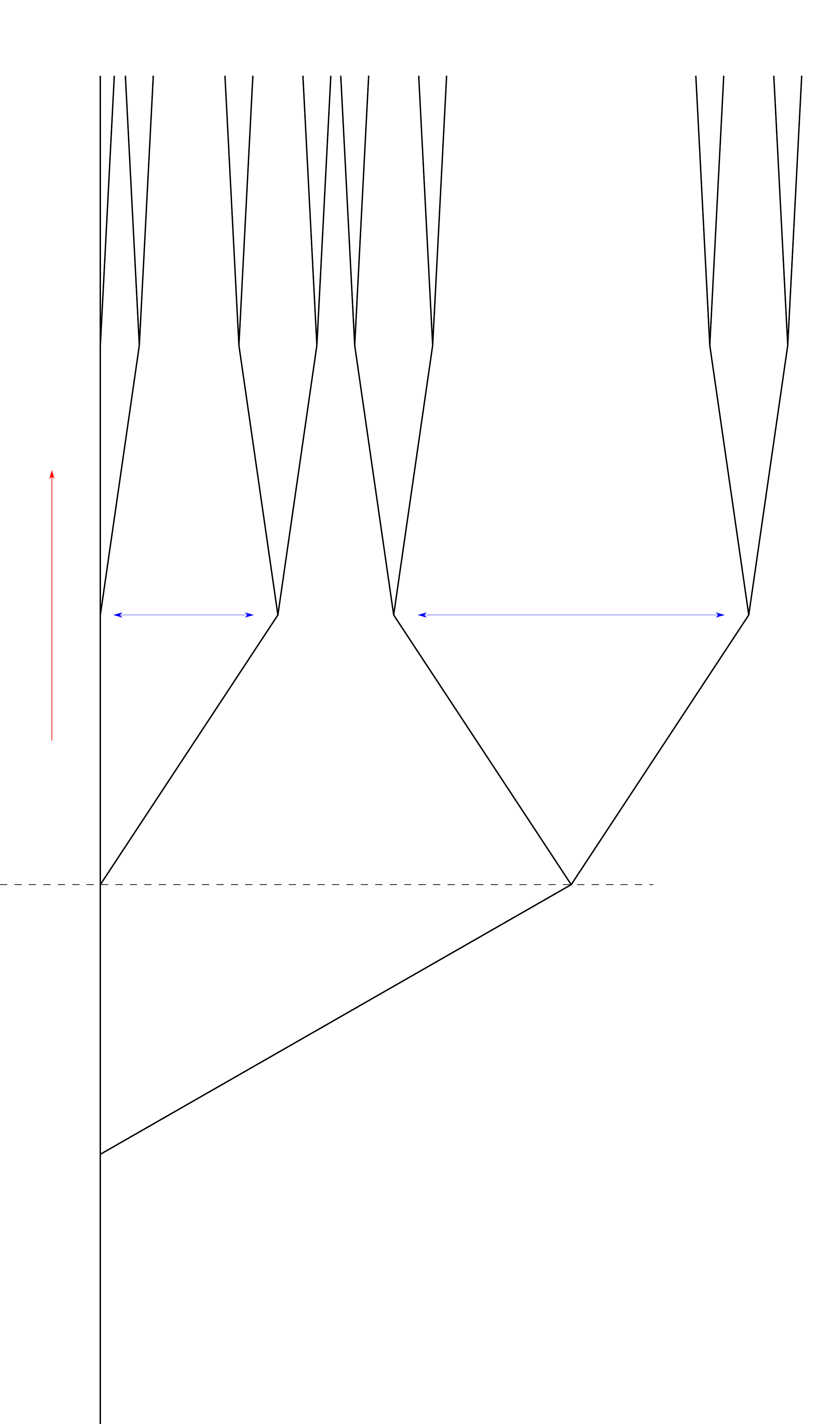}

\put (10,19) {\scriptsize $0$}
\put (10,38) {\scriptsize $0$}
\put (42,38) {\scriptsize $x^{-1}$}
\put (10,58) {\scriptsize $0$}
\put (21,58) {\scriptsize $1$}
\put (30,58) {\scriptsize $x^{-1}$}
\put (54,58) {\scriptsize $x^{-1}+1$}
\put (8,73) {\scriptsize $0$}
\put (11,77) {\scriptsize $x$}
\put (18,73) {\scriptsize $1$}
\put (22.5,77) {\scriptsize $1+x$}
\put (26,73) {\scriptsize $x^{-1}$}
\put (31,77) {\scriptsize $x^{-1}+x$}
\put (51,73) {\scriptsize $x^{-1}+1$}
\put (56,77) {\scriptsize $x^{-1}+1+x$}
\end{overpic}
\end{center}

\caption{The main Bass-Serre tree for $(\Z/2\Z)\wr \Z$. The vertices are represented by pairs in $(\Z/2\Z)[x^{\pm 1}]\times \Z$ where $(p(x),h)\sim (q(x),h)$ if $\overline{v}_0(p-q)\geq h$. Here heights are implicit, with height 0 being indicated by a horizontal dotted line. The generator $t$ of $\Z$ acts loxodromically with axis indicated by the red arrow. It has the effect of shifting $(p(x),h)$ ``vertically upward'' to $(p(x),h+1)$. The (order 2) action of the unit 1 of the ring $(\Z/2\Z)[x^{\pm 1}]$ is pictured in blue. It has the effect of interchanging $(p(x),h)$ with $(p(x)+1,h)$.}
\label{fig:lamplightertree}

\end{figure}

\subsubsection{Bass-Serre trees for Baumslag-Solitar groups}

Consider the Baumslag-Solitar group $BS(1,n)\cong \Z[\frac{1}{n}]\rtimes \Z$. For a divisor $m$ of $n$, $\Z[\frac{1}{m}]$ is a subring of $\Z[\frac{1}{n}]$, and two such divisors define the same subring if they have the same prime divisors. For any divisor $m$ of $n$, $BS(1,n)$ is an ascending HNN extension of $\Z[\frac{1}{m}]$. By Proposition \ref{prop:treeschwarzmilnor} these Bass-Serre trees represent all the hyperbolic structures in $\mc P_+(G)$.

For another perspective, if $n=(-1)^\delta p_1^{k_1} \cdots p_r^{k_r}$ is the prime factorization of $n$ then we may choose numbers $\epsilon_i \in \{0,1\}$ for each $1\leq i\leq r$ and consider the divisor $m=(-1)^\delta p_1^{\epsilon_1 k_1} \cdots p_r^{\epsilon_r k_r}$. The Bass-Serre tree corresponding to the expression of $BS(1,n)$ as an ascending HNN extension of $\Z[\frac{1}{m}]$ is the tree $T_{\overline{v}}$ associated to the $(n/m)$-adic valuation $\overline{v}$ on $\Z[\frac1n]$. Two of these trees are pictured below. See \cite{AR} for more details.

\begin{ex}
The group $BS(1,2)$ acts on its main Bass-Serre tree (corresponding to the expression of $BS(1,2)$ as an ascending HNN extension of $\Z$). This is the tree defined by the standard 2-adic valuation on $\Z[\frac{1}{2}]$. The tree is pictured on the left hand side of Figure \ref{fig:bstrees}.
\end{ex}

\begin{figure}[h]

\begin{center}
\begin{tabular}{l l}

\begin{overpic}[width=0.43\textwidth,percent]{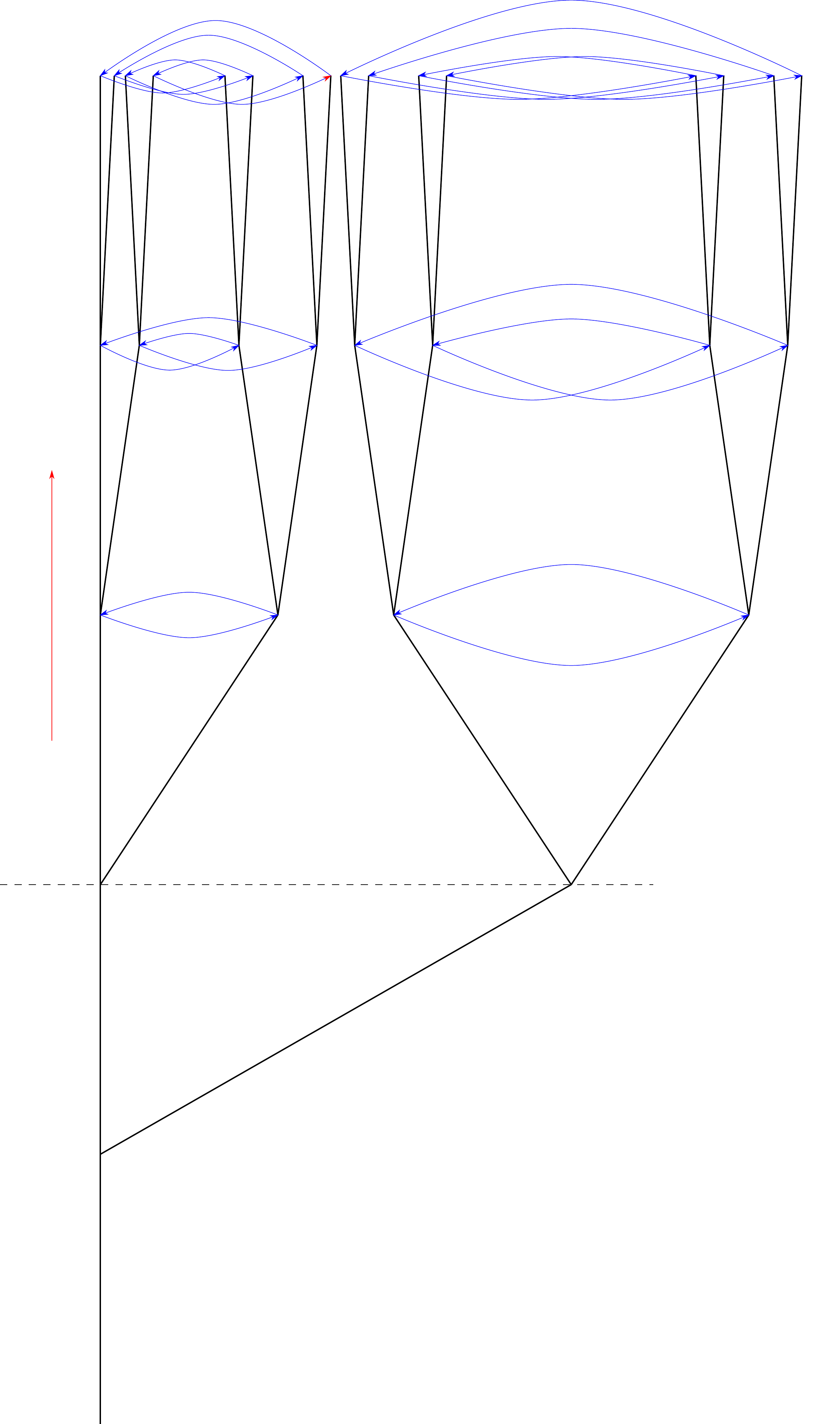}
\put (1,60) {\color{red} $t$}
\put (40,62) {\color{blue} $a$}

\put (10,19) {\scriptsize $0$}
\put (10,38) {\scriptsize $0$}
\put (42,38) {\scriptsize $0.1$}
\put (10,57) {\scriptsize $0$}
\put (21,57) {\scriptsize $1$}
\put (30,57) {\scriptsize $0.1$}
\put (54,57) {\scriptsize $1.1$}
\put (8,75) {\scriptsize $0$}
\put (11,75) {\scriptsize $10$}
\put (18,75) {\scriptsize $1$}
\put (22.5,75) {\scriptsize $11$}
\put (26,75) {\scriptsize $0.1$}
\put (31,75) {\scriptsize $10.1$}
\put (51,75) {\scriptsize $1.1$}
\put (56,75) {\scriptsize $11.1$}
\put (7,97) {\tiny $0$}
\put (8,95.5) {\tiny $100$}
\put (9,97) {\tiny $10$}
\put (11,95.5) {\tiny $110$}
\put (15,97) {\tiny $1$}
\put (17,95.5) {\tiny $101$}
\put (21,97) {\tiny $11$}
\put (22.5,95.5) {\tiny $111$}
\put (24,97) {\tiny $0.1$}
\put (26,95.5) {\tiny $100.1$}
\put (28.5,97) {\tiny $10.1$}
\put (31.5,95.5) {\tiny $110.1$}
\put (48,97) {\tiny $1.1$}
\put (50.5,95.5) {\tiny $101.1$}
\put (54,97) {\tiny $11.1$}
\put (56,95.5) {\tiny $111.1$}
\end{overpic}

&

\begin{overpic}[width=0.43\textwidth,percent]{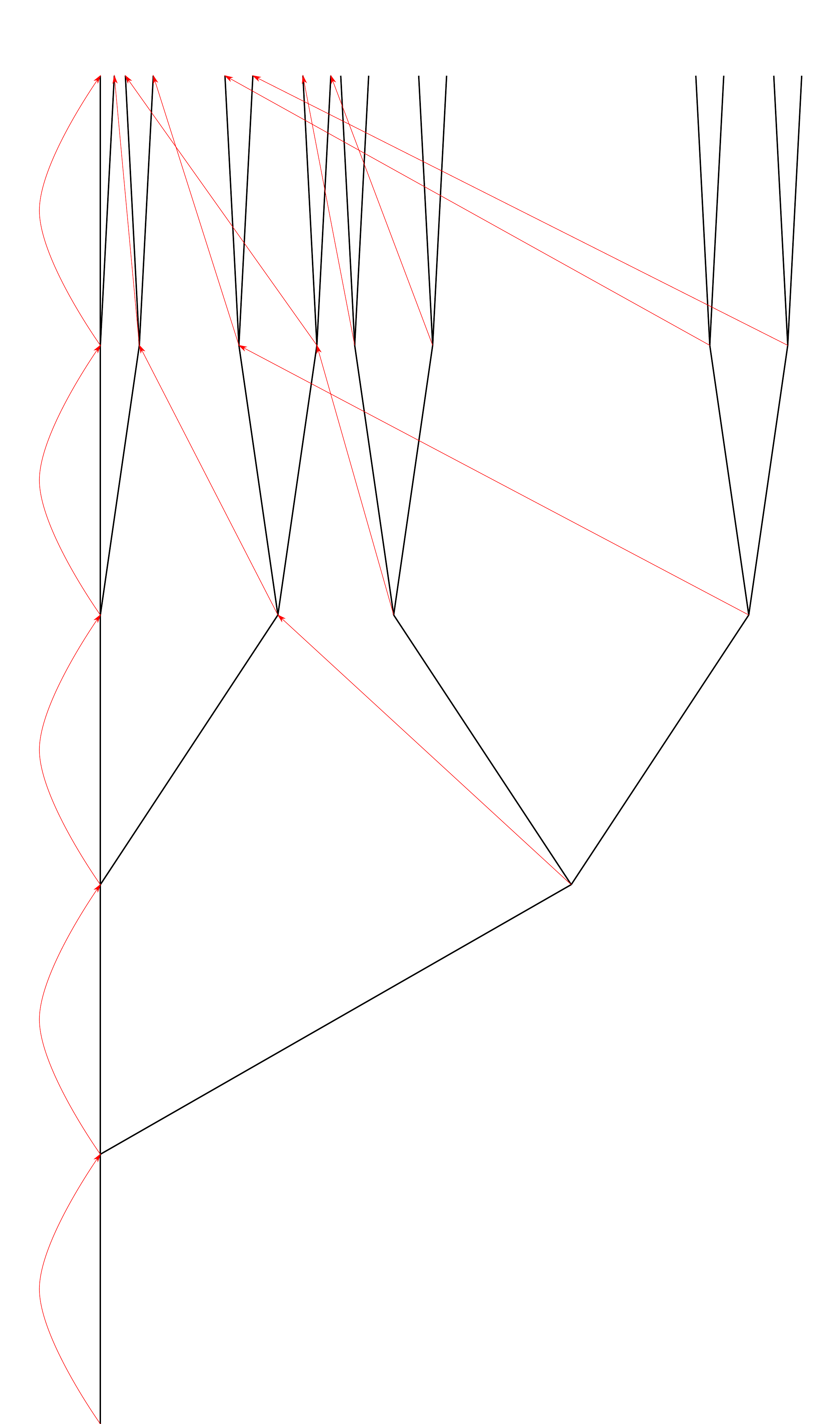}
\put (1,60) {\color{red} $t$}

\put (10,19) {\scriptsize $0$}
\put (10,38) {\scriptsize $0$}
\put (42,38) {\scriptsize $0.1$}
\put (10,57) {\scriptsize $0$}
\put (21,57) {\scriptsize $1$}
\put (30,57) {\scriptsize $0.1$}
\put (54,57) {\scriptsize $1.1$}
\put (8,75) {\scriptsize $0$}
\put (11,75) {\scriptsize $10$}
\put (18,75) {\scriptsize $1$}
\put (22.5,75) {\scriptsize $11$}
\put (26,75) {\scriptsize $0.1$}
\put (31,75) {\scriptsize $10.1$}
\put (51,75) {\scriptsize $1.1$}
\put (56,75) {\scriptsize $11.1$}
\put (7,97) {\tiny $0$}
\put (8,95.5) {\tiny $100$}
\put (9,97) {\tiny $10$}
\put (11,95.5) {\tiny $110$}
\put (15,97) {\tiny $1$}
\put (17,95.5) {\tiny $101$}
\put (21,97) {\tiny $11$}
\put (22.5,95.5) {\tiny $111$}
\put (24,97) {\tiny $0.1$}
\put (26,95.5) {\tiny $100.1$}
\put (28.5,97) {\tiny $10.1$}
\put (31.5,95.5) {\tiny $110.1$}
\put (48,97) {\tiny $1.1$}
\put (50.5,95.5) {\tiny $101.1$}
\put (54,97) {\tiny $11.1$}
\put (56,95.5) {\tiny $111.1$}
\end{overpic}

\end{tabular}
\end{center}

\caption{The action $BS(1,2)$ on its main Bass-Serre tree (left). The action $BS(1,6)$ on its Bass-Serre tree as an ascending HNN extension of $\Z[\frac{1}{3}]$ (right). The element $a$ acts as the same isometry in either action.}
\label{fig:bstrees}
\end{figure}
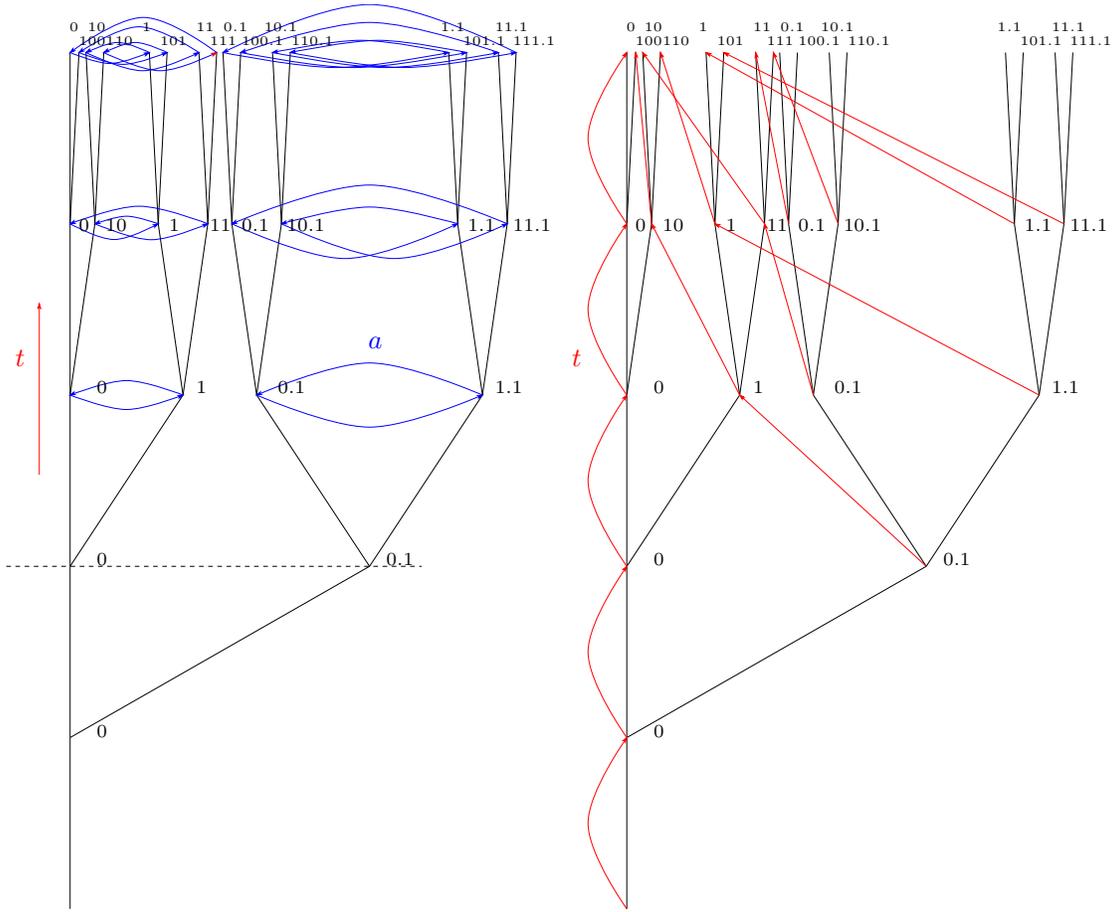

\begin{ex}
The group $BS(1,6)$ acts on its Bass-Serre tree as an ascending HNN extension of $\Z[\frac{1}{3}]$. This is the tree defined by the $2$-adic valuation on $\Z[\frac{1}{6}]$. It is pictured on the right of Figure \ref{fig:bstrees}. The labels correspond to equivalence classes of elements of $\Z[\frac{1}{6}]$ written in base 2.
\end{ex}

\section{Preliminaries on torsion-free finitely presented abelian-by-cyclic groups}
\label{section:preliminaries}

The rest of the paper is devoted to proving Theorems \ref{thm:char=min} and \ref{thm:P-description}. In this section, we consider a general class of abelian-by-cyclic groups and lay the groundwork to classify their hyperbolic actions. Let $G$ be an ascending HNN extension  of $\Z^n$, so that there is an injective endomorphism of $\Z^n$, represented by an $n\times n$ integral matrix $\gamma\in M_n(\Z)$ with non-zero determinant, such that \[G=\langle \Z^n,t:tzt^{-1}=\gamma z \text{ for } z\in \Z^n\rangle.\] We will sometimes denote $G$ by $G(\gamma)$ to emphasize the matrix $\gamma$. Let $p(x)$ be the  characteristic polynomial of $\gamma$, which has degree $n$. There is an associated ring $R=\Z[x]/(p)$, where $(p)$ is the ideal generated by $p$. In this section, we show that $R$ satisfies axioms (A1)--(A5) (Lemmas~\ref{lem:A1A4} and \ref{lem:A5}) and describe some properties of $R$ and $G$.  We then apply the machinery developed in the previous sections of the paper to $G(R,\gamma)$, which, along with the results of Appendix~\ref{sec:appendix}, allows us to characterize the poset $\mc P_+(G)$ (Theorem~\ref{thm:Thm1.2P_+}). In Sections \ref{sec:pminusasinvspaces} and \ref{sec:heintze} we prove Theorem \ref{thm:P-description}, which will describe $\mathcal P_-(G)$.

We first need to study the abelianization of $G(\gamma)$, in order to apply Proposition \ref{prop:generalstructure}.

\begin{lem}
\label{lem:abbycyclicabelianization}
Let $\gamma\in M_n(\Z)$ with non-zero determinant, and set $G=G(\gamma)$. If neither $1$ nor $-1$ is an eigenvalue of $\gamma$, then the abelianization of $G$ is virtually cyclic. The same is true of the finite index subgroup of $G$ generated by $\Z^n$ and $t^2$.
\end{lem}

\begin{proof}
Since $G=\langle \Z^n,t: tzt^{-1}=\gamma z \text{ for all } z\in \Z^n\rangle$, the abelianization of $G$ is $H\times \Z$ where $H$ is the image of $\Z^n$. Thus $H$ is the quotient of $\Z^n$ by the relations $z=\gamma z$ for all $z\in \Z^n$. Let $e_i=(0,\ldots,0,1,0,\ldots,0)$ be the standard generators for $\Z^n$ for $1 \leq i \leq n$. The group $H$ is generated by the images $f_1,\ldots,f_n$ of $e_1,\ldots,e_n$ subject to the relations $[f_i,f_j]=1$ for all $i\neq j$ and  \[(\gamma-I)^T\begin{pmatrix} f_1 \\ \vdots \\ f_n\end{pmatrix}=\begin{pmatrix} 0 \\ \vdots \\ 0\end{pmatrix}\] where $\cdot^T$ denotes the transpose. We claim that each $f_i$ has finite order. This will prove that $H$ is a finite abelian group, as desired.

Since $1$ is not an eigenvalue of $\gamma$, the matrix $(\gamma-I)^T$ is  invertible. We may thus row reduce $(\gamma-I)^T$ (scaling only by integers and adding only integer multiples of one row to another) to a diagonal matrix with non-zero integers $k_1,\ldots,k_n$ on the diagonal. Thus we have $k_if_i=0$ for each $i$, and $f_i$ is finite order.

Finally, we consider the finite index subgroup generated by $\Z^n$ and $t^2$. This is simply isomorphic to $G(\gamma^2)$. Since $\pm 1$ is not an eigenvalue of $\gamma$, 1 is not an eigenvalue of $\gamma^2$. Hence the discussion above with $\gamma$ replaced by $\gamma^2$ proves that the abelianization of $G(\gamma^2)$ is also virtually cyclic (note that until this paragraph we only used that 1 is not an eigenvalue of $\gamma$).
\end{proof}

Now consider the ring $R=\Z[x]/(p)$ where $p$ is the characteristic polynomial of $\gamma$.  We will also denote $R=\Z[x]/(p)=\Z[\gamma]$ to emphasize the matrix $\gamma$. 

We pause  to introduce a convention that will be used for the rest of the paper.  The abelian group $\Z^n$ has the structure of an $R$-module where $x+(p)$ acts on $\Z^n$ by multiplication by the $n\times n$ integer matrix $\gamma$. This also induces on $\Z^n$ the structure of a $\Z[x]$-module where $x$ acts on $\Z^n$ by multiplication by $\gamma$. Our convention is introduced to avoid confusion, since $\Z^n$  admits multiple representations as an $R$-module:

\begin{conv}
Whenever $\Z^n$ is referred to \emph{as an $R$-module}, the $R$-module structure is determined by letting $x+(p)$ act on $\Z^n$ by multiplication by $\gamma$. Whenever $\Z^n$ is referred to \emph{as a $\Z[x]$-module}, the $\Z[x]$-module structure is obtained by pulling back the $R$-module structure.
\end{conv}

\begin{lem}
\label{lem:abbycyclictoring}
Let $\gamma$ be an injective endomorphism of $\Z^n$ whose characteristic polynomial is equal to its minimal polynomial.  The ring $\Z[\gamma]$, considered as an abelian group with addition, is free of rank $n$. If $\Z^n$ is a cyclic $\Z[\gamma]$-module, then $\Z[\gamma]$ is isomorphic to $\Z^n$ as a $\Z[\gamma]$-module and $G(\gamma)$ is isomorphic to the ascending HNN extension of $\Z[\gamma]$ defined by the endomorphism $\gamma$.
\end{lem}

\begin{proof}
Since the characteristic polynomical $p$ is monic, for any $q\in \Z[x]$, we may write $q=ap+r$ where $a,r\in \Z[x]$ and $\deg(r)<\deg(p)=n$. Hence, $q$ and $r$ define the same equivalence class in $R=\Z[x]/(p)$. Moreover, the quotient homomorphism $\Z[x]\to R$ is injective on polynomials of degree $< n$ since every polynomial in $(p)$, except for $0$, has degree $\geq n$. Thus, the abelian subgroup of $\Z[x]$ generated by $1,x,\ldots,x^{n-1}$ is free and maps isomorphically onto $R$ under the quotient homomorphism.

Suppose that $\Z^n$ is a cyclic $\Z[\gamma]$-module. Then there is a vector $v\in \Z^n$ such that $v,\gamma v, \ldots,\gamma^{n-1}v$ generates $\Z^n$ and so is a basis for $\Z^n$ as a free abelian group. The unique homomorphism of $\Z[\gamma]$-modules $\Z[\gamma]\to \Z^n$ defined by $1\mapsto v$ is a bijection and hence an isomorphism of $\Z[\gamma]$-modules. Thus the ascending HNN extensions of $\Z[\gamma]$ and $\Z^n$ defined by multiplication by $\gamma$ are isomorphic.
\end{proof}

A particular case of cyclic $\Z[\gamma]$--modules is furnished by considering companion matrices to polynomials. If $ p(x)=a_0+a_1x+\cdots+a_{n-1}x^{n-1}+x^n\in \Z[x]$ is a monic polynomial, then its companion matrix is \[\gamma =\begin{pmatrix} 0 & 0 & \ldots &0 & -a_0 \\ 1 & 0 & \ldots & 0 & -a_1 \\ 0 & 1 & \ldots & 0 & -a_2 \\ \vdots & \vdots & \ddots & \vdots & \vdots \\  0 & 0 & \ldots & 1 & -a_{n-1}\end{pmatrix}.\] The $\Z[\gamma]$-module $\Z^n$ is cyclic (the vector $(1,0,\ldots,0)$ being a generator). Thus $G(\gamma)$ is  isomorphic to the ascending HNN extension of $\Z[\gamma]$ defined by multiplication by $\gamma$.

\subsection{Ring axioms for abelian-by-cyclic groups}

Let $\gamma$ be an admissible matrix. We now turn our attention to establishing the axioms (A1)--(A5) described in Section~\ref{sec:axioms} for the ring $R=\Z[\gamma]$. Recall that $p(x)\in \Z[x]$ is the characteristic polynomial of $\gamma$ and  $R=\Z[x]/(p)=\Z[\gamma]$. We will consider the ascending HNN extension $G(\gamma)$.

\begin{lem}
\label{lem:A1A4}
The ring $R=\Z[\gamma]$ and the element $\gamma=x+(p)$ satisfy (A1)--(A4). \end{lem}

\begin{proof}
First, $\gamma$ is not a zero divisor since $x$ does not divide $p$ in $\Z[x]$. It is also not a unit, for elements of $\Z[\gamma]$ act on $\Z^n$ as integer matrices, and there is no integer matrix with determinant equal to $1/\det(\gamma)$. The ring $R$ is generated by $\gamma$ as a $\Z$-algebra by definition. Since multiplication by $\gamma$ defines an injective endomorphism of $\Z[\gamma]$, which is isomorphic to $\Z^n$ as an abelian group, we have that $\gamma\Z[\gamma]$ is a sub-module which is also isomorphic to $\Z^n$ as an abelian group. Hence it is finite index. These remarks verify axioms (A1)--(A3).

Finally we verify (A4) using the isomorphism $\Z[\gamma]\to \Z^n$. Embed $\Z^n$ in $\C^n$ and endow $\C^n$ with the Euclidean $L^2$ norm $\|\cdot\|$. By Gelfand's formula, since the eigenvalues of $\gamma^{-1}$ are $<1$ in absolute value, the $L^2$ operator norm $\|\gamma^{-i}\|\to 0$ as $i\to \infty$. If the intersection $\bigcap_{i=0}^\infty (\gamma^i)$ were non-zero, there would be a vector $v\in \bigcap_{i=0}^\infty \gamma^i \Z^n$. In other words, $\gamma^{-i}v\in \Z^n$ for all $i\geq 0$. But $\|\gamma^{-i}v\|\to 0$ as $i\to \infty$, whereas there is a lower bound on the $L^2$ norm of any element of $\Z^n$. Therefore no such $v$ can exist and this completes the proof.
\end{proof}

\begin{lem}
\label{lem:constanttransversal}
If  $d\in \Z_{>0}$ is the absolute value of the constant term of $p(x)$, then $[d]$ is a transversal for $(\gamma)$ in $R=\Z[\gamma]$.
\end{lem}

\begin{proof}
There is a short exact sequence \[0\to ((p)+(x))/(x)\to \Z[x]/(x)\to R/(\gamma)\to 0.\] Elements of $(p)+(x)$ are equivalent modulo $(x)$ to a multiple of the constant term of $p$. Thus the order of $R/(\gamma)$ is exactly $d$. 
\end{proof}

For the last axiom (A5) we use an argument of Vince (\cite[Lemma 2]{vince}).
\begin{lem}
\label{lem:A5}
The pair $\Z[\gamma]$ and $\gamma$ satisfy axiom (A5) with respect to the transversal $[d]$.
\end{lem}

\begin{proof}
Via the isomorphism $\Z[\gamma]\to \Z^n$, the transversal $[d]=\{0,1,\ldots,d-1\}$ is identified with a transversal $\mathscr D=\{v_1,\ldots,v_d\}$ for $\gamma \Z^n$ in $\Z^n$. Each $x\in \Z^n$ has a uniquely determined $\gamma$-adic address $a_0+\gamma a_1+\cdots$, where $a_0,a_1,a_2,\ldots\in\mathscr D$. It is useful to review where the $a_i$ come from. We set $x_0=x$ and $a_0\in \mathscr D$ to be the unique element for which $x_0$ and $a_0$ lie in the same coset of $\gamma \Z^n$. We have $x=a_0+\gamma x_1$ for some unique $x_1\in \Z^n$. Solving for $x_1$ using the matrix $\gamma$, we have $x_1=\gamma^{-1}(x_0-a_0)$. Having chosen $a_0,\ldots,a_{i-1}$ and $x_0,\ldots,x_i$ such that $x_j=a_j+\gamma x_{j+1}$ for $j\leq i-1$, we have \[x=x_0=a_0+\gamma x_1=a_0+\gamma a_1 + \gamma^2 x_2=\ldots = a_0+\gamma a_1 +\cdots +\gamma^{i-1} a_{i-1} + \gamma^i x_i.\] Then $a_i\in \mathscr D$ is chosen to be the unique element with $x_i-a_i \in \gamma \Z^n$. We have $x_i=a_i+\gamma x_{i+1}$ for some unique $x_{i+1}\in \Z^n$; namely, $x_{i+1}=\gamma^{-1}(x_i-a_i)$.

Note that if $x_i=x_j$ for some distinct $i<j$, then the digits $a_i$ will repeat with period $j-i$ starting with $a_i$. 
As before, we embed $\Z^n$ in $\C^n$ and equip it with the $L^2$ norm $\| \cdot \|$. The proof will be complete if we can show that given $x$, $\|x_i\|$ is uniformly bounded for all $i$.

Choose $a<1$ with $|\lambda |<a<1$ for all eigenvalues $\lambda$ of $\gamma^{-1}$, which is possible since $\gamma$ is expanding. By Gelfand's formula, since the eigenvalues of $\frac{1}{a}\gamma^{-1}$ are $<1$ in absolute value, the $L^2$ operator norm $\|\left(\frac{1}{a}\gamma^{-1}\right)^i\|\to 0$ as $i \to \infty$.  It follows that $\|\gamma^{-i}\|<a^i$ for all sufficiently large $i$. Therefore there exists some $c>0$ such that $\|\gamma^{-i}\|<ca^i$ for \emph{every} $i\geq 0$, and so $\|\gamma^{-i} x\|\leq \|\gamma^{-i} \| \|x\|\leq ca^i \|x\|$ for each $i$ and all $x \in \Z^n$. Note that \[x_i=\gamma^{-i}x_0 - \sum_{j=0}^{i-1} \gamma^{j-i} a_j\] for each $i$. Thus, \[\|x_i\|\leq \|\gamma^{-i}x_0\| + \sum_{j=0}^{i-1} \|\gamma^{j-i}a_j\| \leq ca^{i} \|x_0\| + \sum_{j=0}^{i-1} ca^{i-j}\|a_j\|.\] Choosing $b$ larger than the norm of every element of $\mathscr D$, the above expression is at most \[ca^i\|x_0\| + cb(a+a^2+\cdots+a^i)= ca^i \|x_0\| + cb\frac{a-a^{i+1}}{1-a} <   c\|x_0\|+cb \frac{a}{1-a},\] which yields our upper bound, independent of $i$.
\end{proof}

\subsection{The structure of $\mc P_+(G)$ for abelian-by-cyclic groups}
In this section, we prove the portion of Theorem~\ref{thm:char=min} involving $\mc P_+(G)$. The proof of this theorem involves technical arguments from commutative algebra, which are made in Appendix \ref{sec:appendix}.  Here, we show how to use Theorem~\ref{thm:idealclassification} to describe the structure of $\mc P_+(G)$.

\begin{thm}
\label{thm:Thm1.2P_+}
Fix an expanding matrix $\gamma\in M_n(\Z)$, and  
suppose that $\Z^n$ is a cyclic $\Z[x]$-module. If the prime factorization in $\Z[[x]]$ of the characteristic polynomial $p$ of $\gamma$  is $p=up_1^{n_1}\cdots p_r^{n_r}$, then $\mathcal P_+(G)$ is isomorphic to $\Div(n_1,\ldots,n_r)$.
\end{thm}

\begin{proof}
By Lemma \ref{lem:abbycyclictoring}, $\Z^n$ is actually isomorphic to $R$ as an $R$-module and $G(\gamma)$ is isomorphic to $G(R,x+(p))$. By Lemmas \ref{lem:A1A4} and \ref{lem:A5} the pair $(R,x+(p))$ satisfies axioms (A1)--(A5). Hence by Theorem \ref{thm:main}, $\mathcal P_+(G)$ is isomorphic to the opposite of the poset of ideals of the $(x+(p))$-adic completion $\widehat R$ considered up to multiplication by $x+(p)$. Since $\gamma$ is expanding, the constant term of $p$ is not in $\{-1,0,1\}$, and so it follows from Theorem \ref{thm:idealclassification}  that this poset of ideals is isomorphic to the poset of divisors of $p$ in $\Z[[x]]$. If the prime factorization of $p$ in $\Z[[x]]$ is $up_1^{n_1}\cdots p_r^{n_r}$, then the poset of divisors is exactly $\Div(n_1,\ldots,n_r)$, which is isomorphic to its opposite. Hence $\mathcal P_+(G)$ is isomorphic to $\Div(n_1,\ldots,n_r)$.
\end{proof}

\section{The structure of $\mc P_-(G$): invariant subspaces}\label{sec:pminusasinvspaces}

The goal of this section is to prove the conclusions about $\mc P_-(G)$ from Theorem \ref{thm:P-description} involving invariant subspaces. Consider an expanding matrix $\gamma \in M_n(\Z)$  with equal minimal and characteristic polynomial and the group $G(\gamma)$. Notably, we do not assume here that $\Z^n$ is a cyclic $\Z[\gamma]$-module. 

\begin{prop}\label{prop:P-invar}
If $G=G(\gamma)$, then $\mathcal P_-(G)$ is isomorphic to the poset of  subspaces of $\R^n$ invariant under~$\gamma$.
\end{prop}

Recall that $G=G(\gamma)=\langle \Z^n, t : tzt^{-1}=\gamma z \text{ for } z\in \Z^n\rangle$. As for any ascending HNN extension, we may express $G$ as $N\rtimes_\alpha \Z$ where $N=\bigcup_{i=0}^\infty t^{-i}\Z^n t^i$ is the normal closure of $\Z^n$ and $\alpha$ is conjugation by $t$.

To prove Proposition~\ref{prop:P-invar}, we first use Proposition~\ref{prop:Ppmconfining} to identify elements of $\mc P_-(G)$ with subsets of $N$ which are confining under $\alpha^{-1}$.  To such a confining subset, we show how to associate a subspace of $\R^n$ which is invariant under $\gamma$ (see Lemma~\ref{lem:conftoinvar}).  Conversely, given an invariant subspace of $\R^n$, we show how to associate a confining subset of $N$ (see Lemma~\ref{lem:invartoconf}).  To prove the isomorphism of posets, we will show that these operations are ``almost inverses" of each other, in the sense that any confining subset $Q$ of $N$ is equivalent to the confining subset formed by first considering the invariant subspace associated to $Q$, and then taking the confining subset associated to that subspace.  We do this in two steps.  The first is in Section~\ref{sec:tiles}, where we introduce the notion of \emph{tiles}, and the second is in Section~\ref{sec:invarsubspaces}, where we investigate the structure of invariant subspaces.

We begin by showing that we can identify $N$ with a subset of $\R^n$.  
We consider $\Z^n$ with its canonical $\Z[\gamma]$-module structure and embed it as a subset of $\R^n$.

\begin{defn} \label{def:A}
Let $A$ be the $\Z[\gamma]$-module
\[A\vcentcolon=\bigcup_{i\geq 0} \gamma^{-i} \Z^n.\]
\end{defn}
Thus $A$ is a sub-module of $\Q^n \subset \R^n$. The normal subgroup $N\trianglelefteq G$ is also naturally a $\Z[\gamma]$-module with $\gamma$ acting by conjugation by $t$. There is a map $N\to A$ defined by $t^{-i} z t^i \mapsto \gamma^{-i}z$. One may check that this is well-defined and an isomorphism of $\Z[\gamma]$-modules. We will  identify $N$ with the subset $A\subset \R^n$ in all that follows. The automorphism $\alpha^{-1}$ of $A$ induced by conjugation by $t^{-1}$ is multiplication by $\gamma^{-1}$.

We now define another subset of $\R^n$ associated to a subset $Q$ of $A$ that is confining under $\alpha^{-1}$. It may be thought of as the limit set of a confining subset under iteration of $\alpha^{-1}$. We will show that it is in fact a subspace of $\R^n$ invariant under $\gamma$.

\begin{defn}
Let $Q\subseteq A$ be confining under $\alpha^{-1}$.  Define a subset $\mc L(Q)$ of $\mathbb R^n$  consisting of the set of all limits of convergent sequences $\gamma^{-k}x_k$ where $x_k\in Q$ for all $k\geq 0$. That is, \[\mathcal L(Q)\vcentcolon=\left\{\lim_{k\to \infty} \gamma^{-k}x_k : \{x_k\}_{k=0}^\infty \subset Q \text{ and } \{\gamma^{-k} x_k\}_{k=0}^\infty \text{ is convergent}\right\}.\]
\end{defn}

The following lemma gives an alternate characterization of $\mc L(Q)$.
\begin{lem}
\label{lem:limsetequivdefn}
The set $\mathcal L(Q)$ is an intersection of closures
\[
\mathcal{L}(Q)=\bigcap_{k\geq 0}\overline{\gamma^{-k}(Q)}.
\]
\end{lem}

\begin{proof}
For any $y\in \mathcal{L}(Q)$, there exist $x_k\in Q$ such that $\gamma^{-k} x_k\to y$. Since $\gamma^{-k}(Q)$ is closed under applying $\gamma^{-1}$, for any $l\geq k$ we have  $\gamma^{-l} x_l= \gamma^{-k} (\gamma^{-l+k} x_l) \in \gamma^{-k}(Q)$. Thus $\{\gamma^{-l} x_l\}_{l\geq k}$ is a sequence in $\gamma^{-k}(Q)$ converging to $y$, and so $y\in \overline{\gamma^{-k}(Q)}$ for each $k\geq 0$.

On the other hand, if $y\in \bigcap_{k\geq 0} \overline{\gamma^{-k}(Q)}$, then for each $k$ there exists $\gamma^{-k}x_k \in \gamma^{-k}(Q)$ with $d(y,\gamma^{-k} x_k)\leq 1/k$ (here distance is being measured in $\R^n$). Hence $\{x_k\}_{k\geq 0}$ is a sequence in $Q$ and $\gamma^{-k} x_k\to y$, and so $y\in \mathcal{L}(Q)$.
\end{proof}

We now show that $\mc L(Q)$ is invariant under $\gamma$.
\begin{lem}\label{lem:conftoinvar}
The set $\mathcal{L}(Q)$ is a subspace of $\R^n$ that is invariant under $\gamma$. 
\end{lem}

\begin{proof}
First we check that $\mathcal{L}(Q)$ is an additive subgroup of $\R^n$. If $y,z\in \mathcal{L}(Q)$, then we may choose sequences $\{x_k\}_{k=0}^\infty,\{w_k\}_{k=0}^\infty\subset Q$ so that $\gamma^{-k}x_k\to y$ and $\gamma^{-k}w_k \to z$. Then $\gamma^{-k}(-x_k)\to -y$, and since $-x_k\in Q$ (recall $Q$ is symmetric), we conclude that $-y\in \mathcal{L}(Q)$. Furthermore, choosing $N$ large enough that $\gamma^{-N}(Q+Q)\subset Q$, we have $\gamma^{-N}(x_k+w_k)\in Q$.  Since $\gamma^{-k+N}(\gamma^{-N}(x_k+w_k))=\gamma^{-k}x_k+\gamma^{-k}w_k\to y+z$ as $k\to \infty$, we see that $y+z\in \mathcal{L}(Q)$. Thus, $\mathcal{L}(Q)$ is a subgroup of $\R^n$.

 Since $\gamma^{-1}Q\subset Q$, it follows that $\gamma^{-1} \mathcal L(Q)=\mathcal L(\gamma^{-1} Q)\subset \mathcal L(Q)$. As an intersection of closed sets, $\mc L(Q)$ is also closed. By a structure theorem of Bourbaki \cite[Section VII.1.2 Theorem 2]{bourbaki}, the closed subgroup $\mathcal{L}(Q)$  contains a largest linear subspace $V$ of $\R^n$ and  $\mathcal{L}(Q)=V\oplus H$, where $H=\mathcal{L}(Q)\cap V^\perp$ is discrete. Since $V$ and $\gamma^{-1}V$ are both subspaces of $\mathcal L(Q)$, we have $\gamma^{-1}V\subset V$.  Since $V$ and $\gamma^{-1}V$ have the same dimension, $V=\gamma^{-1}V$. In particular, $V$ is invariant under $\gamma$.
 
 We claim that $H=0$, which will show that $\mathcal L(Q)=V$ is an invariant subspace under $\gamma$. 
 The set $\mathcal{L}(Q)$ consists of the disjoint translates $V+h$ where $h$ ranges over $H$. Since $\gamma^{-1} \mathcal L(Q)\subset \mathcal L(Q)$, the element $\gamma^{-1}$ must send translates $V+h$ to other such translates. If $H\neq 0$ then we may choose $h\in H\setminus \{0\}$. It follows that the sets $\gamma^{-k}(V+h)$  are translates of $V$ which do not pass through the origin but accumulate at the origin. But this is impossible, since if $\epsilon$ is small enough that the points in $H$ are distance at least $\epsilon$ apart, then $B_\epsilon(0)\cap \mathcal{L}(Q)$ is contained in $V$. Thus $H=0$, completing  the proof.
\end{proof}

Conversely we may associate confining subsets to invariant subspaces of $\R^n$. Throughout the remainder of this section, $\|M\|$ will denote the operator norm of a matrix $M\in M_n(\R)$ with respect to the standard Euclidean inner product on $\C^n$ and $\|v\|$ will denote the Euclidean norm of a vector $v\in \C^n$. As we are using an operator norm, we have $\|M v\| \leq \|M \| \|v\|$ for any $M\in M_n(\R)$ and any $v\in \C^n$. 

Given an invariant subspace of $\R^n$, we now define a confining subset of $A$ as, essentially, the intersection of a neighborhood of this subspace with $A$, modulo some technical modifications. 

\begin{defn} \label{def:Q_epsilon}
Let $V\subset \R^n$ be a subspace that is invariant under $\gamma$. For any $\epsilon>0$, first define \[ Q_\epsilon^0(V)\vcentcolon=B_\epsilon(V) \cap A.\] That is, $Q_\epsilon^0(V)$ is the set of points of $A$ in the open $\epsilon$-neighborhood of $V$. This set is not necessarily confining under $\alpha^{-1}$, because it may not be closed under applying $\alpha^{-1}$. To remedy this,  choose $k_0$ large enough so that $\|\gamma^{-k}\|<1$ for all $k\geq k_0$, and define \[Q_\epsilon(V)\vcentcolon=\bigcup_{0\leq k< k_0} \gamma^{-k}(Q_\epsilon^0(V)).\] 
\end{defn} 
\noindent When $V$ is understood we will frequently write simply $Q_\epsilon$.

\begin{lem}\label{lem:invartoconf}
If $V$ is a subspace of $\R^n$ that is invariant under $\gamma$, then $Q_\epsilon(V)\subseteq A$ is confining under $\alpha^{-1}$.
\end{lem}

\begin{proof}
We first study properties of the set $Q_\epsilon^0(V)=Q_\epsilon^0$. It follows from the fact that $V$ (and $A$) is invariant under $\gamma$ and our choice of $k_0$ that $\gamma^{-k}(Q_\epsilon^0)\subset Q_\epsilon^0$ for $k\geq k_0$. For any $x\in A$, we have $\|\gamma^{-k}x\|\to 0$ as $k\to \infty$, and so $\gamma^{-k}(x)\in Q_\epsilon^0$ for any sufficiently large $k$. 
Choose $k_1$ large enough that $\|\gamma^{-k_1}\|<\frac{1}{2}$; then $\gamma^{-k_1}(Q_\epsilon^0+Q_\epsilon^0)\subset Q_\epsilon^0$.  

With these properties in hand we can show that $Q_\epsilon$ is confining. The fact that $Q_\epsilon$ is closed under multiplication by $\gamma^{-1}$ (Definition~\ref{def:confining}(a)) follows from the following two facts:
\begin{itemize}
\item[(i)] $\gamma^{-1}(\gamma^{-k}(Q_\epsilon^0))\subset \gamma^{-k-1}(Q_\epsilon^0)$ for $0\leq k<k_0-1$; and
\item[(ii)] $\gamma^{-1}(\gamma^{-k_0+1}(Q_\epsilon^0))\subset Q_\epsilon^0$.
\end{itemize}
\noindent Definition~\ref{def:confining}(b) holds since it holds for $Q_\epsilon^0\subset Q_\epsilon$. 
Finally, for any $y,z\in Q_\epsilon$ we have $\gamma^{-k_0}(y),\gamma^{-k_0}(z)\in Q_\epsilon^0$. Thus, $\gamma^{-k_0-k_1}(y+z)\in Q_\epsilon^0\subset Q_\epsilon$, and Definition~\ref{def:confining}(c) holds. Therefore, $Q_\epsilon$ is confining under $\alpha^{-1}$.
\end{proof}

In the final results in this subsection, we investigate how the sets $Q_\epsilon(V)$ change under modifying $\epsilon$ or $V$. We start by showing that varying $\epsilon$ yields equivalent confining subsets.

\begin{lem}
\label{lem:nbhdequiv}
If $V\subset \R^n$ is a subspace invariant under $\gamma$, then  $Q_\epsilon(V)\sim Q_\delta(V)$ for any $\epsilon,\delta>0$.
\end{lem}

\begin{proof}
Suppose without loss of generality that $\epsilon>\delta$. Then clearly $Q_\epsilon(V) \preceq Q_\delta(V)$. On the other hand, we may choose $n$ large enough that $\|\gamma^{-n}\|<\delta/\epsilon$.  
Thus $\gamma^{-n}(Q_\epsilon(V)) \subset Q_\delta(V)$, and so $Q_\delta(V)\preceq Q_\epsilon(V)$.
\end{proof}

Now we investigate what happens to the confining subsets when we change subspaces. As expected, the order on confining subsets is reversed with respect to the inclusion of subspaces, and distinct subspaces yield inequivalent confining subsets. 

\begin{lem}
\label{lem:subspacechange}
Let $V,W\subset \R^n$ be  subspaces invariant under $\gamma$. If $V\subsetneq W$, then $Q_\epsilon(W)\precneq Q_\epsilon(V)$. Moreover, if neither $V\not\subset W$ nor $W\not\subset V$, then $Q_\epsilon(V)$ and $Q_\epsilon(W)$ are incomparable.
\end{lem}

\begin{proof}
Since $A$ contains $\Z^n$, it is $\delta$-dense in $\R^n$ for any sufficiently large $\delta$. By Lemma \ref{lem:nbhdequiv}, we may choose $\epsilon=\delta$ large enough to satisfy this property, without loss of generality. Suppose  $V\subsetneq W$.  Then $Q_\epsilon(V)\subset Q_\epsilon(W)$, and so $Q_\epsilon(W) \preceq Q_\epsilon(V)$. If $w\in Q_\epsilon(W)$, then the distance from $\gamma^{-r}w$ to $V$ is at least $\frac{1}{\|\gamma\|^r}$ times the distance from $w$ to $V$.  Since $V\neq W$, we may choose elements of $W\subseteq Q_\epsilon(W)$ at arbitrarily large distance from $V$.  
This shows that $\inf \{k : \gamma^{-k} w \in Q_\epsilon(V)\}$ is arbitrarily large for elements $w\in Q_\epsilon(W)$. Therefore, $Q_\epsilon(V)\not\preceq Q_\epsilon(W)$.

If $V\not\subset W$ and $W\not\subset V$, then considering elements of $V$ arbitrarily far from $W$ shows that $Q_\epsilon(W)\not\preceq Q_\epsilon(V)$. Similarly, $Q_\epsilon(V)\not\preceq Q_\epsilon(W)$. This verifies the moreover statement.
\end{proof}

\subsection{Tiles}\label{sec:tiles}

Our goal is to show that all confining subsets of $A$ under $\alpha^{-1}$ are equivalent to the subsets $Q_\epsilon(V)$ defined by  subspaces $V$ invariant under $\gamma$. In this subsection, we use \emph{tiles} \`{a} la Lagarias--Wang \cite{lag_wang} to show that if $Q$ is such a confining subset, then $Q\preceq Q_\epsilon(\mc L(Q))$ for some invariant subspace $V$ (see Corollary~\ref{cor:bddinQ}).   The reverse relation will be shown in the next subsection.

  We now choose a transversal $\mathscr D$ for $\gamma \Z^n$ in $\Z^n$ and require that $0\in \mathscr D$. The cardinality of $\mathscr D$ will be the absolute value of $\det \gamma$. Lagarias--Wang in \cite{lag_wang} consider the following set: 
 
\begin{defn}\label{def:tile}
A \emph{tile} is a set 
\[T=T(\gamma,\mathscr D)= \left\{ \sum_{j=1}^\infty \gamma^{-j} v_j : v_j \in \mathscr D \text{ for all } j \geq 1 \right\}.\]
\end{defn}
   Since $0\in \mathscr{D}$, the tile $T$ contains the set of finite sums $\left\{ \sum_{j=1}^k \gamma^{-j} v_j : k\geq 1 \text{ and } v_j \in \mathscr D \text{ for all } j \geq 1 \right\}\subset A$. As shown in \cite{lag_wang}, the tile $T\subset \R^n$ is a compact subset of $\R^n$.

\begin{lem}
\label{lem:ballwordlength}
If $Q\subset A$ is confining under $\alpha^{-1}$, then for all $\epsilon>0$, there is a bound on the word length of any element of $B_\epsilon(0)\cap A$ with respect to $Q\cup \{t^{\pm 1}\}$.
\end{lem}

\begin{proof}
Fix any $\epsilon>0$, and let $z\in B_\epsilon(0)\cap A$.  We may write $z=x+y$, where $x=\gamma^{-k}a_{-k}+\cdots+\gamma^{-1}a_{-1}$, for some $k\geq 1$ and $a_i \in \mathscr D$, and $y\in \Z^n$. The element $x$ is in $T$. Since $T$ is compact and since $z=x+y$ is lies in the bounded set $B_\varepsilon(0)\cap A$ and the translate $T+y$, there are finitely many possibilities for the element $y\in \Z^n$. In particular, the word length of $y$ is bounded.
By  the method of proof of Lemma \ref{lem:RsubsetQ}, any finite sum $\sum_{i=1}^k \gamma^{-i} v_i$ with $v_i \in \mathscr D$ has bounded word length in $Q\cup \{t^{\pm 1}\}$ (the proof is also more or less identical to that of \cite[Lemma 3.20]{AR}). Hence $x$, and therefore also $z$, has bounded word length in $Q\cup \{t^{\pm 1}\}$.
\end{proof}

\begin{lem}\label{cor:bddinQ}
If $Q\subset A$ is confining under $\alpha^{-1}$, then  $Q\preceq Q_\epsilon(\mathcal{L}(Q))$ for any $\epsilon>0$.
\end{lem}

\begin{proof}
Fix $\epsilon>0$. We will show that the set $Q_\epsilon(\mathcal{L}(Q))$ has bounded word length with respect to $Q\cup \{t^{\pm 1}\}$.
By Lemma \ref{lem:limsetequivdefn}, $\mathcal{L}(Q)$ lies in the closure of $Q$. Note that $Q_\epsilon(\mathcal{L}(Q)) \subset Q_\delta^0(\mathcal{L}(Q))$ for some suitably large $\delta\geq \epsilon$, where $Q_\delta^0(\mc L(Q))$ is as in Definition~\ref{def:Q_epsilon}. So for any $z\in Q_\epsilon(\mathcal{L}(Q))$, there is $x\in \mathcal{L}(Q)$ with $\|z-x\|< \delta$. By the definition of $\mc L(Q)$, there exists $y\in Q$ with $\|x-y\|<\delta$. Thus  
$z\in B_{2\delta}(y)\cap A=y+(B_{2\delta}(0)\cap A)$.  Lemma \ref{lem:ballwordlength} provides a constant $D$ large enough that the elements of $B_{2\delta}(0)\cap A$ have word length at most $D$ with respect to $Q\cup \{t^{\pm 1}\}$.  Therefore the word length of $z$ with respect to $Q\cup \{t^{\pm 1}\}$ is at most $D+1$.
\end{proof}

The next lemma will be useful in the following subsections for proving that $Q\sim Q_\epsilon(\mathcal L(Q))$.

\begin{lem}
\label{lem:gelfand}
There exist constants $A,B>0$ and $0<\zeta<\eta<1$ such that \[A\zeta^i \leq \|\gamma^{-i}\| \leq B\eta^i\] for all $i\geq 0$. Moreover, for any $v\in \R^n$ we have $A\zeta^i \|v\|\leq \|\gamma^{-i}v\| \leq B\eta^i\|v\|$ for all $i\geq 0$.
\end{lem}

\begin{proof}
Set $\zeta_0$ and $\eta_0$ to be the largest and smallest absolute value of an eigenvalue of $\gamma$, respectively. We have $\zeta_0\geq \eta_0>1$, since $\gamma$ is expanding. By Gelfand's formula, we have \[\|\gamma^i\|^{1/i}\to \zeta_0 \text{ and } \|\gamma^{-i}\|^{1/i}\to \eta_0^{-1} \text{ as } i\to \infty.\] Choose any numbers $\zeta_1,\eta_1$ with $1<\zeta_0<\zeta_1$ and $1>\eta_1>\eta_0^{-1}$. Then there exists $J>0$ such that \[\|\gamma^i\|^{1/i} <\zeta_1 \text{ and } \|\gamma^{-i}\|^{1/i}<\eta_1 \text{ for all } i> J.\] Setting $A_1=\sup\{\|\gamma^i\|/\zeta_1^i\}_{0\leq i\leq J}$ and $B_1=\sup\{\|\gamma^{-i}\|/\eta_1^i\}_{0\leq i\leq J}$, we immediately  have  that $\|\gamma^i\|\leq A_1\zeta_1^i$ and $\|\gamma^{-i}\|\leq B_1\eta_1^i$ for each $i\geq 0$. The first of these inequalities  gives \[1=\|\operatorname{Id}\|=\|\gamma^{-i}\gamma^i\|\leq \|\gamma^{-i}\| \|\gamma^i\| \leq \|\gamma^{-i}\| A_1\zeta_1^i,\] and so $\|\gamma^{-i} \|\geq \frac{1}{A_1}\zeta_1^{-i}$. Setting $A=A_1^{-1}$, $B=B_1$, $\zeta=\zeta_1^{-1}$, and $\eta=\eta_1$ gives the claimed inequalities for $\|\gamma^{-i}\|$. For the moreover statement, the inequality $\|\gamma^{-i}v\|\leq B\eta^i\|v\|$ follows immediately by definition of the operator norm. For the inequality $\|\gamma^{-i}v\|\geq A\zeta^i \|v\|$, note that $\|v\|=\|\gamma^i\gamma^{-i} v\| \leq A_1\zeta_1^i \|\gamma^{-i}v\|$.
\end{proof}

\subsection{Invariant subspaces of $\R^n$}\label{sec:invarsubspaces}
In this subsection, we will complete the proof that $Q\sim Q_\epsilon(\mathcal L(Q))$ by showing that $Q_\epsilon(\mc L(Q))\preceq Q$ (Proposition~\ref{prop:Q_epsilonpreceqQ}).  To do so, we need to understand the structure of invariant subspaces of $\R^n$.  We begin with a brief review of the generalized eigenspaces of $\gamma$ and their properties. 

The \emph{generalized eigenspaces} of $\gamma$ are the subspaces of $\C^n$, which we denote $C_{\lambda,i}$, defined by $C_{\lambda,i}\vcentcolon=\ker(\gamma-\lambda I)^i$, where $\lambda \in \C$ is an eigenvalue of $\gamma$ and $i\geq 0$. Each $C_{\lambda, i}$ is invariant under $\gamma$.  Moreover,  \[0=C_{\lambda,0} \subsetneq C_{\lambda,1} \subsetneq \ldots \subsetneq C_{\lambda,i_\lambda} = C_{\lambda,i_\lambda+1}=C_{\lambda,i_\lambda+2}=\ldots\] for some unique number $i_\lambda$. The subspace $C_{\lambda,i_{\lambda}}$ is the space of generalized $\lambda$-eigenvectors. Since we assume that  the minimal and characteristic polynomials of $\gamma$ are equal, we have in fact that $\dim_\C C_{\lambda,i}=i$ for $0\leq i\leq i_\lambda$ and $i_\lambda$ is the algebraic multiplicity of $\lambda$.

If $\lambda$ is real, then $\gamma$ also preserves the real vector space $R_{\lambda,i}\vcentcolon=C_{\lambda,i}\cap \R^n$ for each $i$, and $\dim_\R R_{\lambda,i}=i$. In this case, there is an ordered basis $\mathcal B=\{w_1,\ldots,w_{i_\lambda}\}$ for $R_{\lambda,i_\lambda}$ which restricts to a basis $\{w_1,\ldots,w_i\}$ of $R_{\lambda,i}$ for each $i\leq i_\lambda$ and such that $\gamma$ acts on $R_{\lambda,i}$ relative to this ordered basis 
as the $i\times i$ Jordan block with $\lambda$ on the diagonal.

If $\lambda$ is complex, then similarly there is an ordered basis $\mathcal B=\{w_1,\ldots,w_{i_\lambda}\}$ for the complex vector space $C_{\lambda,i_\lambda}$ which restricts to a basis $\{w_1,\ldots,w_i\}$ for each $C_{\lambda,i}$ and for which the action of $\gamma$ on $C_{\lambda,i}$ is given by an $i\times i$ Jordan block as  above. In this case, the vectors $w_i$ are not real, and we consider the real and imaginary parts $u_i=\Re w_i$ and $v_i=\Im w_i$. The element $\gamma$ preserves each of the subspaces $R_{\lambda,i}\vcentcolon=\langle u_1,v_1,\ldots,u_i,v_i\rangle$ of $\R^n$ for $i=1,\ldots,i_\lambda$. Note that $\dim_\R R_{\lambda,i}=2i$. Taking $ \{u_1,v_1,\ldots,u_i,v_i\}$ as an ordered basis of $R_{\lambda,i}$ and writing $\lambda = a+bi$,  the matrix for $\gamma$ relative to $\{u_1,v_1,\ldots,u_i,v_i\}$ is given by the $2i\times 2i$ \emph{real Jordan block} \begin{equation}\label{eqn:realblock}\begin{pmatrix} Q(a,b) & I_2 & 0 & \cdots & 0 \\ 0 & Q(a,b) & I_2 & \cdots & 0 \\ 0 & 0 & Q(a,b) & \cdots & 0 \\ \vdots & \vdots & \vdots & \ddots & \vdots \\ 0 & 0 & 0 & \cdots & Q(a,b) \end{pmatrix}\end{equation} where $Q(a,b)$ is the $2\times 2$ matrix $\begin{pmatrix} a & b \\ -b & a \end{pmatrix}$ and $I_2$ is the $2\times 2$ identity matrix.

If $\lambda$ is not real, then we may consider the complex conjugate eigenvalue $\overline{\lambda}$. If $\mathcal B=\{w_1,\ldots,w_{i_\lambda}\}$ is the ordered basis for $C_{\lambda,i_\lambda}$ constructed in the last paragraph, then $\overline{\mathcal B}=\{\overline{w_1},\ldots,\overline{w_{i_\lambda}}\}$ is an ordered basis for $C_{\overline{\lambda},i_\lambda}$ satisfying all the same properties. We will assume that $\overline{\mathcal B}$ has been chosen for $C_{\overline{\lambda},i_\lambda}$ if $\mathcal B$ has been chosen for $C_{\lambda,i}$. Then $R_{\lambda,i}=R_{\overline{\lambda},i}$ for each $i$, and moreover, \[(C_{\lambda,i}\oplus C_{\overline{\lambda},i})\cap \R^n=R_{\lambda,i}=R_{\overline{\lambda},i}.\]

The following is well-known (see \cite[Theorems 12.2.1 and 12.2.4]{invariant}). 

\begin{prop}
\label{prop:invariantsubspaceclassification}
The  subspaces  of $\R^n$ that are invariant under $\gamma$ are exactly the direct sums of subspaces of the form $R_{\lambda,i}$, where $\lambda$ is a real or complex eigenvalue of $\gamma$ and $0\leq i\leq i_\lambda$.
\end{prop}

Lastly, we discuss real Jordan forms of matrices. Our discussion follows \cite[Section 3.1]{farb_mosher} with certain modifications. A \emph{real Jordan form} of $\gamma$ is a real block diagonal matrix $\beta$ with diagonal blocks that may be described as follows. Denote  the real eigenvalues of $\gamma$ by $\lambda_1,\ldots, \lambda_u$ and  the complex eigenvalues of $\gamma$ (which occur in complex conjugate pairs) by $\mu_1,\overline{\mu_1},\ldots,\mu_s,\overline{\mu_s}$. Let $j_t$ be the algebraic multiplicity of $\lambda_t$ and  $k_t$  the algebraic multiplicity of $\mu_t$. Then $\beta$ has one $j_t\times j_t$ Jordan block for each real eigenvalue $\lambda_t$ and one $2k_t\times 2k_t$ real Jordan block for each complex conjugate pair $\mu_t,\overline{\mu_t}$. The matrix $\gamma$ is conjugate to $\beta$ via a real invertible matrix. Moreover, the real Jordan form $\beta$ is uniquely determined by $\gamma$ up to permuting the diagonal blocks.

With this description of invariant subspaces and real Jordan forms in hand, we proceed to prove $Q_\epsilon(\mathcal{L}(Q))\preceq Q$, which, combined with Lemma \ref{cor:bddinQ}, will show that $Q_\epsilon(\mathcal{L}(Q))\sim Q$.

\begin{prop}\label{prop:Q_epsilonpreceqQ}
If $Q\subset A$ is confining under $\alpha^{-1}$ and $\epsilon>0$, then $Q_\epsilon(\mathcal{L}(Q))\preceq Q$.
\end{prop}

\begin{proof}
For simplicity, we will write $V=\mathcal L(Q)$ throughout the proof. By Lemma~\ref{lem:nbhdequiv}, $Q_\epsilon(V) \sim Q_\delta(V)$ for any $\epsilon,\delta>0$, so it suffices to prove the statement for $Q_1(V)$. If $Q$ is contained in some metric neighborhood of $V$, then $Q_1(V)\preceq Q$ by the proof of Lemma \ref{lem:nbhdequiv}. 

Assume for contradiction that $Q$ is not contained in \emph{any} metric neighborhood of $V$. We will show that there is a compact set $C\subset \R^n\setminus V$ that contains points of $\gamma^{-i}(Q)$ for $i$ arbitrarily large. Since $C$ is compact, we may then take a subsequence to find points of $\gamma^{-i}(Q)$ which converge to a point of $C$ as $i\to \infty$. This will imply that that there is a point of $V$ in $C$, which is a contradiction.

To prove the existence of such a set $C$, we consider the quotient vector space $\R^n/V$. Since $V$ is invariant, $\gamma$ induces a linear automorphism $\overline{\gamma}$ of $\R^n/V$.  The real Jordan form for $\overline{\gamma}$ may be obtained by taking sub-blocks of the real Jordan blocks of $\gamma$. This shows that the eigenvalues of $\overline{\gamma}$ are a subset of the eigenvalues of $\gamma$, and hence 
$\overline{\gamma}$ is also expanding.

The Euclidean norm $\|\cdot\|$ on $\R^n$ induces a norm $\|\cdot\|'$ on $\R^n/V$ as follows. If $u\in \R^n$ and $\overline{u}\in \R^n/V$ is its equivalence class, then \[\|\overline{u}\|'=\|u-\operatorname{proj}_V(u)\|=d(u,V)\] where $\operatorname{proj}_V$ denotes orthogonal projection to $V$ and $d(u,V)$ denotes the distance from $u$ to $V$ in $\R^n$. By Lemma \ref{lem:gelfand} applied to the expanding matrix $\overline{\gamma}$, there are constants $A,B>0$ and $0<\zeta<\eta<1$ such that \[A\zeta^i \|\overline{v}\|' \leq \|\overline{\gamma}^{-i}\overline{v}\|'\leq B\eta^i \|\overline{v}\|'\] for all $i\geq 0$ and all $\overline{v}\in \R^n/V$. Translating this into a statement about vectors in $\R^n$ yields that for all $i\geq 0$,
\[A\zeta^i d(v,V)\leq d(\gamma^{-i}v,V)\leq B\eta^id(v,V).\]

By assumption, there are elements $v\in Q$ with $d(v,V)$ arbitrarily large. Since $d(\gamma^{-i}v,V)\geq A\zeta^id(v,V)$, the quantity $\min \{i : \gamma^{-i} v \in Q_1^0(V)=B_1(V)\}$ is arbitrarily large. Fix $v\in Q$, and set $j=\min\{i:\gamma^{-i}v\in Q_1^0(V)\}$. We may assume that $j>1$. By definition of $j$, we have $d(\gamma^{-j+1}v,V)>1$. Consequently, 
\[1\geq d(\gamma^{-j}v,V)=d(\gamma^{-1}(\gamma^{-j+1}v),V)\geq A\zeta d(\gamma^{-j+1}v,V)\geq A\zeta.\]
 Since $V$ lies in the closure of $\gamma^{-j}(Q)$, we may choose $u\in Q$ with $\gamma^{-j}(u)$ arbitrarily close to $\operatorname{proj}_V(\gamma^{-j}v)$. Thus $\gamma^{-j}(v)-\gamma^{-j}(u)$ may be taken to be nearly orthogonal to $V$ and  of length approximately $d(\gamma^{-j}(v),V)$. More precisely, we may choose $u\in Q$ such that 
\[\|\gamma^{-j}(v)-\gamma^{-j}(u)\|\leq 2 \text{ and } d(\gamma^{-j}(v)-\gamma^{-j}(u),V) > A\zeta/2.\] 
Consequently $\gamma^{-j}(v)-\gamma^{-j}(u)$  lies in the compact set $C=\overline{B_2(0)}\setminus B_{A\zeta/2}(V)$.  Since  $\gamma^{-j}(v)-\gamma^{-j}(u)\in \gamma^{-j}(Q+Q)$, taking $k_0$ large enough that $\gamma^{-k_0}(Q+Q)\subset Q$ yields $\gamma^{-j}(v)-\gamma^{-j}(u)\in \gamma^{-j+k_0}(Q)$.

Since $j$ may be taken arbitrarily large, we have shown  that there are points of $\gamma^{-k}(Q)$ in the compact set $C$ for $k$ arbitrarily large. Choose a sequence $k_i\to \infty$ and $v_i\in Q$ with $\gamma^{-k_i}(v_i)\in C$. Since $C$ is compact, we may pass to a subsequence to assume that $\gamma^{-k_i}(v_i)$ converges to a point $z\in C$. Then $\gamma^{-k_i}(v_i)\in \gamma^{-k}(Q)$ for any fixed $k$ and for any $i$ large enough, and hence $z$ lies in the closure of $\gamma^{-k}(Q)$. Since $k$ is arbitrary, this show that $z\in \mathcal L(Q)=V$ by Lemma \ref{lem:limsetequivdefn}. This is a contradiction, as $C\cap V=\emptyset$ by construction.
\end{proof}

\subsection{The poset of invariant subspaces}

The  subspaces of $\R^n$ invariant under $\gamma$  are partially ordered by inclusion. In this section we describe the resulting poset concretely. Let $\lambda_1,\ldots,\lambda_r$ be the real eigenvalues of $\gamma$ and $\mu_1,\overline{\mu_1},\ldots,\mu_s,\overline{\mu_s}$ the complex eigenvalues of $\gamma$. Let $k_i$ be the algebraic multiplicity of $\lambda_i$ and $l_i$ the algebraic multiplicity of $\mu_i$. Then any invariant subspace has the form \[R_{i_1,\ldots,i_r,j_1,\ldots,j_s} = \left( \bigoplus_{u=1}^r R_{\lambda_u,i_u}\right)\oplus \left(\bigoplus_{v=1}^s R_{\mu_v,j_v}\right)\] for some $i_u\leq k_u$ and $j_v \leq l_v$ for each $u$ and $v$. We have $R_{i_1,\ldots,i_r,j_1,\ldots,j_s}\subset R_{i'_1,\ldots,i'_r,j'_1,\ldots,j'_s}$ if and only if $i_u\leq i'_u$ and $j_v\leq j'_v$ for each $u$ and $v$. Thus the following proposition follows immediately:

\begin{prop}
\label{prop:posetofsubspaces}
Let $k_1,\ldots,k_r$ be the algebraic multiplicities of the real eigenvalues of $\gamma$. Choose one eigenvalue from each complex conjugate pair of complex eigenvalues of $\gamma$, and let $l_1,\ldots,l_s$ be their algebraic multiplicities.  The poset of  subspaces of $\R^n$ invariant under $\gamma$ is isomorphic to $\Div(k_1,\ldots,k_r,l_1,\ldots,l_s)$.  
\end{prop}

We are now ready to prove Proposition~\ref{prop:P-invar}.

\begin{proof}[Proof of Proposition~\ref{prop:P-invar}]
 Consider an expanding matrix $\gamma\in M_n(\Z)$  with equal minimal and characteristic polynomial, and form the group $G=G(\gamma)$. The lattice $\mathcal P_-(G)$ is isomorphic to the poset of confining subsets of $A$ under $\alpha^{-1}$, which is multiplication by $\gamma^{-1}$, by Proposition~\ref{prop:Ppmconfining}. If $V\subset \R^n$ is an invariant subspace for $\gamma$ then the set $Q_1(V)$ is confining. Thus there is a function $V\mapsto Q_1(V)$ from invariant subspaces of $\R^n$ to confining subsets of $A$ considered up to equivalence. By Lemma \ref{lem:subspacechange} this function is injective and order-reversing. By Lemmas \ref{cor:bddinQ} and \ref{prop:Q_epsilonpreceqQ} this function is surjective: if $Q\subset A$ is confining then $\mathcal L(Q)$ is an invariant subspace by Lemma \ref{lem:conftoinvar} and we have $Q\sim Q_1(\mathcal L(Q))$.  Therefore,  $\mathcal P_-(G)$ is isomorphic to the opposite of the poset of invariant subspaces of $\gamma$. By Proposition \ref{prop:posetofsubspaces}, this poset is isomorphic to its own opposite, so $\mathcal P_-(G)$ is isomorphic to the poset of invariant subspaces.
\end{proof}

\subsection{Corollaries to Theorem \ref{thm:char=min}}

Having discussed invariant subspaces, we will now pause to derive the various corollaries to Theorem \ref{thm:char=min}. We do this to illustrate the immediate applications of the invariant subspace machinery. The only component of Theorem \ref{thm:char=min} that remains to be proven is the association to actions on Heintze groups; this will be completed in Section~\ref{sec:heintze}.

\begin{proof}[Proof of Corollary \ref{cor:polynomialring}]
Consider the monic polynomial $p(x)=a_0+a_1x+\cdots +a_{n-1}x^{n-1}+x^n$ satisfying the hypotheses of the corollary and the associated ring $R=\Z[x]/(p)$. Then $R$ is isomorphic to $\Z^n$ as an abelian group, with $\{1+(p),x+(p),\ldots,x^{n-1}+(p)\}$ being a basis of free abelian groups. Of course, $R$ is itself a cyclic $R$-module. Moreover, with respect to our free basis for $R$, the matrix of the linear endomorphism of $R$ defined by multiplication by $\gamma=x+(p)$ is exactly the companion matrix $\delta$ to $p$. The matrix $\delta$ is expanding since its characteristic polynomial (and minimal polynomial) is $p$, all of whose roots lie outside the unit disk. Hence $G(R,\gamma)$ is isomorphic to $G(\delta)=\langle \Z^n, t : trt^{-1}=\delta r \text{ for } r\in \Z^n\rangle$ and the poset $\mathcal P_+(G(R,\gamma))$ is described by Theorem \ref{thm:char=min}. Thus, $\mathcal P_+(G)$ is isomorphic to $\Div(n_1,\ldots,n_r)$, where $p=up_1^{n_1}\cdots p_r^{n_r}$ is the prime factorization of $p$ in $\Z[[x]]$.

It remains only to describe $\mathcal P_-(G)$, the poset of invariant subspaces for $\delta$. Since $p$ is irreducible, it has no repeated roots in $\C$. By Proposition \ref{prop:posetofsubspaces}, the poset of invariant subspaces is isomorphic to $\Div(1,\ldots,1)$ with the number of 1's between the parentheses being the number of real roots plus half the number of complex roots of $p$. This completes the proof.
\end{proof}

\begin{proof}[Proof of Corollary \ref{cor:squarefreeconstant}]
The statement about $\mc P_-(G)$ follows from Corollary \ref{cor:polynomialring}. It remains  to describe $\mathcal P_+(G)$, which is isomorphic to the poset of divisors of $p$ in $\Z[[x]]$, up to associates. Let $(-1)^{s} q_1\cdots q_r$ be the prime factorization of the constant term of $p$. Then by Corollary \ref{cor:seriesfactorization}, $p$ factors as $p=(-1)^{s} p_1\cdots p_r$ where the constant term of $p_i$ is $q_i$. By Lemma \ref{lem:primeconstantterm}, $p_i$ is a prime power series. Hence $p=(-1)^{s} p_1\cdots p_r$ is the prime factorization of $p$, and so the poset $\mathcal P_+(G)$ is isomorphic to $\Div(1,\ldots,1)$ where the number of 1's between the parentheses is $r$, the number of prime factors of the constant term of $p$.
\end{proof}

\section{Elements of $\mc P_-(G)$: actions on Heintze groups}\label{sec:heintze}

The goal of this section is to prove the remaining part of Theorem \ref{thm:char=min} and Theorem \ref{thm:P-description},  that each element of $\mc P_-(G)$ contains an action on a Heintze group. The theorems will follow from:

\begin{thm}\label{thm:heintzegroup}
Fix an expanding matrix $\gamma\in M_n(\Z)$ whose characteristic and minimal polynomials are equal, and consider the ascending HNN extension $G=G(\gamma)$ defined by $\gamma$. 
Every element of $\mc P_-(G)$ contains an action of $G$ on a quasi-convex subspace $\R^k \times \R$ of a Heintze group $\C^k \rtimes \R$ for some $k\leq n$.
\end{thm}

Recall that $A$ is the $\Z[\gamma]$-module $A=\bigcup_{i\geq 0} \gamma^{-i} \Z^n$ (see Definition~\ref{def:A}), a sub-module of $\Q^n \subset \R^n$, and that $A$ is identified with the normal closure of $\Z^n$ in $G(\gamma)$.

Our strategy for proving Theorem~\ref{thm:heintzegroup} is to embed $G=A\rtimes \Z$ into a \emph{Heintze group} $\C^n\rtimes \R$. This Heintze group  can be equipped with a left-invariant Riemannian metric that makes it a hyperbolic metric space on which $G$  acts on by isometries (details are provided below). However, this action will not be cobounded and, in particular, will not define an element of $\mc H(G)$.  To deal with this problem, we single out an embedded quasi-convex subspace of $\C^n\rtimes \R$ on which $G$ acts coboundedly.  Quasiconvexity of the subspace ensures that it is also a hyperbolic metric space, and so the action on this subspace \emph{will} define an element of $\mc H(G)$. 

To form a Heintze group, we need a one-parameter subgroup of $\GL_n(\C)$ containing $\gamma$. We will denote such a subgroup by $\gamma^t=e^{tX}$ where $X$ lies in $\mathfrak{gl}_n(\C)$, the Lie algebra of $\GL_n(\C)$ consisting of all $n\times n$ complex matrices. However, we emphasize that the matrix $\gamma$ lies on multiple one-parameter subgroups, so that $\gamma^t$ is not defined solely in terms of $\gamma$. For ease in the proofs that follow, we will choose the one-parameter subgroup carefully as in the following lemma, which we prove later in this section.

\begin{lem} 
\label{lem:unitaryrealgroup}
The matrix $\gamma$ lies on a one-parameter subgroup $\gamma^t$ of $\GL_n(\C)$ satisfying $\gamma^t=\delta^t\epsilon^t$, where
\begin{itemize}
\item the one-parameter subgroup $\delta^t$ is conjugate into $U(n)$, the group of unitary $n\times n$ matrices;
\item the one-parameter subgroup $\epsilon^t$ is contained in $\GL_n(\R)$; and
\item the groups $\delta^t$ and $\epsilon^t$ commute. 
\end{itemize}
\end{lem}

In the special case that $\gamma$ lies on a one-parameter subgroup of $\GL_n(\R)$, we may also define a Heintze group $\R^n\rtimes \R$ into which $G$ embeds. It will carry a left-invariant hyperbolic metric. The action of $G$ will, in fact, be cobounded in this case, and hence it will not be necessary to pass to a subspace. However, in the case that $\gamma$ does not lie on a one-parameter subgroup of $\GL_n(\R)$, we will need the full generality of the following discussion.

Consider $\C^n$ with the Euclidean norm $\| \cdot\|$. This is preserved by the unitary matrices $U(n)$. We let $\|\cdot\|$ denote the induced operator norm on $\GL_n(\C)$. 

\begin{cor}
\label{cor:unitarynorm}
Let $\gamma^t=\delta^t\epsilon^t$ be a one-parameter subgroup as in Lemma \ref{lem:unitaryrealgroup}.  There is a constant $M\geq 1$ such that $\|\delta^t\|\leq M$ for all $t$.
\end{cor}

\begin{proof}
Choose $\xi\in \GL_n(\C)$ with $\xi \delta^t \xi^{-1} \in U(n)$ for all $t$. Then $\|\delta^t\|\leq \|\xi^{-1}\| \|\xi\delta^t \xi^{-1}\|\|\xi\|=\|\xi^{-1}\|\|\xi\|$, and we set $M=\|\xi^{-1}\|\|\xi\|$. Note that $M\geq \|\xi^{-1} \xi\|=1$.
\end{proof}

With this particular one-parameter subgroup in hand, we  let $\R$ act on $\C^n$ by the group $\gamma^t$ and define the group $H=\C^n\rtimes_{\gamma^t} \R$. That is, $H$ is identified with $\C^n\times \R$ as a set and the group operation is given by $(z,t)\cdot (w,s)=(z+\gamma^tw,t+s)$. The Lie group $H$ is an example of a Heintze group (see \cite{heintze} for the general definition) and is equipped with a left-invariant Riemannian metric defined as follows. As a manifold, $H$ may be equipped with a single chart $(z,t)\mapsto (x_1,y_1,\ldots,x_n,y_n,t)$ if  $z=(x_1+iy_1,\ldots,x_n+iy_n)\in \C^n.$ The vector fields $\frac{\partial}{\partial x_i},\frac{\partial}{\partial y_i}$, and $\frac{\partial}{\partial t}$ identify each tangent space of $H$ with $\R^{2n+1}$ in the natural way. At the identity $(0,0)$, we equip $H$ with the usual Euclidean metric $dx_1^2+dy_1^2+\cdots +dx_n^2+dy_n^2+dt^2.$ 
The inner product on $T_{(z,t)}(H)$ is given by pulling back the metric at $(0,0)$ via left-translations, i.e.,  the inner product of $(v,r),(w,s)\in T_{(z,t)}(H)\cong \C^n\times \R$ is $\langle \gamma^{-t}v,\gamma^{-t}w\rangle +rs$.

\begin{thm}[{\cite[Theorem~3]{heintze}}]
\label{thm:heintze}
The Lie group $H$ with the left-invariant Riemannian metric described above has negative sectional curvature.
\end{thm}

Since $H$ has a cocompact group of isometries (i.e., its action on itself by left translations), $H$ in fact has sectional curvature $\leq \kappa$ for some $\kappa<0$. Thus $H$ is a complete, locally CAT($\kappa$) space. It follows from the Cartan-Hadamard Theorem \cite[Theorem II.4.1]{bridson_haefliger} that $H$ is in fact globally CAT($\kappa$):

\begin{cor}
\label{cor:heintzehyperbolic}
The Lie group $H$ is CAT($\kappa$) for some $\kappa<0$ and therefore Gromov hyperbolic.
\end{cor}

The group $G=A\rtimes \Z$ naturally embeds into $H=\C^n\rtimes \R$ and therefore acts isometrically on $H$ by  left translations. However, this action is never cobounded. Instead, we see that $G$ preserves the subspace $P=\R^n\times \R$ of the space $H=\C^n\times \R$. It is important to emphasize that $P$ is \emph{not} a subgroup of $H$ (unless $\gamma^t\subset \GL_n(\R)$) but simply a sub-manifold. We equip $P$ with the induced Riemannian metric as a sub-manifold of $H$. However, the submanifold $P$ still admits a useful structure proved in the next result. 

\begin{prop}
\label{prop:lipschitzretraction}
The subspace $P$ is quasi-isometrically embedded and hence quasiconvex in $H$.
\end{prop}

\begin{proof}
There is a retraction $\rho\colon H\to P$ defined by $\rho(z,t)=(\Re z,t)$. It suffices to show that $\rho$ is a Lipschitz map. 
For this, it suffices in turn to show that the derivative $\rho_*$ of $\rho$ is uniformly Lipschitz on tangent spaces.

Consider a point $(z,t)\in H$ and a tangent vector $(v,s)\in T_{(z,t)}(H)$. The map $\rho$ sends $(z,t)$ to $(\Re z,t)$ and $(v,s)$ to $(\Re v,s)\in T_{(\Re z,t)}(H)$. We have \[\|(v,s)\|^2=\|\gamma^{-t}v\|^2+s^2=\|\delta^{-t}\epsilon^{-t} v\|^2+s^2.\] By Corollary \ref{cor:unitarynorm}, we have $\|\delta^{-t}\epsilon^{-t}v\|\geq \frac{1}{\|\delta^{t}\|}\|\epsilon^{-t}v\|\geq \frac{1}{M}\|\epsilon^{-t}v\|$. Thus \[\|(v,s)\|^2\geq \frac{1}{M^2}\|\epsilon^{-t}v\|^2+\frac{1}{M^2}s^2.\] On the other hand, $\|(\Re v,s)\|^2=\|\delta^{-t}\epsilon^{-t} \Re v\|^2 + s^2$. Since $\epsilon^{-t}$ is real, we have $\epsilon^{-t} \Re v=\Re \epsilon^{-t} v$, and thus, \[\|(\Re v,s)\|^2=\|\delta^{-t}\Re \epsilon^{-t} v\|^2+s^2\leq M^2\|\Re \epsilon^{-t} v\|^2+M^2s^2\leq M^2\|\epsilon^{-t}v\|^2+M^2s^2.\] Therefore $\rho_*$ and $\rho$ are $M^2$-Lipschitz. 
\end{proof}

Lastly, it remains to prove Lemma \ref{lem:unitaryrealgroup} in order to proceed to the proof of Theorem \ref{thm:heintzegroup}. For this we recall the real Jordan form of the matrix $\gamma$. To describe it, consider as before the real and complex eigenvalues of $\gamma$. We further partition the real eigenvalues of $\gamma$ into a set $\lambda_1,\ldots,\lambda_u$ of positive eigenvalues and a set $-\nu_1,\ldots,-\nu_v$ of negative eigenvalues. As before, denote the complex eigenvalues by $\mu_1,\overline{\mu_1},\ldots,\mu_s,\overline{\mu_s}$, and set $\mu_j=a_j+b_ji$ with $a_j,b_j\in \R$. For a real number $\lambda$, denote by $J(\lambda,k)$ the $k\times k$ Jordan block with $\lambda$ on the diagonal.  For $a,b\in \R$ denote by $M(a,b,k)$ the $2k\times 2k$ real Jordan block defined in \eqref{eqn:realblock}.  Denote by $j_t, k_t,$ and $l_t$ the algebraic multiplicities of $\lambda_t, \nu_t,$ and $\mu_t$, respectively. Then there is a real invertible matrix $\eta$ such that $\eta \gamma \eta^{-1}$ is the real Jordan form of $\gamma$; i.e., $\eta \gamma \eta^{-1}$ is block diagonal with diagonal blocks $J(\lambda_1,j_1),\ldots,J(\lambda_u,j_u),J(-\nu_1,k_1),\ldots,J(-\nu_v,k_v),M(a_1,b_1,l_1),\ldots,M(a_s,b_s,l_s)$ occurring in that order.   

Our discussion of one-parameter groups is related to the discussion in \cite[Section 3.1]{farb_mosher}.

\begin{proof}[Proof of Lemma \ref{lem:unitaryrealgroup}]
To prove the lemma, we construct one-parameter groups of matrices containing each block of the real Jordan form of $\gamma$. We will then put these all together to create a one-parameter subgroup of the desired form containing $\gamma$ .  We consider each type of block in turn.

\noindent{\bf Case 1: $J(\lambda,k)$, $\lambda >0$.}  Consider a Jordan block $J(\lambda,k)$ for $\lambda$ a positive real number and $k> 0$. We will show that  $J(\lambda,k)$ lies on a one-parameter subgroup $\mathcal{G}(\lambda,k,t)$ such that  $\mathcal{G}(\lambda,k,1)=J(\lambda,k)$. 

 Consider the $k\times k$ nilpotent matrix $N$ with 1's on the superdiagonal and zeros elsewhere.   The matrix exponential $e^{tN}$ is seen to be \[\begin{pmatrix} 1 & t & \frac{1}{2!} t^2 & \cdots & \frac{1}{(k-2)!}t^{k-2} & \frac{1}{(k-1)!}t^{k-1} \\ 0 & 1 & t & \cdots & \frac{1}{(k-3)!}t^{k-3} & \frac{1}{(k-2)!}t^{k-2} \\ 0 & 0 & 1 & \cdots & \frac{1}{(k-4)!}t^{k-4} & \frac{1}{(k-3)!}t^{k-3} \\ \vdots & \vdots & \vdots & \ddots & \vdots & \vdots \\ 0 & 0 & 0 & \cdots & 1 & t \\ 0 & 0 & 0 & \cdots & 0 & 1 \end{pmatrix}.\] In particular, consider what happens when $t=1$. The matrices \[e^{N}=\begin{pmatrix} 1 & 1& \frac{1}{2!}  & \cdots & \frac{1}{(k-2)!} & \frac{1}{(k-1)!} \\ 0 & 1 & 1& \cdots & \frac{1}{(k-3)!}& \frac{1}{(k-2)!} \\ 0 & 0 & 1 & \cdots & \frac{1}{(k-4)!}& \frac{1}{(k-3)!}\\ \vdots & \vdots & \vdots & \ddots & \vdots & \vdots \\ 0 & 0 & 0 & \cdots & 1 & 1 \\ 0 & 0 & 0 & \cdots & 0 & 1 \end{pmatrix}, \begin{pmatrix} 1 & 1 & 0 & \cdots & 0 & 0 \\ 0 & 1 & 1 & \cdots & 0& 0 \\ 0 & 0 & 1 & \cdots & 0& 0 \\ \vdots & \vdots & \vdots & \ddots & \vdots & \vdots \\ 0 & 0 & 0 & \cdots & 1 & 1 \\ 0 & 0 & 0 & \cdots & 0 & 1 \end{pmatrix}, \begin{pmatrix} 1 & \lambda^{-1} & 0 & \cdots & 0 & 0\\ 0 & 1 & \lambda^{-1} & \cdots & 0 & 0 \\ 0 & 0 & 1 & \cdots & 0 & 0 \\ \vdots & \vdots & \vdots & \ddots & \vdots & \vdots \\ 0 & 0 & 0 & \cdots & 1 & \lambda^{-1} \\ 0 & 0 & 0 & \cdots & 0 & 1 \end{pmatrix} \tag{\textasteriskcentered} \label{eq:conjugatematrices}\] are all conjugate in $\GL_k(\R)$, as can be seen by considering the dimensions of their generalized eigenspaces for the unique eigenvalue 1. The matrix in the middle is the Jordan normal form of all three. Thus, choosing $\theta\in \GL_k(\R)$ conjugating the first matrix to the last, we have that $J(\lambda,k)=\lambda^1 \theta e^{N} \theta^{-1}$.  In particular, $J(\lambda, k)$  lies on the one-parameter subgroup $\mathcal{G}(\lambda,k,t)\vcentcolon=\lambda^t (\theta e^{Nt}\theta^{-1})$ of $\GL_k(\R)$ (note that $\lambda^t$ here is just a scalar, which commutes with any matrix $\theta e^{Nt'} \theta^{-1}$).

\noindent{\bf Case 2: $J(-\nu,k)$, $\nu>0$.}  Next we consider a Jordan block $J(-\nu,k)$ where $\nu$ is a positive real number. As in the last paragraph, there is a one-parameter subgroup of $\GL_k(\R)$ containing  $-J(-\nu,k)$.  Let $\mathcal H(-\nu,k,t)\vcentcolon=e^{tA}$ with $A\in M_k(\R)$ be such a subgroup, so that $e^A=-J(-\nu,k)$. Note that $e^{tA}$ commutes with the scalars $e^{i\pi t'}\in \C$ and $J(-\nu,k)=e^{i\pi}e^A$. Thus $e^{i\pi t}\mathcal{H}(-\nu,k,t)=e^{i\pi t} e^{tA}$ is a one-parameter subgroup of $\GL_k(\C)$ (not $\GL_k(\R)$) which contains $J(-\nu,k)$. 

\noindent{\bf Case 3: $M(a,b,k)$.}  Finally, consider a matrix $M(a,b,k)$ where $a,b\in \R$ are not both zero, and let $r=\sqrt{a^2+b^2}>0$. We first consider the matrix $Q(a,b)$, which can be written as $rQ(a/r,b/r)$ where $Q(a/r,b/r)$ is a rotation matrix. Consequently, $Q(a/r,b/r)$ lies on the one-parameter subgroup $e^{t K}$ of $\GL_2(\R)$, where $K$ is the matrix $\begin{pmatrix} 0 & -1 \\ 1 & 0 \end{pmatrix}$. Replace the matrix $K\in M_2(\R)$ by a scalar multiple to ensure that $Q(a/r,b/r)=e^K$. Then $Q(a,b)$ lies on the one-parameter subgroup $r^t e^{tK}$ of $\GL_2(\R)$. Note that $e^{tK}$ commutes with the scalars $r^{t'}$ and all the matrices $e^{tK}$ are orthogonal (and, in particular, unitary). 

Now consider the full matrix  $M(a,b,k)$. The nilpotent matrix $N\in M_{2k}(\R)$ defined by \[N= \begin{pmatrix} 0 & I_2 & 0 & \cdots & 0 \\ 0 & 0 & I_2 & \cdots & 0 \\ \vdots & \vdots &\vdots & \ddots & \vdots \\ 0 & 0 & 0 & \cdots & I_2 \\ 0 & 0 & 0 & \cdots & 0\end{pmatrix} \text{ satisfies } e^{tN}= \begin{pmatrix} I & t I & \frac{1}{2!} t^2 I & \cdots & \frac{1}{(k-1)!} t^{k-1} I \\ 0 & I & tI & \cdots & \frac{1}{(k-2)!}t^{k-2} I \\ \vdots & \vdots &\vdots & \ddots & \vdots \\ 0 & 0 & 0 & \cdots & t I \\ 0 & 0 & 0 & \cdots & I \end{pmatrix} \] (where all blocks in these matrices are $2\times 2$). As before, we see that $e^N$ is conjugate via a real invertible matrix to the real Jordan block \[\begin{pmatrix} I & I & 0 & \cdots & 0 \\ 0 & I & I & \cdots & 0 \\ \vdots & \vdots &\vdots & \ddots & \vdots \\ 0 & 0 & 0 & \cdots & I \\ 0 & 0 & 0 & \cdots & I\end{pmatrix}.\] In fact, if $(a_{ij})_{i,j=1}^k$ is any matrix conjugating the first matrix to the second in equation \eqref{eq:conjugatematrices}, then the same matrix with each entry $a_{ij}$ replaced by the $2\times 2$ block $a_{ij} I$ conjugates $e^N$ to the desired $2k\times 2k$ real Jordan block. Denote this conjugating block matrix by $\xi$. Our $2k\times 2k$ real Jordan block is in turn conjugate to \[\begin{pmatrix} I & Q(a,b)^{-1} & 0 & \cdots & 0 \\ 0 & I & Q(a,b)^{-1} & \cdots & 0 \\ \vdots & \vdots & \vdots & \ddots & \vdots \\ 0 & 0 & 0 & \cdots & Q(a,b)^{-1} \\ 0 & 0 & 0 & \cdots & I\end{pmatrix} \text{ via } \zeta = \begin{pmatrix} Q(a,b)^{-k+1} &   &   &   &   \\   & Q(a,b)^{-k+2} &   &   &   \\   &   &   & \ddots &   \\   &   &   &   &   \\   &   &   &  & Q(a,b)^0 \end{pmatrix}. \tag{\textasteriskcentered\textasteriskcentered} \label{eq:conjugatematrices2}\]  We have shown that $\zeta \xi$ conjugates $e^N$ to the first matrix in \eqref{eq:conjugatematrices2}. Define $\mathcal V(a,b,k,t)\vcentcolon= r^t \zeta \xi e^{tN}\xi^{-1} \zeta^{-1}$, and let $\mathcal U(a,b,k,t)$  be the $2k\times 2k$ block diagonal matrix with diagonal blocks $e^{tK}$. Thus, $M(a,b,k)$ lies on the one-parameter subgroup $\mathcal W(a,b,k,t)\vcentcolon = \mathcal U(a,b,k,t) \mathcal V(a,b,k,t)$. 
The matrix $\mathcal{U}(a,b,k,t)$ is orthogonal (in particular unitary), $\mathcal{V}(a,b,k,t)$ is real, and moreover, $\mathcal U(a,b,k,t)$ commutes with $\mathcal V(a,b,k,t')$ for any $t,t'\in \R$. To see this last fact, note that $\mathcal{U}(a,b,k,t)$ commutes with: $r^{t'}$ since it's a scalar; $\zeta$, since $Q(a,b)$ and $e^{tK}$ lie in the same one-parameter subgroup of $\GL_2(\R)$ up to scaling; and $\xi$ and $e^{t'N}$ since $\mathcal{U}(a,b,k,t)$ is block-diagonal and $\xi$ and $e^{t'N}$ are block matrices whose blocks are scalar multiples of $I_2$.

We now show how to combine these one-parameter groups to form the desired one-parameter subgroup containing $\gamma$.   Recall the positive real eigenvalues $\lambda_1,\ldots,\lambda_u$, the negative real eigenvalues $-\nu_1,\ldots,-\nu_v$, and the complex eigenvalues $\mu_1,\ldots,\mu_s$ with $\mu_j=a_j+ib_j$. Denote by $\mathcal G(t)$ the block diagonal matrix with diagonal blocks $\mathcal G(\lambda_1,j_1,t),\ldots, \mathcal G(\lambda_u,j_u,t)$. Denote by $\mathcal I(t)$ the block diagonal matrix with diagonal blocks $e^{i\pi t} \mathcal H(-\nu_1,k_1,t),\ldots, e^{i\pi t} \mathcal H(-\nu_v,k_v,t)$. The matrix $\mathcal I(t)$ is the product of $e^{i\pi t}$ with $\mathcal H(t)$, where $\mathcal H(t)$ is the block diagonal matrix with diagonal blocks $\mathcal H(-\nu_j,k_j,t)$. Finally denote by $\mathcal W(t)$ the block diagonal matrix with diagonal blocks $\mathcal W(a_1,b_1,l_1,t),\ldots, \mathcal W(a_s,b_s,l_s,t)$. We have $\mathcal W(a_j,b_j,l_j,t)=\mathcal U(a_j,b_j,l_j,t) \mathcal V(a_j,b_j,l_j,t)$ for each $j$, and thus $\mathcal W(t)=\mathcal U(t) \mathcal V(t)$ where $\mathcal U(t)$ is block diagonal with diagonal blocks $\mathcal U(a_j,b_j,l_j,t)$ and $\mathcal V(t)$ is block diagonal with diagonal blocks $\mathcal V(a_j,b_j,l_j,t)$. Combining the three cases considered above, we see that $\eta \gamma \eta^{-1}$,  the real Jordan form of $\gamma$,  lies on the one-parameter subgroup \[\begin{pmatrix} \mathcal G(t) & & \\ & \mathcal I(t) & \\ & & \mathcal W(t) \end{pmatrix}.\]

Setting $J=j_1+\cdots+j_u$ and $K=k_1+\cdots+k_v$, we see that this is the product of two commuting one-parameter subgroups \[\begin{pmatrix} I_J & & \\ & e^{i\pi t} I_K & \\ & & \mathcal U(t)\end{pmatrix} \text{ and } \begin{pmatrix} \mathcal G(t) & & \\ & \mathcal H(t) & \\ & & \mathcal V(t)\end{pmatrix}. \] 
The first of these two subgroups is a unitary subgroup of $\GL_n(\C)$ and the second is a subgroup of $\GL_n(\R)$. Finally, set $\delta^t$ to be the conjugate of the (first) unitary one-parameter group by $\eta^{-1}$ and $\epsilon^t$ to be the conjugate of the (second) real one-parameter group by $\eta^{-1}$. Since $\eta$ is a real invertible matrix, $\delta^t$ and $\epsilon^t$ are subgroups of the form specified by the lemma. 
\end{proof}

\subsection{The word metric for the main Heintze group}

In this section we study the action of $G=A\rtimes \Z$ on the quasi-convex subspace $P=\R^n\times \R$ of the Heintze group $\C^n\rtimes_{\gamma^t} \R$. We will show that the action of $G$ on $P$ is cobounded (Lemma~\ref{lem:heintzecobounded}) and that the action $G\curvearrowright P$ is equivalent to the action $G\curvearrowright \Gamma(G,Q_\epsilon(0)\cup \{t^{\pm 1}\})$ (Lemma~\ref{lem:heintzeconfsubset}). This will show that the quasi-parabolic structure determined by the invariant subspace 0 is represented by the action on $P$. In the next subsection we investigate the quasi-parabolic structures corresponding to other invariant subspaces of $\R^n$ and show they  correspond to actions of $G$ on subspaces of quotients of the Heintze group $\C^n\rtimes \R$.

\begin{lem}
\label{lem:heintzecobounded}
The action $G\curvearrowright P$ is cobounded.
\end{lem}

\begin{proof}
 We first claim that the orbit  of the base point $(0,0)\in P$ under the action of $A$ is dense in the subset $\R^n \times \{0\}$ of $P$. To see this, note that $A$ contains the set \[\left\{ \sum_{j=1}^k \gamma^{-j}v_j : k\geq 1 \text{ and } v_j \in \mathscr D \text{ for all } 1\leq j\leq k\right\},\] which is dense in the tile \[T=T(\gamma,\mathscr D) =\left\{ \sum_{j=1}^\infty \gamma^{-j} v_j : v_j \in \mathscr D \text{ for all } j\geq 1\right\}.\] By \cite[Corollary 1.1]{lag_wang} the tile $T$ contains an open set. Hence $A$ is dense in some open subset of $\R^n$. Since the closed subgroup $\overline{A}$ of $\R^n$ contains an open set of $\R^n$,  Bourbaki's structure theorem \cite[Chapter 7, Theorem 2]{bourbaki} implies that  it must in fact be equal to all of $\R^n$. Therefore, $A$ is dense.

Now, every point of $P$ lies at distance $\leq 1$ from a horosphere $\R^n \times \{m\}$ for some $m\in \Z$. Since $t^m$ translates $\R^n\times \{0\}$ to this horosphere, we see that the orbit of $(0,0)$ under $G$ is dense in this horosphere as well. Choosing any $\epsilon>0$, we see that every point of $P$ is distance $\leq 1+\epsilon$ from the orbit of $(0,0)$.
\end{proof}

\begin{lem}
\label{lem:heintzeconfsubset}
We have $(G\curvearrowright P) \sim (G\curvearrowright \Gamma(G,Q_\epsilon(0)\cup \{t^{\pm 1}\}))$ for any $\epsilon>0$.
\end{lem}

\begin{proof}
Since the orbit of the base point $b=(0,0)$ is $2$-dense in $P$, it follows from the Schwarz--Milnor Lemma (Lemma~\ref{lem:MS})  that $G\curvearrowright P$ is equivalent to $G\curvearrowright \Gamma(G,S)$, where $S=\{g\in G : d(b,gb)\leq 5\}$.

Any element $g\in S\subset G$ can be written as $g=xt^j$ with $x\in A$ and $j\in \Z$, so that $g b=(x,j)$. Consider a geodesic $p\colon [0,1]\to P$ from $b$ to $g b$, where  $p(u)=(z(u),s(u))$ for $0\leq u\leq 1$. Since every point on the horosphere $\R^n\times \{j\}$ is distance at least $|j|$ from $b$ and the length of $p$ is at most 5, we must have $|s(u)|\leq 5$ for all $0\leq u\leq 1$. Choose $L\geq 1$ with $\frac{1}{L}\leq \|\gamma^s\|\leq L$ for all $s$ with $|s|\leq 5$. Then we have (denoting derivatives by $\cdot'$), \[d(b,gb)=\ell(p)=\int_0^1 \sqrt{\|\gamma^{-s(u)}z'(u)\|^2+s'(u)^2}du\geq \frac{1}{L} \int_0^1 \sqrt{\|z'(u)\|^2+s'(u)^2}du.\] This lower bound is $\frac{1}{L}$ times the length of $p$ when considered as a path in $\R^n\times \R$ with the Euclidean metric. Thus, \[5\geq d(b,gb)\geq \frac{1}{L} \sqrt{\|x\|^2+j^2},\] and we have both $\|x\|\leq 5L$ and $|j| \leq 5L$. In other words, $Q_{5L}(0) \cup \{t^{\pm 1}\}\preceq S$.

On the other hand, we have $t\in S$. Furthermore, the norm of any element of $Q_\epsilon(0)$ is uniformly bounded in terms of $\epsilon$, and this uniform bound approaches 0 as $\epsilon\to 0$. Choose $\eta$ small enough that every element of $Q_\eta(0)$ has norm bounded by 5. For any $x\in Q_\eta(0)$,  the (Euclidean) straight line path $p$ from $b=(0,0)$ to $xb=(x,0)$ in $\R^n\times \{0\}$ has length equal to  its Euclidean length. Hence $d(b,xb)\leq \|x\|\leq 5$, so that $x\in S$ as well. This shows that $S\preceq Q_\eta(0)\cup \{t^{\pm 1}\}$, and the result now follows from Lemma \ref{lem:nbhdequiv}, as $Q_\eta(0) \sim Q_{5L}(0)$.
\end{proof}

\subsection{Quotients of Heintze groups}
\label{sec:heintzequotient}

The previous section showed that the quasi-parabolic structure corresponding to the invariant subspace $0$ is represented by the action on the quasiconvex subspace $P=\R^n\times \R$ of the Heintze group $\C^n\rtimes_{\gamma^t} \R$.   In this section,  we show that the quasi-parabolic structures corresponding to other invariant subspaces of $\R^n$ correspond to actions of $G$ on subspaces of \emph{quotients} of the Heintze group $\C^n\rtimes \R$. The proofs in this section are almost identical to those in the previous section, and we simply note where there are differences.

Consider a $\gamma$--invariant subspace $V$ of $\R^n$. The complexification $W$ of $V$ is a subspace of $\C^n$ invariant under $\gamma$ and containing $V$. The matrix $\gamma$ induces a linear automorphism $\overline{\gamma}$ of the (real) quotient vector space $\R^n/V$. The complexification of $\R^n/V$ is the (complex) quotient vector space $\C^n/W$, and the action of $\overline{\gamma}$ on $\R^n/V$ extends to a linear automorphism of $\C^n/W$. Choosing a basis for $\R^n/V$ yields a representative of $\overline{\gamma}$ as a real matrix.  The Jordan normal form of $\overline{\gamma}$ is obtained by taking sub-blocks of the Jordan blocks of $\gamma$, and so, as before, it follows that $\overline{\gamma}$ is expanding.  Unlike the case of $V=0$ discussed in the previous section, the matrix $\overline{\gamma}$ may not be integral with respect to a particular basis of $\R^n/V$. However, this will not cause any problems in what follows.

Choose a (real-valued) inner product $\langle \cdot, \cdot \rangle$ on $\R^n/V$, and extend it to a complex-valued inner product on $\C^n/W$. Thus, the notion of a unitary matrix acting on $\C^n/W$ is well-defined. Note that any real matrix acting on $\C^n/W$ preserves $\R^n/V$. We first show that $\overline{\gamma}$ lies on a particular one-parameter subgroup of $\GL(\C^n/W)$.

\begin{lem}
\label{lem:oneparamquotient}
There is a one-parameter subgroup $\overline{\gamma}^t$ of $\GL(\C^n/W)$ of the form $\overline{\gamma}^t=\overline{\delta}^t\overline{\epsilon}^t$ where $\overline{\delta}^t$ is a one-parameter group conjugate into the subgroup of unitary matrices, which commutes with the one-parameter group of real matrices $\overline{\epsilon}^t$.
\end{lem}

\begin{proof}
The proof is identical to that of Lemma \ref{lem:unitaryrealgroup}, which did not rely on $\gamma$ being integral.
\end{proof}

Form the Heintze group $(\C^n/W)\rtimes_{\overline{\gamma}^t} \R$, and equip it with the left-invariant metric induced by the real-valued inner product $\langle \langle \cdot,\cdot\rangle \rangle$ on the tangent space $(\C^n/W)\times \R$ at $(0,0)$ defined by $\langle \langle (v,s),(w,t)\rangle \rangle \vcentcolon= \langle \Re v, \Re w\rangle + \langle \Im v, \Im w\rangle + st$.  The natural homomorphisms $A\to \C^n\to \C^n/W$ and $\Z\to \R$ induce a homomorphism $G=A\rtimes_\gamma \Z\to (\C^n/W)\rtimes_{\overline{\gamma}^t} \R$, so that $G$ acts on this Heintze group isometrically by left translations. Denote the homomorphism $A\to \C^n/W$ by $\pi$.

Define the subspace $\overline{P}\vcentcolon=(\R^n/V)\times \R$ of the Lie group $(\C^n/W)\rtimes_{\overline{\gamma}^t} \R$, which is preserved by $G$.

\begin{prop}\label{prop:olPqiemb}
The subspace $\overline{P}$ is  quasi-isometrically embedded, and hence quasi-convex. Further, the action $G\curvearrowright \overline{P}$ is cobounded. 
\end{prop}

\begin{proof}
Since $\C^n/W$ is the complexification of the real vector space $\R^n/V$, there is a linear map $\C^n/W\to \R^n/V$ defined by $z\mapsto \Re z$. This defines a retraction $\C^n/W\to P$ given by  $(z,t)\mapsto (\Re z,t)$. The fact that this map is Lipschitz  follows from Lemma~\ref{lem:oneparamquotient} just as Proposition~\ref{prop:lipschitzretraction} follows from Lemma~\ref{lem:unitaryrealgroup}.

To prove coboundedness, note that since $A$ is dense in $\R^n$, the orbit of $A$ is dense in $\R^n/V \times \{0\}$. The rest of the proof follows exactly as in Lemma \ref{lem:heintzecobounded}. In particular, the orbit of $G$ is 2-dense.
\end{proof}

Proposition \ref{prop:olPqiemb} implies that $\overline{P}$ is a hyperbolic metric space, as quasi-convex subspaces of hyperbolic spaces are hyperbolic (\cite[Theorem H.1.9]{bridson_haefliger}). We are now ready to show the association between classes in $\mc P_-(G)$ and actions on Heintze groups.

\begin{lem}
\label{lem:heintzerepresentative}
 For any $\epsilon>0$, we have $(G\curvearrowright \overline{P})\sim (G\curvearrowright \Gamma(G,Q_\epsilon(V)\cup\{t^{\pm 1}\}))$.
\end{lem}

\begin{proof}
Let  $b=(0,0)$ be the basepoint of $\overline{P}$. Let $S'=\{g\in G:d(b,gb)\leq 5\}$, so that $G\curvearrowright \overline P \sim G\curvearrowright \Gamma(G,S')$ by the Schwarz-Milnor Lemma (Lemma~\ref{lem:MS}). We consider two norms on $\C^n/W$. The first is the norm $\|\cdot\|_{\operatorname{Euc}}$ induced by the inner product that we chose on $\C^n/W$ at the beginning of this subsection. The second is the quotient norm $\|\cdot\|_{\operatorname{Quo}}$ induced by the Euclidean norm on $\C^n$, i.e., $\|x\|_{\operatorname{Quo}}=d(x,W)$. Since any norms on a finite-dimensional vector space are equivalent, there is a number $F\geq 1$ such that \[\frac{1}{F} \|\cdot\|_{\operatorname{Quo}} \leq \|\cdot\|_{\operatorname{Euc}} \leq F \|\cdot\|_{\operatorname{Quo}}.\]

Consider an element $g=xt^j \in S'\subset G$ with $x\in A$ and $j\in \Z$. Then $g\cdot b=(\pi(x),j)$, where $\pi$ is the homomorphism $A\to \C^n/W$. The first part of the proof of Lemma \ref{lem:heintzeconfsubset} yields a number $L'>0$ such that $\|\pi(x)\|_{\operatorname{Euc}}\leq 5L'$ and $|j| \leq 5L'$. Thus $\|\pi(x)\|_{\operatorname{Quo}}\leq 5L'F$, and $x$ lies in the $5L'F$-neighborhood of $W$. Since $x$ is real, its orthogonal projection to $W$ lies in $W\cap \R^n=V$, and $x\in Q_{5L'F}(V)$. Therefore $Q_{5L'F}(V)\cup \{t^{\pm 1}\}\preceq S'$.

On the other hand,  $t\in S'$. The distance from any element of $Q_\epsilon(V)$ to $V$ is uniformly bounded in terms of $\epsilon$, and the bound approaches 0 as $\epsilon\to 0.$ Choose $\eta$ small enough that every point of $Q_\eta(V)$ has distance $\leq 5/F$ from $V$. The second part of the proof of Lemma \ref{lem:heintzeconfsubset} shows that for any $x\in Q_{\eta}(V)$, $d(b,xb)\leq \|\pi(x)\|_{\operatorname{Euc}}\leq F\|\pi(x)\|_{\operatorname{Quo}}\leq 5$. Thus $x\in S'$ as well. This shows that $S'\preceq Q_{\eta}(V)\cup \{t^{\pm 1}\}$, and the result now follows from Lemma \ref{lem:nbhdequiv}.
\end{proof}

As $\R^n/V$ and $\C^n/W$ are isomorphic to $\R^k$ and $\C^k$, respectively, for some $k\leq n$, Theorem~\ref{thm:heintzegroup} follows immediately from Lemma~\ref{lem:heintzerepresentative}.

\subsection{Proof of Theorems \ref{thm:char=min}
and \ref{thm:P-description}}

In this section we prove the remaining main results from the introduction. 

\begin{proof}[Proof of Theorem \ref{thm:P-description}]

Consider an expanding matrix $\gamma\in M_n(\Z)$  with equal minimal and characteristic polynomial, and consider the group $G=G(\gamma)$. By Proposition \ref{prop:generalstructure} and Lemma \ref{lem:abbycyclicabelianization}, $\Hl(G)$ consists of the single elliptic structure plus two lattices $\mathcal P_+(G)$ and $\mathcal P_-(G)$ meeting in the single lineal structure. By Proposition~\ref{prop:P-invar}, $\mathcal P_-(G)$ is isomorphic to the poset of invariant subspaces.

Consider an element of $\mathcal P_-(G)$. It has the form $[Q_1(V)\cup \{t^{\pm 1}\}]$ for an invariant subspace $V$ of $\gamma$. By the discussion in Section \ref{sec:heintzequotient} and Lemma \ref{lem:heintzerepresentative}, the hyperbolic structure $[Q_1(V)\cup \{t^{\pm 1}\}]$ is represented by an action on the subset $(\R^n/V)\times \R$ of a Heintze group $(\C^n/W)\rtimes_{\overline{\gamma}^t}\R$.
\end{proof}

Theorem~\ref{thm:char=min} now follows immediately from  Theorems~\ref{thm:P-description} and \ref{thm:Thm1.2P_+}.

\appendix

\section{Appendix: Classification of ideals and factorization of formal power series  \\ {\normalsize \normalfont \it Co-authored with Sam Payne}}\label{sec:appendix} 

The goal of this section is to prove the following theorem, which is used to prove Theorem~\ref{thm:Thm1.2P_+}.

\begin{restatable}{thm}{idealclassification}
\label{thm:idealclassification}
Let $f$ be a monic polynomial in $\Z[x]$ with constant term not lying in $\{-1,0,1\}$. Then the $(x+(f))$-adic completion of $\Z[x]/(f)$ is isomorphic to $\Z[[x]]/(f)$. Moreover, every ideal of the localization $x^{-1}\Z[[x]]/(f)$ is generated by a divisor of $f$ in $\Z[[x]]$. The poset of ideals of $x^{-1}\Z[[x]]/(f)$ is isomorphic to the poset of divisors of $f$. Equivalently, the poset of ideals of $\Z[[x]]/(f)$ considered up to equivalence is isomorphic to the poset of divisors of $f$.
\end{restatable}

Consider the ring $R=\Z[x]/(f)$ and its element $\gamma=x+(f)$. First we give an alternative description of the $(\gamma)$-adic completion $\widehat{R}$ of $R$, with no assumptions on $f$.

\begin{lem}
\label{lem:completionquotient}
The $(\gamma)$-adic completion of $\Z[x]/(f)$ is isomorphic to $\Z[[x]]/(f)$.
\end{lem} 

\begin{proof}
The  sequence \[0\to (f)\to \Z[x]\to R\to 0\] is an exact sequence of finitely generated $\Z[x]$-modules. Since $\Z[x]$ is Noetherian, the sequence of $(x)$--adic completions \[0\to \widehat{(f)}\to \widehat{\Z[x]}\to \widehat{R} \to 0\] is exact by \cite[Proposition 10.12]{atiyah_macdonald}. Since $x$ acts on $R$ in the same way as $\gamma$, the $(x)$-adic completion $\widehat{R}$ is just the usual $(\gamma)$-adic completion. Moreover, $\widehat{\Z[x]}$ is the formal power series ring $\Z[[x]]$. The $\Z[x]$-module homomorphism $(f)\to \widehat{(f)}$ induces a $\Z[[x]]$-module isomorphism $\Z[[x]]\otimes_{\Z[x]} (f)\to \widehat{(f)}$ by \cite[Proposition 10.13]{atiyah_macdonald}, whose image is the ideal of $\Z[[x]]$ generated by $f$.
\end{proof}

For the rest of the section, we consider quotients of $\Z[[x]]$ by monic polynomials with non-zero constant terms that are not units in $\Z$. To prove Theorem \ref{thm:idealclassification}, we must classify the ideals of $x^{-1}(\Z[[x]]/(f))$. Note that localization commutes with taking quotients (\cite[Corollary 3.4(iii)]{atiyah_macdonald}). Therefore $x^{-1}(\Z[[x]]/(f))$ is canonically identified with the quotient of the Laurent series ring $x^{-1}\Z[[x]]$ by its principal ideal $(f)$, and we denote this ring by $x^{-1}\Z[[x]]/(f)$.

Recall that elements $a$ and $b$ in a ring are \emph{associates} if $b=ua$ for some unit $u\in R$. We first consider quotients $\Z[[x]]/(g)$, where $g$ is an irreducible power series that is neither associate to any prime in $\Z$ nor to the monomial $x$. Thus the constant term of $g$ is equal to a power of a prime $p\in\Z$ times $\pm 1$.  By \cite[Proposition 3.1.3]{mcdonough} (see also \cite[Theorem 1.4]{elliott}), there is an isomorphism $\Z[[x]]/(g)\to \Z_p[\alpha]$ sending $x$ to $\alpha$, where 
$\alpha$ is a root of an irreducible polynomial $w(x)\in \Z_p[x]$ such that \begin{equation} \label{eq:w}
w(x)=p u(x)+x^n \text{ where } u\in\Z_p[x]  \text{ and } \deg(u)\leq n-1.
\end{equation}

\begin{lem}
\label{lem:localizationfield}
Let $g\in \Z[[x]]$ be an irreducible power series that is neither associate to $x$ nor to any prime in $\Z$. Then $x^{-1}\Z[[x]]/(g)$ is a field.
\end{lem}

\begin{proof}
Consider the prime $p\in \Z$ coming from the constant term of $g$ and the polynomial $w\in \Z_p[x]$ as in (\ref{eq:w}).  Then $x^{-1}\Z[[x]]/(g)$ is isomorphic to $\alpha^{-1} \Z_p[\alpha]$. Since $\alpha^n=-pu(\alpha)$, it follows that $p$ is a unit in $\alpha^{-1}\Z_p[\alpha]$. There is a natural injection from $\alpha^{-1}\Z_p[\alpha]$ to the finite field extension $\Q_p(\alpha)$: $\Q_p(\alpha)=\Q_p[x]/(w)$, and $\Z_p[x]/(w)$ includes into $\Q_p[x]/(w)$ in such a way that $x+(w)$ is a unit. We claim that this homomorphism is also surjective and thus an isomorphism of fields. We may write an element of $\Q_p(\alpha)$ as $h(\alpha)$ with $h\in \Q_p[x]$. Multiplying by a high enough power of $p$ to clear the denominators of the coefficients of $h$, we have $p^ih(x)\in \Z_p[x]$ for some $i$. Since $p$ is a unit in $\alpha^{-1}\Z_p[\alpha]$, we have $p^{-i}(p^i h(\alpha))\in \alpha^{-1}\Z_p[\alpha]$, and the image of this element is $h(\alpha)$. Thus, the natural map $\alpha^{-1}\Z_p[\alpha]\to \Q_p(\alpha)$ is surjective as desired.
\end{proof}

Next, we consider quotients by powers of irreducibles $\Z[[ x]]/(g^j)$, where $g$ is neither associate to $x$ nor to any prime $p\in \Z$.  Let $\bar h \vcentcolon= h + (g^j)$ denote the image of a power series $h \in \Z[[ x]]$ in this quotient.

\begin{lem}
\label{lem:primepowerideals}
Let $g\in \Z[[ x]]$ be an irreducible power series that is neither associate to $x$ nor to any prime in $\Z$. If $j \geq 2$, then $(g)$ is the unique non-zero prime ideal of $x^{-1}\Z[[ x]]/(g^j)$.  Moreover, the ideals of $x^{-1}\Z[[ x]]/(g^j)$ are exactly the ideals $({g}^i)$ generated by powers ${g}^i$ for $0\leq i\leq j$.
\end{lem}

\begin{proof}
The quotient of $x^{-1}\Z[[ x]]/(g^j)$ by $(g)$ is a field by Lemma~\ref{lem:localizationfield}. Thus $(g)$ is maximal, and hence prime. Moreover, we have $(g)^j=0$. Let $\frak p$ be a non-zero prime ideal of $x^{-1}\Z[[ x]]/(g^j)$. Then $(g)^j\subset \frak p$, and, since $\frak p$ is prime, we must have $(g)\subset \frak p$. As $(g)$ is maximal, it follows that $\frak p = (g)$.  Thus $(g)$ is the unique prime ideal of $x^{-1}\Z[[ x]]/(g^j)$.  This proves the first statement.

For the moreover statement, note that any element of $x^{-1}\Z[[ x]]/(g^j)$ can be written as $\bar h/ \bar{x}^k$ for some  $h\in \Z[[ x]]$.  If $g$ does not divide $h$, then $\bar{h}/\bar{x}^k$ is a unit. To see this, suppose for contradiction that $\bar{h}/\bar{x}^k\in (g)$. Then $\bar{h}/\bar{x}^k=\bar{g}\bar{h'}/\bar{x}^\ell$ for some $h'\in \Z[[ x]]$ and $\ell\in \Z_{\geq 0}$, and so $\bar{h}\bar{x}^\ell=\bar{g}\bar{h'}\bar{x}^k$. In other words, $hx^\ell+(g^j)=gh'x^k+(g^j)$ so that $hx^\ell=gh'x^k+g^j h''$ for some $h''\in \Z[[ x]]$. However, $g$ divides the right hand side of this equation while it does not divide the left, which is a contradiction. Thus, if $g$ does not divide $h$ then $\bar{h}/\bar{x}^k$ does not lie in $(g)$, so it is a unit.

Given any $\bar{h}/\bar{x}^k \in x^{-1}\Z[[ x]]/(g^j)$, we may write $h=h'g^i$, where $i\geq 0$ and $h'\in \Z[[x]]$ is not divisible by $g$. Then $\bar{h}/\bar{x}^k=(\bar{h'}/\bar{x}^k)(\bar{g}^i/1)$, and, by the previous paragraph, $\bar{h'}/\bar{x}^k$ is a unit. In other words, $\bar{h}/\bar{x}^k$ is associate to $\bar{g}^i$. Consider an ideal $\frak a$, and choose $i$ to be the minimum such that there is an element of $\frak a$ associate to $\bar{g}^i$. Then  $\frak a\subset (\bar{g}^i)$, but also $\bar{g}^i\in \frak a$. That is, $\frak a=(\bar{g}^i)=(\bar{g})^i$.
\end{proof}

A localization of a unique factorization domain is a unique factorization domain. Thus $x^{-1}\Z[[x]]$ is a unique factorization domain, and so its prime and irreducible elements coincide.

\begin{lem}
\label{lem:coprime}
Let $g\in\Z[[x]]$ be a prime power series which is neither associate to $x$ nor to a prime $p\in \Z$. Then $g$ is prime as an element of $x^{-1}\Z[[x]]$. Moreover, consider two prime power series $g_1,g_2\in \Z[[x]]$, neither of which is associate to $x$ or to a prime $p\in \Z$. If $g_1$ and $g_2$ are not associate to each other in $\Z[[x]]$, then $g_1^{n_1}$ and $g_2^{n_2}$ are coprime in $x^{-1}\Z[[x]]$ for any $n_1,n_2\geq 1$.
\end{lem}

\begin{proof}
If $g$ is a unit in $x^{-1}\Z[[x]]$, then there is an element $h/x^k\in x^{-1}\Z[[x]]$ with $gh/x^k=1$ in $x^{-1}\Z[[x]]$. Hence $gh=x^k$. However, this contradicts  that $g$ does not divide the prime power $x^k$, and so $g$ is not a unit.

Consider a product of two elements of $x^{-1}\Z[[x]]$ that is equal to $g$; that is, $g=(h_1/x^{k_1})(h_2/x^{k_2})$ where $h_1,h_2\in \Z[[x]]$. Then  $gx^{k_1}x^{k_2}=h_1h_2$ in $\Z[[x]]$, and so $g$ divides exactly one of the $h_i$. Say $g$ divides $h_1$. Then $g$ and $h_1/x^{k_1}$ divide each other in the domain $x^{-1}\Z[[x]]$, and therefore they are associates while $h_2/x^{k_2}$ is a unit. This proves that $g$ is irreducible in $x^{-1}\Z[[x]]$ and hence also prime.

For the last statement, consider two non-associate prime power series $g_1,g_2\in \Z[[x]]$ with the desired properties. By the proof of Lemma \ref{lem:primepowerideals}, an element of $x^{-1}\Z[[x]]/(g_2^{n_2})$ is either a unit or nilpotent. Moreover, an element $\overline{h}/\overline{x}^k$ with $h\in \Z[[x]]$ is nilpotent exactly if $g_2$ divides $h$ in $\Z[[x]]$. The element $\overline{g_1}^{n_1} \in x^{-1}\Z[[x]]/(g_2^{n_2})$ is a unit, since $g_2$ does not divide $g_1^{n_1}$ in $\Z[[x]]$. Thus, $\overline{g_1}^{n_1}(\overline{h}/\overline{x}^k)=\overline{1}$ for some $h\in \Z[[x]]$. In other words, $g_1^{n_1}h+(g_2^{n_2})=x^k+(g_2^{n_2})$, and hence $ g_1^{n_1}h=x^k+g_2^{n_2}h' $ for some $ h'\in \Z[[x]].$ Therefore $1=g_1^{n_1}(h/x^k)-g_2^{n_2}(h'/x^k)\in (g_1^{n_1},g_2^{n_2})$.
\end{proof}

We are now ready to prove Theorem~\ref{thm:idealclassification}.
\begin{proof}[Proof of Theorem \ref{thm:idealclassification}]
Consider a monic polynomial $f\in \Z[x]$ whose constant term does not lie in $\{-1,0,1\}$ and  its prime factorization $f=uf_1^{n_1}\cdots f_r^{n_r}$ in $\Z[[x]]$, where $u\in \Z[[x]]$ is a unit and each $f_i$ is prime. Since the constant term of $f$ is not $\pm 1$, the polynomial $f$ is not a unit in $\Z[[x]]$. Thus there is at least one prime power series in its prime factorization. The prime power series of $\Z[[x]]$ are either  associate to $x$, associate to a prime in $\Z$, or neither associate to $x$ nor to any prime in $\Z$. A power series of the last type has constant term equal to a power of a prime in $\Z$ times $\pm 1$.
A prime of the first type has constant term 0, while a prime of the second type has all its coefficients divisible by $p$. Since $f$ has non-zero constant term, the same holds for each $f_i$, and so no $f_i$ is associate to $x$. As $f$ is monic, the coefficients of $f$ are not all divisible by a common prime integer $p$, and so the same holds for each $f_i$, and no $f_i$ is associate to a prime integer $p$.  In particular, this implies that the constant term of each $f_i$ is a power of a prime in $\Z$ times $\pm 1$.

By Theorem \ref{lem:localideals}, the poset of ideals of $\Z[[x]]/(f)$ considered up to multiplication by $\gamma=x+(f)$ is isomorphic to the poset of ideals of the localization $x^{-1}\Z[[x]]/(f)$. We have $(f)=(f_1^{n_1})\cdots(f_r^{n_r})$ in $x^{-1}\Z[[ x]]$, and by Lemma \ref{lem:coprime}, the ideals $(f_i^{n_i})$ are pairwise coprime. Thus, by the Chinese Remainder Theorem, 
\[
x^{-1}\Z[[ x]]/(f) \cong 
x^{-1}\Z[[ x]]/(f_1^{n_1}) \times \cdots \times x^{-1}\Z[[ x]]/(f_r^{n_r}).
\] An ideal of this ring has the form $\frak a_1 \times \cdots \times \frak a_r$, where each $\frak a_i$ is an ideal of $x^{-1}\Z[[x]]/(f_i^{n_i})$. By Lemma \ref{lem:primepowerideals}, an ideal must actually have the form $(f_1^{i_1}) \times \cdots \times (f_r^{i_r})$, where $0\leq i_j\leq n_j$ for each $j$.

The divisors of $f$ in $\Z[[x]]$ up to associates are $f_1^{i_1} \cdots f_r^{i_r}$, and the map \[f_1^{i_1} \cdots f_r^{i_r} \mapsto (f_1^{i_1}) \times \cdots \times (f_r^{i_r})\] is a bijection from the poset of divisors of $f$ with the order of divisibility to the poset of ideals of $x^{-1}\Z[[x]]/(f)$. This bijection is order-reversing, and thus the poset of ideals of $x^{-1}\Z[[x]]/(f)$ is isomorphic to the opposite of the poset of divisors of $f$ in $\Z[[x]]$. The poset of divisors of $f$ is isomorphic to $\Div(n_1,\ldots,n_r)$, which is isomorphic to its own opposite. Therefore, the poset of ideals of $x^{-1}\Z[[x]]/(f)$ is isomorphic to $\Div(n_1,\ldots,n_r)$ and to the poset of divisors of $f$ in $\Z[[x]]$.
\end{proof}

\section{Appendix: Pseudo-valuations and confining subsets} \label{sec:pseudoval}

The definition of a confining subset in this paper (Definition \ref{def:confining}) is only for groups with abelianizations of rank 1.
Nonetheless, groups of the more general form $H\rtimes\Z^n$, for $H$ an abelian group and $n \geq 1$, were considered by the authors in \cite{ABR}, in which an analogous concept of confining subsets was defined and some of the results of \cite{Amen} were generalized.
In this section, we work in this more general setting and introduce \emph{pseudo-valuations}, which  are functions that are a slight weakening of valuations. 
The main result of this section is the following; the terms in the statement are defined immediately after. 

\begin{thm}\label{thm:PvalandConfining}
Let $G=H\rtimes_\alpha \Z^n$ where $H$ is abelian. Given a homomorphism $\rho\colon \Z^n \to \R$, the poset of equivalence classes of pseudo-valuations on $H$ subordinate to $\rho$ is isomorphic to the opposite of the poset of equivalence classes of subsets of $H$ which are confining under $\alpha$ with respect to $\rho$.
\end{thm}

We start by giving the general definition of a confining subset and then define pseudo-valuations  subordinate to $\rho$. Let $H$ be abelian, and let $G= H\semi_\alpha \Z^n$, where $\alpha\colon \Z^n \to \Aut(H)$ is a fixed homomorphism. An element $z\in \Z^n$ acts on $H$ by conjugation via $zhz^{-1}=\alpha(z)(h)$. Fix a (non-zero) homomorphism $\rho\colon \Z^n \to \R$.  

\begin{defn} \label{def:higherconfining}
A symmetric subset $Q$ of $H$ is  \emph{confining under $\alpha$ with respect to $\rho$} if the following hold:
\begin{itemize}
\item[(a)] For all $z\in \Z^n$ with $\rho(z)\geq 0$, $\alpha(z)(Q)\subseteq Q$.
\item[(b)] For each $h\in H$, there exists $z\in \Z^n$ such that $\alpha(z)(h)\in Q$.
\item[(c)] There exists $z_0\in\Z^n$ such that $\alpha(z_0)(Q+Q)\subseteq Q$.
\end{itemize}
If there exists $z\in\Z^n$ with $\rho(z) >0$ such that $\alpha(z)(Q)\subsetneq Q$, then $Q$ is \emph{strictly confining}.
\end{defn}

\begin{defn}\label{pval} A \emph{pseudo-valuation subordinate to $\rho$} is a function $\pval \colon H \to \R \cup \{+\infty \}$ satisfying:
\begin{enumerate}[(PV1)]
\item there exists a constant $\lambda\geq 0$ such that $\pval(h_1 + h_2) \geq \op{min}\{ \pval(h_1), \pval(h_2) \} - \lambda$ for all $h_1, h_2 \in H$;
\item$ \pval(h) = \pval(-h)$ for all $h \in H$; and
\item $\pval( \alpha(z)(h)) = \rho(z) + \pval(h)$ for all $z \in \Z^n$ and $h \in H$.
\end{enumerate} 
\end{defn} 

\begin{rem} \label{rem:pval0} Since the image of $\rho$ is unbounded in $\R$ and $\pval(\alpha(z)(0))=\pval(0)$ for all $z\in \Z^n$, it follows from condition (PV3) above that $\pval(0) = + \infty$. \end{rem}

There is an equivalence relation on the set of homomorphisms $\Z^n \to \R$ given by $\rho_1 \sim \rho_2$ if and only if  $\rho_1$ and $\rho_2$ are positive scalar multiples of each other. 

Fix a generating set $\{t_1,\dots, t_n\}$ of $\Z^n$.  We construct a (possibly infinite) generating set of $\Z^n$ as follows.  Fix a constant $C_\rho>0$  such that $\rho(t_i) \in [-C_\rho, C_\rho]$ for all $i \in \{1,2,\dots,n\}$ and there exists $y\in \Z^n$ such that $\rho(y)=C_\rho$, and let 
\begin{equation}\label{eqn:Zrho}
Z _\rho= \{z \in \Z^n : |\rho(z)| \leq C_\rho \}.
\end{equation}

Suppose that $Q\subseteq H$ is confining under $\alpha$ with respect to $\rho$. It is straightforward to check that $Q\cup Z_\rho$ is symmetric, and $Q\cup Z_\rho$ generates $G$ by Definition \ref{def:higherconfining}(b). We denote the word norm on $G$ with respect to $Q\cup Z_\rho$ by $\|\cdot\|_{Q\cup Z_\rho}$. If $Q_1$ and $ Q_2$ are two confining subsets with respect to $\rho$, then we say that $Q_1 \preceq_\rho Q_2$  if $\displaystyle \op{sup}_{q_2 \in Q_2}\| q_2 \|_{Q_1 \cup Z_\rho} < \infty$ or, equivalently, $[Q_1 \cup Z_\rho] \preccurlyeq [Q_2 \cup Z_\rho]$ as generating sets. The confining subsets are equivalent, denoted $Q_1 \sim_\rho Q_2$,  if $Q_1 \preceq_\rho Q_2 $ and $Q_2 \preceq_\rho Q_1 $. 

We also consider an equivalence relation on pseudo-valuations. 

\begin{defn} Suppose that $\pval, \widetilde{u}$ are two pseudo-valuations on $H$ subordinate to $\rho$. We say that $\pval \preceq_\rho \widetilde{u}$ if there exists a constant $K\in \R$ such that $ \pval(h) \leq \tilde{u}(h) +K$ for all $h \in H$. The preorder $\preceq_\rho$ induces an equivalence relation $\sim_\rho$ as usual via $\pval \sim_\rho \widetilde{u}$ if $\pval \preceq_\rho \widetilde{u}$ and $\widetilde{u}\preceq_\rho \pval$.
\end{defn}

The first result of this section describes how to obtain a confining subset from a pseudo-valuation.  
 Define a map $\conf$ from the set of pseudo-valuations subordinate to $\rho$ to subsets of $H$ as follows: 
\begin{equation}\label{eqn:conf}
\conf(\pval)=\{h \in H : \pval(h) \geq 0 \}.
\end{equation}

\begin{lem}\label{frompvaltoq} Let $\pval$ be a pseudo-valuation subordinate to $\rho$. 
Then $\conf(\pval)$ is symmetric and confining under $\alpha$ with respect to $\rho$. \end{lem}

\begin{proof} Set $Q=\conf(\pval)$. Note that (PV2) ensures that $Q$ is symmetric. 
We verify the conditions of Definition~\ref{def:higherconfining}. Let $q \in Q$ and $z \in \Z^n$ be such that $\rho(z) \geq 0$. Then $\pval(q) \geq 0$, and thus $\pval(\alpha(z)(q)) = \rho(z) + \pval(q) \geq 0$ and $\alpha(z)(q) \in Q$. This verifies Definition \ref{def:higherconfining}(a).
To show Definition~\ref{def:higherconfining}(b) holds,  let $h \in H$. By choosing $z \in \Z^n$ such that $\rho(z) + \pval(h) \geq 0$, we ensure that $\pval(\alpha(z)(h)) \geq 0$.  Thus $\alpha(z)(h) \in Q$, as desired.
For Definition \ref{def:higherconfining}(c), it follows from (PV3) that $\pval$ is at least $ -\lambda$ on $Q+Q$. Taking $z_0\in \Z^n$ with $\rho(z_0)\geq \lambda$ yields $\alpha(z_0)(Q+Q)\subset Q$.
\end{proof} 

Conversely, define a map $\Pmap$ from subsets of $H$ which are confining under $\alpha$ with respect to $\rho$ to the set of functions $H\to \R\cup\{+\infty\}$ by 
\begin{equation}\label{eqn:pmap}
\Pmap(Q)(h) = - \op{inf} \{ \rho(z) : z\in 
\Z^n \textrm{ and } \alpha(z)(h) \in Q \}.
\end{equation}

\begin{lem}\label{fromqtopval} Let $Q\subseteq H$ be a subset of $H$ which is confining under $\alpha$ with respect to $\rho$.  Then $\Pmap(Q) $ is a pseudo-valuation subordinate to $\rho$. \end{lem} 
\begin{proof} Let $\Pmap(Q)=\pval$.  Note that $\pval$ always takes values in $\R \cup \{+ \infty \}$, as $\pval(h) = - \infty$ is impossible by Definition \ref{def:higherconfining}(b). We verify the conditions from Definition \ref{pval}. 

Condition (PV1) is equivalent to the existence of $\lambda \geq 0$ such that the following holds for every $h_1,h_2\in H$: $$ \op{inf}\{ \rho(z) : \alpha(z)(h_1 + h_2) \in Q\} \leq \op{max}\{i_1 ,i_2\} + \lambda,$$ where $i_j = \op{inf}\{ \rho(z) : \alpha(z)(h_j) \in Q\}$ for $j =1,2$. We may assume, without loss of generality, that $i_1 \leq i_2$.

Let $z \in \Z^n$ be such that $\alpha(z)(h_2) \in Q$. As $i_1 \leq i_2 \leq \rho(z)$, we also have $\alpha(z)(h_1) \in Q$.
Thus $\alpha(z)(h_1 +h_2) \in Q+ Q$. Let $z_0$ be as in Definition \ref{def:higherconfining}(c). We may assume $\rho(z_0) \geq 0$, and so  $\alpha(z_0z)(h_1 +h_2) \in Q$. Thus \[\op{inf}\{ \rho(z) : \alpha(z)(h_1 + h_2) \in Q\} \leq \rho(z_0z) = \rho(z) + \rho(z_0).\]
Setting $\lambda = \rho(z_0)$ and taking the infimum over all $z$ such that $\alpha(z)(h_2) \in Q$, we obtain $\op{inf}\{ \rho(z) : \alpha(z)(h_1 + h_2) \in Q\} \leq i_2+ \lambda$, establishing (PV1). 

To check condition (PV2), note that $\alpha(z)(h) \in Q$ if and only if $\alpha(z)(-h)\in Q$ since $Q$ is symmetric. This shows that $\pval(h)=\pval(-h)$.

Finally, for a fixed $z \in \Z^n$ and $h \in H$, we have that
$\pval(\alpha(z)(h)) = - \op{inf} \{ \rho(z') : \alpha(z')(\alpha(z)(h)) \in Q \}  = - \op{inf} \{ \rho(z') : \alpha(z'z)(h)) \in Q \}. $
Let $z'z = y$. Then $z' = yz^{-1}$, and  \[ \op{inf} \{ \rho(z') : \alpha(z'z)(h) \in Q \}
= \op{inf} \{ \rho(y) -\rho(z) : \alpha(y)(h) \in Q \} = \op{inf} \{ \rho(y) : \alpha(y)(h) \in Q \} - \rho(z).\] Taking the negative infimum yields $\pval(\alpha(z)(h)) = \pval(h) + \rho(z)$, which proves condition (PV3).
\end{proof}

Thus far, we have a map $\conf$ from pseudo-valuations subordinate to $\rho$ to confining subsets with respect to $\rho$, and a map $\Pmap$ in the opposite direction. Since we have equivalence relations on the confining subsets and on pseudo-valuations, our next goal is to ensure that these maps descend to the equivalence classes. We first show that $\Pmap$ is order-reversing, and as a consequence, respects equivalence. For this and other results proved in this section, we utilize the following alternative characterization of the order on confining subsets.

\begin{lem}\label{altdescofconf} Suppose that $Q_1$ and $Q_2$ are confining subsets with respect to $\rho$. Then $Q_2 \preceq_\rho Q_1 $ if and only if there exists $\mu \in \R$ such that $\alpha(z)(Q_1) \subset Q_2$ for all $z \in \Z^n$ such that $\rho(z) \geq \mu$.  
\end{lem}

\begin{proof} We first show the forward implication. Let $q \in Q_1$. Since $Q_2  \preceq_\rho Q_1$, there is a constant $K_1$ such that $\|q\|_{Q_2 \cup Z_\rho} \leq K_1$ for all $q \in Q_1$. It follows from \cite[Lemma 3.6]{ABR} that there is a geodesic representative for $q$ in $Q_2 \cup Z_\rho$ of the form $w^{-1}s_1s_2\cdots s_mw$ for some $w \in \Z^n$, some $ s_i \in Q_2$, and some $m \in \N$ such that $\|w\|_{Q_2\cup Z_\rho}\leq K_1$ and $m\leq K_1$. By the definition of $Z_\rho$, $|\rho(w)| \leq K_1 C_\rho$.

Since $w q w^{-1} = s_1s_2 \cdots s_m \in Q^{K_1}_2$, we have $\alpha(w)(q) \in Q^{K_1}_2$. Letting $z_0$ be the element from Definition \ref{def:higherconfining}(c) for $Q_2$, we see that $\alpha(z^{K_1}_0w)(q) \in \alpha(z^{K_1}_0)(Q^{K_1}_2) \subset Q_2.$ 

We claim that the forward implication holds for $\mu = K_1 \rho(z_0) + K_1C_\rho \in \R$. Indeed, if $\rho(z) \geq \mu$, then $\rho(zz_0^{-K_1}w^{-1})\geq 0$, 
and so $\alpha(z)(q) = \alpha(zz^{-K_1}_0w^{-1}) \alpha(z^{K_1}_0w)(q) \in \alpha(zz^{-K_1}_0w^{-1})(Q_2) \subset Q_2$.

For the reverse implication, suppose that we have a $\mu \in \R$ such that $\alpha(z)(Q_1) \subset Q_2$ for all $z \in \Z^n$ satisfying $\rho(z) \geq \mu$. Fix a $z_1$ such that $\rho(z_1) \geq \mu$. Then  $\|q\|_{Q_2 \cup Z_\rho} \leq 1 + 2\|z_1\|_{Z_\rho}$ for any $q \in Q_1$, giving a uniform bound independent of $q$. \end{proof}

\begin{prop}\label{Sorderrev} Suppose that $Q_1$ and $Q_2$ are confining subsets with respect to $\rho$ such that $Q_2 \preceq_\rho Q_1 $. Then $\Pmap(Q_1) \preceq_\rho \Pmap(Q_2)$. In particular, if $Q_2 \sim_\rho Q_1 $, then $\Pmap(Q_2)\sim_\rho \Pmap(Q_1)$. \end{prop} 

\begin{proof} Set $\widetilde{v}_i=\Pmap(Q_i)$. Choose $\mu\geq 0$ such that $\alpha(w)(Q_1)\subset Q_2$ for any $w\in \Z^n$ with $\rho(w)\geq \mu$ and such that there exists $w\in \Z^n$ with $\rho(w)=\mu$. Given $h \in H$, suppose that $z \in \Z^n$ is such that $\alpha(z)(h) \in Q_1$. If $w\in \Z^n$ with $\rho(w)= \rho(z)+\mu$ then we see that \[\alpha(w)(h)=\alpha(wz^{-1})(\alpha(z)(h))\in \alpha(wz^{-1})(Q_1)\subset Q_2.\] Thus $\inf\{\rho(w):\alpha(w)(h)\in Q_2\}\leq \rho(z)+\mu$. Taking the negative infimum over all such $z$ yields $\widetilde{v}_2(h)\geq \widetilde{v}_1(h)-\mu$, and so $\widetilde{v}_2\succeq_\rho \widetilde{v}_1$, as desired.
\end{proof} 

We now show that the map $\conf$ is also order-reversing, and consequently respects the equivalence relation on pseudo-valuations. 
\begin{lem}\label{Lorderrev} If $\pval_1, \pval_2$ are pseudo-valuations such that $\pval_1 \preceq_\rho \pval_2$, then $\conf(\pval_2) \preceq_\rho \conf(\pval_1)$. In particular, if $\pval_1 \sim_\rho \pval_2$, then $\conf(\pval_1)  \sim_\rho \conf(\pval_2) .$\end{lem} 

\begin{proof} Set $Q_i=\conf(\pval_i)$. Since $\pval_1 \preceq_\rho \pval_2$, there exists a constant $K$ such that $\pval_1(h) \leq \pval_2(h) + K$ for all $h \in H$. It then follows from the property (PV3) applied to $\pval_2$ that  $\alpha(z)(Q_1) \subset Q_2$ for any $z \in \Z^n$ such that $\rho(z) \geq K$. Thus $Q_2 \preceq_\rho Q_1$ by  Lemma \ref{altdescofconf}. \end{proof} 

It follows that our maps $\conf$ and $\Pmap$ descend to maps between the equivalence classes of confining sets and equivalence classes of pseudo-valuations subordinate to $\rho$. We will now show that these maps are inverses, which will prove that they are order-reversing isomorphisms. 
We first show the composition $\Pmap \circ \conf$ is the identity on equivalence classes of pseudo-valuations. 

\begin{prop}\label{SLid} If $\pval$ is a pseudo-valuation subordinate to $\rho$, then $\pval\sim_\rho \Pmap(\conf(\pval))$. 
\end{prop} 
\begin{proof}
Set $Q = \conf(\pval)$ and $\widetilde u = \Pmap(Q)$. 
We first show that $\pval(h) = +\infty$ if and only if $\widetilde u(h) = +\infty$. Let $\pval(h) = + \infty$. Observe that $\pval(\alpha(z)(h)) = \rho(z) + \pval(h) = +\infty$ for any $z \in \Z^n$.  It follows that $\op{inf} \{ \rho(z) : \alpha(z)(h) \in Q \} = - \infty$,  and thus $\widetilde u(h) = + \infty$. Conversely, suppose that $\widetilde u(h) = +\infty$. By the definition of $\widetilde u$, there is a sequence $\{y_i \} \subset \Z^n$ with $\rho(y_i)  \to- \infty$ such that $\alpha(y_i)(h) \in Q$ for all $i$. But then $ \pval(\alpha(y_i)(h)) = \rho(y_i) + \pval(h) \geq 0$ for all $i$. As the $\rho(y_i)$ are arbitrarily negative (and $\rho$ takes values only in $\R$), this is only possible if $\pval(h) = + \infty$.  

We now show that there is a constant $K_0$ such that $$ \widetilde u(h) \leq \pval(h) \leq \widetilde{u}(h) + K_0$$ for all $h \in H$, which will prove the result. Let $h\in H$. In light of the above, we may assume that both $\widetilde u(h), \pval(h)$ are finite values, for if $\widetilde u(h)=+\infty = \pval(h)$, they satisfy the inequality for any constant $K_0$. Suppose that $z \in \Z^n$ such that $\alpha(z)(h) \in Q$. By definition $\pval(\alpha(z)(h)) \geq 0$, and therefore $\rho(z) + \pval(h) \geq 0$ and $\rho(z) \geq - \pval(h).$ Taking an infimum on the left-hand side over all such $z$ gives us that  $$\op{inf} \{ \rho(z) : \alpha(z)(h) \in Q \} \geq - \pval(h).$$ Thus $\widetilde u(h) \leq \pval(h)$.
To establish the other inequality, note that since the image of $\rho$  in $\R$ is either dense or cyclic, there exists $K_0>0$ such that every point of $\R$ is within $K_0$ of a point of $\Im(\rho)$.
Choose  $z \in \Z^n$ such that $0 \leq \pval(h) + \rho(z) \leq K_0$. Thus $0 \leq \pval(\alpha(z)(h)) \leq K_0$ so that $\alpha(z)(h) \in Q$ and $\rho(z) \leq -\pval(h) + K_0$. It then follows that $\widetilde u(h) \geq \pval(h) - K_0$ so that $\pval(h) \leq \widetilde u(h) + K_0$. This completes the proof.\end{proof} 

Next, we prove that the composition $\conf \circ \Pmap$ is the identity on equivalence classes of confining subsets. 

\begin{prop}\label{LSid} If  $Q$ is confining with respect to $\rho$, then $Q \sim_\rho \conf(\Pmap(Q))$. 
\end{prop} 
\begin{proof} Let $\pval = \Pmap(Q)$ and $Q' = \conf(\pval)$. Since $Q$ is confining with respect to $\rho$,  Definition \ref{def:higherconfining}(a) yields that
$ q \in Q \Rightarrow    \op{inf} \{ \rho(z) : \alpha(z)(q) \in Q \} \leq 0  \Rightarrow  \pval(q) \geq 0  \Rightarrow   q \in Q'.$ Thus $Q \subset Q'$ and hence $Q' \preceq_\rho  Q$. Conversely, let $q \in Q'$. Then $\pval(q)\geq 0$, and the definition of $\pval$ yields that $\alpha(z)(q) \in Q$ for any $z \in \Z^n$ such that $\rho(z) > 0$. Thus  $Q \preceq_\rho Q'$  by Lemma \ref{altdescofconf}, completing the proof.
\end{proof}

Finally we are ready to put together the above results to prove Theorem \ref{thm:PvalandConfining}.

\begin{proof}[Proof of Theorem \ref{thm:PvalandConfining}] 
By Lemmas \ref{Sorderrev} and \ref{Lorderrev}, the maps $\conf$ and $ \Pmap$ descend to order-reversing maps on the equivalence classes of pseudo-valuations subordinate to $\rho$ and confining subsets with respect to $\rho$, respectively. It follows from Propositions \ref{SLid} and \ref{LSid} that the maps are inverses of each other on the equivalence classes. Since both are order-reversing, they descend to inverse isomorphisms between the poset of pseudo-valuations and the opposite of the poset of confining subsets.
\end{proof}

We end this section with a discussion of the structure of a confining subset associated to a pseudo-valuation on $H$ subordinate to $\rho$.  In particular, we show that, up to equivalence of generating sets, $Q$ is the union of a collection of cosets of a subgroup $K$ of $H$, defined by
$$K = \{h \in H :\pval(h) = + \infty \}.$$ The reader may check using conditions (PV2) and (PV3) that $K$ is indeed a subgroup. Since $Q= \{ h \in H : \pval(h) \geq 0\}$, it is clear that $K \subset Q$. Let $T$ be a transversal for $K$ in $H$ containing $0$.  Let $T_Q\subseteq T$ be  the set of  representatives of cosets of $K$ that have non-empty intersection with $Q$, and let $\overline{K} = \displaystyle  \bigcup_{g \in T_Q} (g+K) $.

Our goal is to prove the following. 

\begin{prop}\label{prop:StructureofQ}
Let $\pval$ be a pseudo-valuation on $H$ subordinate to $\rho$, and let $Q=\conf(\pval)$ be the associated confining set. Then we have $
[Q \cup Z_\rho] = \left[  \overline K \cup Z_\rho \right].
$
\end{prop}

If $K=0$, then the proposition is vacuous. However, it may yield interesting information if $K\neq 0$. We begin by proving some properties of the subgroup $K$.
\begin{lem} If $\pval$ is a pseudo-valuation on $H$ subordinate to $\rho$ and $Q = \conf(\pval)$, then $K$ is a subgroup of $H$ invariant under the action of $\Z^n$ and hence normal in $G$. Further, $K$ is the (unique) largest such subgroup contained in $Q$. 
\end{lem} 

\begin{proof} To see the invariance of $K$ under the action of $\Z^n$, observe that $\pval(\alpha(z)(k)) = \rho(z) + \pval(k) = + \infty$ for any $k \in K$ and $z \in \Z^n$, as $\rho$ only takes finite values in $\R$. Thus $\alpha(z)(k) \in K$ for all $k \in K, z \in \Z^n$. The normality of $K$ follows from its invariance under $\Z^n$ and the fact that $H$ is abelian.

To prove the final statement, suppose for contradiction that $K'$ is a subgroup contained in $Q$ that is invariant under the action of $\Z^n$ and such that $K' \backslash K \neq \emptyset$. Then there exists an $h \in K'$ with $0 \leq \pval(h) < +\infty$. Since $K'$ is $\Z^n$--invariant,  $\alpha(z)(h) \in K'$  for any $z \in \Z^n$. However, $\pval(\alpha(z)(h)) = \rho(z) + \pval(h)$, and so  choosing $z$ with $\rho(z)$ very negative shows there are elements in $K'\subseteq Q$ whose valuations are negative. This contradicts the definition of $Q$, and so $K'\subseteq K$. 
\end{proof} 

The subgroup $K$ can be very different in different cases. For instance, let $\phi \in SL_3(\Z)$ be a matrix with irrational eigenvalues. If $G=\Z^3 \rtimes_\phi \Z$, then $K$ has to be a $\phi-$invariant subgroup of $\Z^3$ for any confining subset $Q$. If $Q$ is a strictly confining subset,  we must have $K = \{0 \}$. On the other hand, if $Q=\Z^3$, then $K = \Z^3$. In the case of $G=\R^3 \rtimes_\phi \Z$, the subgroup $K$ will be one of the invariant subspaces for the action of $\phi$ on $\R^3$, depending on the choice of confining subset.

The next lemma shows  that every strictly confining subset  intersects one of the non-trivial cosets of $K$. 

\begin{lem} If $\displaystyle Q \cap \left( \bigcup_{g \in T\setminus\{0\}} g + K \right) = \emptyset$, then $Q = K = H$. \end{lem}
\begin{proof} We must have $Q=K$, as $K\subset Q$ and $Q$ does not intersect any non-trivial coset of $K$ by assumption. Thus $Q$ is a $\Z^n$--invariant subgroup of $H$. By Definition \ref{def:higherconfining}(b), for every $h \in H$, there is a $z \in \Z^n$ such that $\alpha(z)(h) \in Q$, and hence $h \in \alpha(z^{-1})(Q) = Q$.\end{proof} 

We are now ready to prove Proposition~\ref{prop:StructureofQ}. 

\begin{proof}[Proof of Proposition~\ref{prop:StructureofQ}] 

We first show that $\overline K \cup Z_\rho$ is a generating set for $G=H \rtimes_\alpha \Z^n$.   As $Z_\rho$ generates $\Z^n$, it suffices  to show that $H$ is contained in the subgroup generated by $\overline{K}\cup Z_\rho$.

Let $h \in H$.  Then $h \in g+K$ for some $g \in T$. By Definition \ref{def:higherconfining}(b), there is a $z \in \Z^n$ such that $\alpha(z)(h) \in Q$. Hence $\alpha(z)(h) \in Q \cap (g' +K)$ for some $g' \in T$. In particular, $g' \in T_Q$ and $h \in \alpha(z^{-1})(g' + K)$, so  $h$ is contained in the subgroup generated by $\overline{K}\cup Z_\rho$.

To establish the equivalence of the generating sets, we will show that each generating set has uniformly bounded word length with respect to the other. First of all, $Q\subset\overline{K}$ by definition of $T_Q$.

Conversely, $K \subset Q$ by definition. If $Q$ intersects the coset $g+K$ non-trivially, we will show that this coset has uniformly bounded word length in $Q \cup Z_\rho$, which will prove the result. By assumption, there exists an element $q \in Q$ of the form $q = g+k$ with $k\in K$. For any other $k' \in K$, write $g + k' = q + (k' -k)$. As $K$ is a subgroup contained in $Q$, this implies  that $g +k' \in Q+ Q$. If $z_0$ is the element from Definition \ref{def:higherconfining}(c), then  $\alpha(z_0)(g+k') \in Q$. As $k'$ was arbitrary, it follows that the word length of any element of $g+K$ with respect to $Q\cup Z_\rho$ is at most $2\|z_0\|_{Z_\rho} +1$. 
\end{proof}

\bibliographystyle{abbrv}
\bibliography{ab_by_cyclic}
\addcontentsline{toc}{section}{References}

\end{document}

%% file: treeconstruction.pdf_tex
\begingroup%
  \makeatletter%
  \providecommand\color[2][]{%
    \errmessage{(Inkscape) Color is used for the text in Inkscape, but the package 'color.sty' is not loaded}%
    \renewcommand\color[2][]{}%
  }%
  \providecommand\transparent[1]{%
    \errmessage{(Inkscape) Transparency is used (non-zero) for the text in Inkscape, but the package 'transparent.sty' is not loaded}%
    \renewcommand\transparent[1]{}%
  }%
  \providecommand\rotatebox[2]{#2}%
  \newcommand*\fsize{\dimexpr\f@size pt\relax}%
  \newcommand*\lineheight[1]{\fontsize{\fsize}{#1\fsize}\selectfont}%
  \ifx\svgwidth\undefined%
    \setlength{\unitlength}{1584.96191791bp}%
    \ifx\svgscale\undefined%
      \relax%
    \else%
      \setlength{\unitlength}{\unitlength * \real{\svgscale}}%
    \fi%
  \else%
    \setlength{\unitlength}{\svgwidth}%
  \fi%
  \global\let\svgwidth\undefined%
  \global\let\svgscale\undefined%
  \makeatother%
  \begin{picture}(1,0.35963011)%
    \lineheight{1}%
    \setlength\tabcolsep{0pt}%
    \put(0,0){\includegraphics[width=\unitlength,page=1]{treeconstruction.pdf}}%
    \put(0.00507676,0.20820691){\color[rgb]{1,0,0}\makebox(0,0)[lt]{\lineheight{1.25}\smash{\begin{tabular}[t]{l}\textit{$\{0\}\times \R$}\end{tabular}}}}%
    \put(0,0){\includegraphics[width=\unitlength,page=2]{treeconstruction.pdf}}%
    \put(0.15068435,0.20820691){\color[rgb]{0,0,1}\makebox(0,0)[lt]{\lineheight{1.25}\smash{\begin{tabular}[t]{l}\textit{$\{1\}\times \R$ }\end{tabular}}}}%
    \put(0,0){\includegraphics[width=\unitlength,page=3]{treeconstruction.pdf}}%
    \put(0.29650453,0.20820692){\color[rgb]{0,0.58823529,0}\makebox(0,0)[lt]{\lineheight{1.25}\smash{\begin{tabular}[t]{l}\textit{$\{2\}\times \R$}\end{tabular}}}}%
    \put(0.5456628,0.20820691){\color[rgb]{0,0,0}\makebox(0,0)[lt]{\lineheight{1.25}\smash{\begin{tabular}[t]{l}\textit{$\pi$}\end{tabular}}}}%
    \put(0,0){\includegraphics[width=\unitlength,page=4]{treeconstruction.pdf}}%
    \put(0.95289286,0.20820692){\color[rgb]{0,0,0}\makebox(0,0)[lt]{\lineheight{1.25}\smash{\begin{tabular}[t]{l}\textit{$T_v$}\end{tabular}}}}%
  \end{picture}%
\endgroup%